\documentclass[12pt]{article}
\usepackage[margin =1in]{geometry}
\usepackage{amssymb,amsthm}
\usepackage{amsmath}
\usepackage{enumerate}
\usepackage{enumitem}
\usepackage{hyperref}
\usepackage{cite}
\usepackage{color}

\numberwithin{equation}{section}
\usepackage{tikz}
\usepackage{pgfplots}

\newcommand{\ip}[1] {\langle #1 \rangle }
\newcommand{\norm}[1] {\left \| #1 \right \|}
\newcommand{\inclu}[0] {\ar@{^{(}->}}

\newcommand{\dist}{{\rm dist}}
\newcommand{\R}{\mathbb{R}}

\newcommand{\EE}{\mathbb{E}}
\newcommand{\PP}{\mathbb{P}}

\newcommand{\sign}{\mathrm{sign}}

\newcommand{\RR}{\mathbb{R}}
\newcommand{\pfail}{p_{\mathrm{fail}}}
\newcommand{\cX}{\mathcal{X}}

\newcommand{\cS}{\mathcal{S}}
\newcommand{\sphere}{\mathbb{S}}
\newcommand{\proj}{\mathrm{proj}}
\newcommand{\cT}{\mathcal{T}}
\newcommand{\lipsymb}{\mathsf{L}}

\newcommand{\cF}{\mathcal{F}}
\newcommand{\inneridx}{k}
\newcommand{\INNERIDX}{K}
\newcommand{\outeridx}{t}
\newcommand{\OUTERIDX}{T}
\newcommand{\quadapprox}{\eta}

\newcommand{\weakphase}{\tilde \eta}
\newcommand{\SMBM}{\mathrm{MBA}}
\newcommand{\SMBMSC}{\mathrm{PMBA}}
\newcommand{\SMBMSCE}{\mathrm{EPMBA}}
\newcommand{\RSMBMSCE}{\mathrm{RPMBA}}
\newcommand{\RSMBM}{\mathrm{RMBA}}
\newcommand{\TRIALCOUNT}{M}

\newcommand{\strongconvex}{\nu}
\newcommand{\isconv}{\texttt{is\_conv}}

\newcommand{\true}{\texttt{true}}
\newcommand{\false}{\texttt{false}}
\newcommand{\data}{z}
\newcommand{\initqual}{\delta_1}
\newcommand{\initquallem}{\delta}
\newcommand{\stepprob}{\delta_2}
\newcommand{\rank}{\mathrm{rank}}
\newcommand{\bern}{u}

\newcommand{\blindnorm}{D}

\newcommand{\allasumption}{\ref{assumption:sample}-\ref{assumption:lipschitz} }

\newcommand{\abs}[1]{\left| #1 \right|}

\newcommand{\argmin}{\operatornamewithlimits{argmin}}

\newcommand{\NN}{\mathbb{N}}

\newtheorem{thm}{Theorem}[section]
\newtheorem{definition}[thm]{Definition}
\newtheorem{proposition}[thm]{Proposition}
\newtheorem{lem}[thm]{Lemma}
\newtheorem{cor}[thm]{Corollary}

\newtheorem{assumption}{Assumption}

\newtheorem{remark}{Remark}

\newtheorem{example}{Example}[section]

\usepackage{mathtools}
\DeclarePairedDelimiter{\dotp}{\langle}{\rangle}
\usepackage[boxruled]{algorithm2e}
\usepackage{authblk}

\begin{document}
	
	\title{Stochastic algorithms with geometric step decay converge linearly on sharp functions}
	
	
\author{Damek Davis\thanks{School of ORIE, Cornell University,
		Ithaca, NY 14850, USA;
		\texttt{people.orie.cornell.edu/dsd95/}.}\qquad Dmitriy Drusvyatskiy\thanks{Department of Mathematics, U. Washington,
		Seattle, WA 98195; Microsoft Research, Redmond, WA 98052;  \texttt{www.math.washington.edu/{\raise.17ex\hbox{$\scriptstyle\sim$}}ddrusv}. Research of Drusvyatskiy was supported by the NSF DMS   1651851 and CCF 1740551 awards.} \qquad Vasileios Charisopoulos\thanks{School of ORIE, Cornell University, Ithaca, NY 14850,
		USA; \texttt{people.orie.cornell.edu/vc333/}}}

	\date{}
	\maketitle

\begin{abstract}
Stochastic (sub)gradient methods require step size schedule tuning to perform well in practice. Classical tuning strategies decay the step size polynomially and lead to optimal sublinear rates on (strongly) convex  problems.  An alternative schedule, popular in nonconvex optimization, is called \emph{geometric step decay} and proceeds by halving the step size after every few epochs. In recent work, geometric step decay was shown to improve exponentially upon classical sublinear rates for the class of \emph{sharp} convex functions. In this work, we ask whether geometric step decay similarly improves stochastic algorithms for the class of sharp nonconvex problems. Such losses feature in modern statistical recovery problems and lead to a new challenge not present in the convex setting: the region of convergence is local, so one must bound the probability of escape. Our main result shows that for a large class of stochastic, sharp, nonsmooth, and nonconvex problems a geometric step decay schedule endows well-known algorithms with a local linear rate of convergence to global minimizers. This guarantee applies to the stochastic projected subgradient, proximal point, and prox-linear algorithms. As an application of our main result, we analyze two statistical recovery tasks---phase retrieval and blind deconvolution---and match the best known guarantees under Gaussian measurement models and establish new guarantees under heavy-tailed distributions.
\end{abstract}

\section{Introduction} 
Stochastic (sub)gradient methods form the algorithmic core of much of modern statistical and machine learning. Such algorithms are typically sensitive to algorithm parameters, and require extensive step size tuning to achieve adequate performance. Classical tuning strategies decay the step size polynomially and lead to optimal sublinear  rates of convergence on convex and strongly convex problems
\begin{equation}\label{eqn:stoch_approx_1}
\displaystyle\min_{x\in\cX}~ f(x)=\EE_{\data} [f(x,\data)], \tag{$\mathcal{SO}$}
\end{equation}  
where the loss functions $f(\cdot,z)$ are convex and 
$\cX$ is a closed convex set~\cite{complexity}. An alternative schedule, called \emph{geometric step decay}, decreases the step size geometrically by halving it after every few epochs. In recent work \cite{xu2016accelerated}, geometric step decay was shown to improve exponentially upon classical sublinear rates under the sharpness assumption:
\begin{align}\label{eq:sharp}
f(x)- \min_{\cX} f\geq \mu\cdot \dist(x,\cX^\ast) \qquad \forall x\in \cX,
\end{align}
where $\mu>0$ is some constant and $\cX^\ast =\argmin_{x\in \cX} f(x)$ is the solution set.  This result complements early works of Goffin~\cite{goffin} and Shor~\cite{shor2012minimization}, which show that deterministic subgradient methods converge linearly on sharp convex functions if their step sizes decay geometrically.  The work \cite{xu2016accelerated} also reveals a departure from the smooth strongly convex setting, where deterministic linear rates degrade to sublinear rates when the gradient is corrupted by noise~\cite{complexity}.\footnote{There are notable exceptions for highly structured problems, such as finite sums~\cite{SAG} or interpolation problems where all terms in the sum share the same minimizer~\cite{schmidt2013fast}.}

Beyond the convex setting, sharp problems appear often in nonconvex statistical recovery problems, for example, in robust matrix sensing \cite{li2018nonconvex}, phase retrieval \cite{eM,duchi_ruan_PR}, blind deconvolution \cite{charisopoulos2019composite}, and quadratic/bilinear sensing and matrix completion \cite{charisopoulos2019low}. For such problems, sharpness is surprisingly common and corresponds to strong identifiability of the statistical model. Despite this,  we do not know whether stochastic algorithms equipped with geometric step decay---or any other step size schedule---linearly converge on sharp nonconvex problems.\footnote{There are some exceptions for highly-structured problems, such as interpolation problems where all terms in the expectation share the same minimizer~\cite{asi2018stochastic,tan2018phase}.} Such algorithms, if available, could pave the way for new sample efficient strategies for these and other statistical recovery problems.

The main result of this work shows that for a large class of stochastic, sharp, nonsmooth, and nonconvex problems
\begin{quote}
\centering a geometric step decay schedule endows well-known algorithms \\  with a local linear rate of convergence to global minimizers.
\end{quote}
This guarantee applies, for example, to the stochastic projected subgradient, proximal point, and prox-linear algorithms.  Beyond sharp growth~\eqref{eq:sharp}, we also analyze losses that grow sharply away from some closed set $\cS$, which is strictly larger than $\cX^\ast$. Such sets $\cS$ are akin to ``active manifolds" in the sense of Lewis~\cite{lewis_active} and Wright \cite{MR1227547}. For example, the loss $f(x,y)=x^2+|y|$ is not sharp relative to its minimizer, but is sharp relative to 
the $x$-axis.  For these problems, our algorithms converge linearly to the $\cS$. 
 Finally, we illustrate the result with two statistical recovery problems: phase retrieval and blind deconvolution. For these recovery tasks, our results match the best known computational and sample complexity guarantees under Gaussian measurement models and establish new guarantees under heavy-tailed distributions. 

\subsection*{Related work}
Our paper is closely related to a number of  influential techniques in stochastic, convex, and nonlinear optimization. We now survey these related topics.

{\bf Stochastic model-based methods.} In this work, we use algorithms that iteratively sample and minimize simple stochastic convex models of the loss function. Throughout, we call these methods \emph{model-based algorithms.} Such algorithms include the stochastic projected subgradient, prox-linear, and proximal point methods.
Stochastic model based algorithms are known to converge globally to stationary points at a sublinear rate on a large class of nonsmooth and nonconvex problems~\cite{davis2019stochastic,duc_ruan_stoch_comp}. Some model-based algorithms also possess superior stability properties and can be less sensitive to step size choice than traditional stochastic subgradient methods~\cite{asi2019importance,asi2018stochastic}.

{\bf Geometrically decaying learning rate in deterministic optimization.}
polynomially decaying step-sizes  are common in stochastic optimization  \cite{MR1993642,MR0042668,MR1167814}. In contrast, we develop algorithms with step sizes that decay geometrically.  Geometrically decaying step sizes were first analyzed in convex optimization by Shor \cite[Thm 2.7, Sec. 2.3]{shor2012minimization} and Goffin \cite{goffin}. This step size schedule is also closely related to the step size rules of Eremin \cite{MR0175627} and Polyak \cite{MR0250452}. 
Similar schedules are known to accelerate convex subgradient methods under H\"{o}lder growth as shown in \cite{Johnstone2019,yang2018rsg}.
Geometrically decaying step sizes for deterministic nonconvex subgradient methods were systematically studied in \cite{MR3869491}.

{\bf Geometric step decay in stochastic optimization.}
The geometric step decay schedule is common in practice: for example, see Krizhevsky et. al. \cite{krizhevsky2012imagenet} and He et al. \cite{he2016deep}. It is a standard option in the popular deep learning libraries, such as Pytorch \cite{paszke2017automatic} and TensorFlow \cite{abadi2015tensorflow}. Geometric step decay has been analyzed in a number of recent papers in stochastic convex optimization, including  \cite{ge2019step,kulunchakov2019generic, aybat2019universally,yang2018does,ghadimi2013optimal,xu2016accelerated}.
Among these papers, the work \cite{xu2016accelerated} relates most to ours. There, the authors propose two geometric step decay strategies that converge linearly on convex functions that are sharp and have bounded stochastic subgradients. These algorithms either use a moving ball constraint or follow a proximal-point type procedure. We follow the latter strategy, too, at least to obtain high-probability guarantees. The paper \cite{xu2016accelerated} differs from our work in that they assume convexity and a uniform bound on the stochastic subgradients. In contrast, we do not assume convexity and only assume that stochastic subgradients have a finite second moment.   We are aware of only one subgradient method for stochastic nonconvex problems that converges linearly \cite{asi2018stochastic} under favorable assumptions. In \cite{asi2018stochastic}, the authors develop a ``clipped" subgradient method, which resembles a safeguarded stochastic Polyak step. In contrast to our work, their algorithms converge linearly only under ``perfect interpolation," meaning that all terms in the expectation share a minimizer. We do not make this assume here.

{\bf Restarts in deterministic optimization.}
Restart techniques have a long history in nonlinear  programming, such as for
conjugate gradient and limited memory quasi-Newton methods. They have also been
used more recently to improve the complexity of algorithms in deterministic
convex optimization. For example, restart schemes can accelerate sublinear
convergence rates for convex problems that satisfy growth conditions as shown
by Nesterov \cite{nest_orig} and Nemirovskii and Nesterov
\cite{NEMIROVSKII198521}, Renegar \cite{renegar2018simple}, O'Donoghue and
Candes \cite{MR3348171}, Roulet and d'Aspremont \cite{NIPS2017_6712}, Freund and Lu~\cite{Freund2018}, and Fercoq and Qu \cite{fercoq2016restarting,fercoq2017adaptive} and others.
In the nonconvex setting, stochastic restart methods are challenging to analyze, since the region of linear convergence is local. To overcome this challenge, one must bound the probability that the iterates leave this region. One of our main technical contributions is a technique for bounding this probability.

{\bf Finite sums.}
For finite sums, stochastic algorithms that converge linearly are more common.
For example, for finite sums that are sharp and convex, Bertsekas and Nedi{\'c}~\cite{nedic2001convergence} prove that an incremental Polyak-type algorithm converges linearly. For finite sums that are smooth and strongly convex, variance reduced methods, such as SAG \cite{SAG}, SAGA \cite{SAGA2}, SDCA \cite{sdca},  SVRG \cite{svrg}, MISO/Finito \cite{pmlr-v32-defazio14,MR3335503}, SMART \cite{davis2016smart} and their proximal extensions converge linearly. 
The algorithms we develop here do not assume a finite sum structure.

{\bf Verifying sharpness.}
Sharp growth is a central assumption in this work. This property is surprisingly common in statistical recovery problems. For example, sharp growth has been established for robust matrix sensing \cite{li2018nonconvex}, phase retrieval \cite{eM,duchi_ruan_PR}, blind deconvolution \cite{charisopoulos2019composite}, quadratic and bilinear sensing and matrix completion \cite{charisopoulos2019low} problems. Consequently, the results of this paper apply in these settings.

\subsection*{Notation}
We will mostly follow standard notation used in convex analysis and stochastic optimization. Throughout, the symbol $\R^d$ will denote a $d$-dimensional Euclidean space with the inner product $\langle \cdot,\cdot\rangle$ and the induced norm $\|x\|=\sqrt{\langle x,x\rangle}$. We denote the open ball of radius $\varepsilon>0$ around a point $x\in \R^d$ by the symbol $B_{\varepsilon}(x)$.
For any set $Q\subset\R^d$, the  {\em distance function} and the {\em projection map} are defined by
\begin{equation*}
\dist(x,Q):=\inf_{y\in Q} \|y-x\|\qquad \textrm{and}\qquad
\proj_Q(x):=\argmin_{y\in Q} \|y-x\|,
\end{equation*}
respectively.  Consider a function $f\colon\R^d\to\R\cup\{\pm\infty\}$ and a point $x$, with $f(x)$ finite. The {\em subdifferential} of $f$ at $x$, denoted by $\partial f(x)$, consists of all vectors $v\in\R^d$ satisfying
$$f(y)\geq f(x)+\langle v,y-x\rangle+o(\|y-x\|)\qquad \textrm{as }y\to x.$$
A function $f$ is called {\em $\rho$-weakly convex} on an open convex set $U$ if the perturbed function $f+\frac{\rho}{2}\|\cdot\|^2$ is convex on $U$. The subgradients of such functions automatically satisfy the uniform approximation property:
$$f(y)\geq f(x)+\langle v,y-x\rangle-\frac{\rho}{2}\|y-x\|^2\qquad \textrm{for all }x,y\in U, v\in \partial f(x).$$

\section{Algorithms, assumptions, and main results}

In this section, we formalize our target problem and introduce algorithms to solve it. We then outline our main results. The complete theorem statements and proofs appear in Section~\ref{sec:formal}.
Throughout, we consider the minimization problem
\begin{align}\label{eqn:stoch_approx}
\min_{x \in \cX}~ f (x).
\end{align}
for some function $f\colon\R^d\to\R$ and a
closed convex set $\cX \subseteq \RR^d$. We define the set  $\cX^*:=\argmin_{x\in \cX} f(x)$ and assume it to be nonempty. We also fix a probability space $(\Omega,\mathcal{F},\PP)$ and equip $\R^d$ with the Borel $\sigma$-algebra and make the following assumption.
\begin{enumerate}[label=$\mathrm{(A\arabic*)}$]
	\item \label{assumption:sample} {\bf (Sampling)} It is possible to generate i.i.d.\ realizations $\data_1,\data_2, \ldots \sim \PP$.
\end{enumerate} 

The algorithms we develop rely on a stochastic oracle model for Problem~\eqref{eqn:stoch_approx_1} that was recently introduced in~\cite{davis2019stochastic}. These algorithms assume access to a family of functions ${f_x(\cdot,\data)}$---called stochastic models---indexed by basepoints $x\in\R^d$ and random elements $\data\sim \PP$. Given these models, the generic \emph{stochastic model-based algorithm} of~\cite{davis2019stochastic} iterates the steps:
\begin{equation}\label{eq:MBA}
\begin{aligned}
&\textrm{Sample } \data_\inneridx \sim \PP\\
& \textrm{Set } y_{\inneridx+1} = \argmin_{y\in\cX}~ \left\{f_{y_\inneridx}(y,\data_\inneridx) + \frac{1}{2\alpha} \|y - y_\inneridx\|^2\right\}
\end{aligned}
\end{equation}
In this work, model-based algorithms form the core of the following restart strategy: given inner and outer loop sizes $\INNERIDX$ and $\OUTERIDX$, respectively, as well as initial stepsize $\alpha_0 > 0 $, perform:
\begin{equation}\label{eq:RMBA}
\begin{aligned}
&\textrm{For $\outeridx = 0, \ldots, \OUTERIDX-1$:} \\
&\hspace{20pt} \textrm{Initialize~\eqref{eq:MBA} at $y_0 = x_\outeridx$, set $\alpha = 2^{-\outeridx}\alpha_0$, and run  $K$ iterations;}\\
&\hspace{20pt} \textrm{Sample $x_{\outeridx + 1}$ uniformly from $y_0, \ldots, y_\INNERIDX$.}
\end{aligned}
\end{equation} 
This restart strategy is common in machine learning practice and is called \emph{geometric step decay.} Restart schemes date back to the fundamental work of Nesterov \cite{nest_orig} and Nemirovskii and Nesterov \cite{NEMIROVSKII198521} and more recently appear in~\cite{renegar2018simple,MR3348171,NIPS2017_6712,fercoq2016restarting,fercoq2017adaptive,xu2016accelerated,ge2019step}. These strategies often improve the convergence guarantees of the algorithm they restart under growth assumptions, for example, by boosting an algorithm that converges sublinearly to one that converges linearly. In this work, we will show that restart scheme~\eqref{eq:RMBA}  similarly improves~\eqref{eq:MBA} for a large class of nonconvex stochastic optimization problems. 

\subsection{Assumptions} 

In this section, we formalize our assumptions on sharp growth of~\eqref{eqn:stoch_approx} as well as on accuracy, regularity, and Lipschitz continuity of the models.

\subsubsection*{Sharp Growth}

We assume that $f(\cdot)$ grows sharply as $x$ moves in the direction normal to a closed set $\cS$.
\begin{enumerate}[label=$\mathrm{(A\arabic*)}$]
\setcounter{enumi}{1}
	\item\label{assumption:sharp} {\bf (Sharpness)} There exists a constant $\mu > 0$ and a closed set $\cS \subseteq \cX$ satisfying $ \cX^\ast \subseteq \cS  $ such that the following bound holds: 
	\begin{align}\label{eq:sharp_S}
	f(x) - \inf_{y \in \proj_{\cS}(x)}f(y) \geq \mu \cdot \dist(x, \cS)  \qquad \forall x \in \cX.
	\end{align}
\end{enumerate}
This property generalizes the classical sharp growth condition~\eqref{eq:sharp}, where $\cS=\cX^*$. The setting $\cS = \cX^\ast$ is well-studied in nonlinear programming and often underlies rapid convergence guarantees for deterministic local search algorithms. Beyond the classical setting, $\cS$ could be a sublevel set of $f$ or an ``active manifold''  in the sense of Lewis~\cite{lewis_active}; see Section~\ref{sec:sharpness_identif}. When $\cS \neq \cX^\ast$, we design stochastic algorithms that do not necessarily converge to global minimizers, but instead linearly converge to $\cS$.

\subsubsection*{Accuracy and Regularity} We assume that the models are convex and under-approximate $f$ up to quadratic error.
\begin{enumerate}[label=$\mathrm{(A\arabic*)}$]
\setcounter{enumi}{2}
	\item \label{assumption:oneside} {\bf (One-sided accuracy)} There exists $\quadapprox > 0$ and an open convex set $U$ containing $\cX$ and a measurable function $(x,y,\data)\mapsto f_x(y,\data)$, defined on $U\times U\times\Omega$, satisfying  $$\EE_{\data}\left[f_x(x,\data)\right]=f(x)
	\qquad \forall x\in U,$$ and 
	$$\EE_{\data}\left[f_x(y,\data)-f(y)\right]
	\leq\frac{\quadapprox}{2}\|y-x\|^2\qquad \forall x,y\in U.$$
	\item\label{assumption:convex} {\bf (Convex models)} The function $f_x(\cdot,\data)$ is convex $\forall x\in U$ and a.e. $\data\in \Omega$.
\end{enumerate}


Models satisfying \ref{assumption:oneside} and \ref{assumption:convex}, and their algorithmic implications, were analyzed in~\cite[Assumption B]{davis2019stochastic}; a closely related family of models was investigated in \cite{asi2019importance,asi2018stochastic}. Assumptions~\ref{assumption:oneside} and~\ref{assumption:convex} imply that $f$ is $\quadapprox$-weakly convex on $U$, meaning that the assignment $x \mapsto f(x) + \frac{\quadapprox}{2}\|x\|^2$ is convex \cite[Lemma 4.1]{davis2019stochastic}. While we assume~\ref{assumption:oneside} throughout the paper, we show in Remark~\ref{remark:beyondweak} that our results hold under an even weaker assumption. For certain losses, models that satisfy~\ref{assumption:oneside} and~\ref{assumption:convex} are easy to construct, as the following example shows. 

\begin{example}[Convex Composite Class]\label{example:composite}
{\rm  Stochastic convex composite losses take the form 
$$f(x)=\EE_{\data} f(x,\data)\quad \textrm{with}\quad f(x,\data)=h(c(x,\data),\data),$$
where $h(\cdot, \data)$ are Lipschitz and convex and the nonlinear maps $c(\cdot,\data)$ are $C^1$-smooth with Lipschitz Jacobian. Such losses appear often in data science and signal processing (see \cite{duc_ruan_stoch_comp,charisopoulos2019low,proxpoint_rev_dima} and references therein). For this problem class, natural models  include
\begin{itemize}
	\item[] {\bf (subgradient)}\quad \hspace{15pt} $f_x(y,\data)=f(x,\data)+\langle \nabla c(x, \data)^\top v,y-x\rangle $ for any $v\in \partial  h(c(x, \data), \data)$.
	\item[] {\bf (prox-linear)}\quad \hspace{20pt} $f_x(y,\data)=h(c(x,\data)+\nabla c(x,\data)(y-x),\data).$
	\item[] {\bf (proximal point)}\quad $f_x(y,\data)=f(y,\data) +  \frac{\quadapprox}{2} \|x - y \|^2$,
\end{itemize}
where $\quadapprox$ is large enough to guarantee the proximal model is convex.\footnote{One could choose the product of the Lipschitz constants of $h$ and $\nabla c$.} If a lower bound $\ell(\data)$ on $\inf f(\cdot,\data)$ is known, one can also choose a clipped model 
\begin{itemize}
	\item[] {\bf (clipped)}\quad \hspace{40pt} $\tilde f_x(y, \data) = \max\{  f_x(y,\data), \ell(\data)\} $, 
\end{itemize} for any of the models $ f_x(\cdot,\data)$ above, as was suggested by  \cite{asi2018stochastic}. Intuitively, models that better approximate $f(x, z)$ are likely to perform better in practice; see~\cite{asi2018stochastic} for theoretical evidence supporting this claim. \qed }
\end{example}

\begin{figure}[h!]
		\begin{center}
			\begin{tikzpicture}[scale=0.6]
			\pgfplotsset{every tick label/.append style={font=\large}}
			\begin{axis}[%
			domain = -2:2,
			samples = 200,
			axis x line = center,
			axis y line = center,
			xtick={0.5},
			ytick=\empty
			]		
			\addplot[black, ultra thick] {abs(x^2-1)} [yshift=3pt] node[pos=.95,left] {\huge{$f$}};
			{         \addplot[purple, very thick] {abs(0.75-(x-0.5))} [yshift=6pt] node[pos=.95,left] {\huge{$f_x$}};      }
			
			\addplot [only marks,mark=*] coordinates { (0.5,0.75) };
			\end{axis}		
			\end{tikzpicture}
			\qquad\qquad\qquad
			\begin{tikzpicture}[scale=0.6]
			\pgfplotsset{every tick label/.append style={font=\large}}
			\begin{axis}[%
			domain = -2:2,
			samples = 200,
			axis x line = center,
			axis y line = center,
			xtick={0.5},
			ytick=\empty
			]		
			\addplot[black, ultra thick] {abs(x^2-1)+(x-0.5)^2} [yshift=3pt] node[pos=.95,left] {\large{$f+(x-0.5)^2$}};      
			{         \addplot[purple, very thick] {abs(0.75-(x-0.5))} [yshift=10pt] node[pos=.95,left] {\huge{$f_x$}};      }
			
			\addplot [only marks,mark=*] coordinates { (0.5,0.75) };
			\end{axis}		
			\end{tikzpicture}
		\end{center}

	\vspace{-20pt}
	\caption{Illustration of a one-sided model: $f(x)=|x^2-1|$, $f_{0.5}(y)=|1.25-y|$}
	\label{fig:illustr_lower_model}
\end{figure}
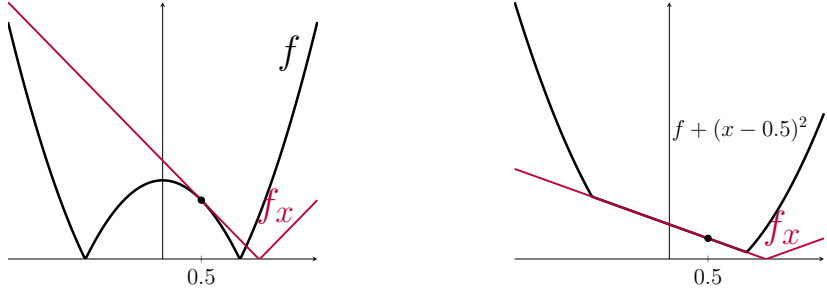

\subsubsection*{Lipschitz continuity} 

We assume that the models are Lipschitz on a tube surrounding $\cS$.
\begin{enumerate}[label=$\mathrm{(A\arabic*)}$]
\setcounter{enumi}{4}
	\item\label{assumption:lipschitz} {\bf (Lipschitz property)} Define the tube
$$
\cT_{\gamma} := \left\{ x \in \cX\mid \dist(x, \cS) \leq \frac{\gamma \mu}{\quadapprox }\right\} \qquad \forall \gamma > 0.
$$
We assume that there exists a measurable function $L \colon \RR^d \times \Omega\to \R_+$  
 such that 
	\begin{equation}\label{eqn:lip_mod}
	\min_{v\in \partial f_x(x,z)}\|v\|\leq L(x,z)
	\end{equation}
	for all $x\in\cT_2$ and a.e. $\data\in \Omega$.
Moreover, we assume there exists $\lipsymb > 0$ such that $$\sup_{x \in \cT_2} \sqrt{\EE_{\data}\left[L(x, \data)^2\right]} \leq \lipsymb.$$
\end{enumerate}
This Lipschitz property is local and differs from the global assumption of~\cite[Assumption B4]{davis2019stochastic}. The property holds only in $\cT_2$ since our algorithms will be initialized in this tube and will never leave it.\footnote{The set $\cT_2$ is a natural region to initialize in since it is provably the widest tube around the solution set containing no extraneous stationary points; see \cite[Lemma 2.1]{MR3869491} and the discussion following  \cite[Lemma 3.1]{charisopoulos2019composite}.} The local nature of this property is crucial to signal recovery applications, for example, blind deconvolution and phase retrieval. In these problems, global Lipschitz continuity does not hold; see Section~\ref{sec:examples}.

\subsection{Algorithms and results}\label{sec:algorithmandresults}

Stochastic model-based algorithms (Algorithm~\ref{alg:stoc_prox}) iteratively sample and minimize stochastic convex models of the loss function. When equipped with models satisfying~\ref{assumption:oneside}, \ref{assumption:convex} and a global Lipschitz condition, these algorithms converge to stationary points of~\eqref{eqn:stoch_approx} at a sublinear rate~\cite{davis2019stochastic,duc_ruan_stoch_comp}. In this section, we show that such sublinear rates can be improved to local linear rates by using a simple restart strategy. We introduce two such strategies that succeed with probability $1-\delta$, for any $\delta > 0$. The first (Algorithm~\ref{alg:stoc_prox_outer}) allows for arbitrary sets $\cS$ in Assumption~\ref{assumption:sharp}, but its sample complexity and initialization region scale poorly in $\delta$. The second (Algorithm~\ref{alg:stoc_prox_outer_sc}) assumes $\cS = \cX^\ast$, but has much better dependence on $\delta$. %

\begin{algorithm}[h!]
	{\bf Input:} $y_0\in \R^d$, $\alpha\geq 0$, iteration count $\INNERIDX$,  flag $\isconv \in \{\true, \false\}$ \\
	{\bf Step } $k=0,\ldots,K$:\\		
	\begin{equation*}\left\{
	\begin{aligned}
	&\textrm{Sample } \data_\inneridx \sim \PP\\
	& \textrm{Set } y_{\inneridx+1} = \argmin_{y \in \cX}~ \left\{f_{y_\inneridx}(y,\data_\inneridx) + \frac{1}{2\alpha} \|y - y_\inneridx\|^2\right\}
	\end{aligned}\right\}.
	\end{equation*}
	{\bf If } $\isconv$\\
	\hspace{20pt} {\bf Return } $\frac{1}{K+1} \sum_{\inneridx = 1}^{\INNERIDX+1}  y_{\inneridx}$;\\
	{\bf Else } \\

	\hspace{20pt} {\bf Return} $y_{\INNERIDX^*}$, where $\INNERIDX^*\in \{0,\ldots,\INNERIDX\}$ is selected uniformly at random;		
	\caption{$\SMBM(y_0, \alpha, \INNERIDX, \isconv)$
	}
	\label{alg:stoc_prox}
\end{algorithm}

\smallskip

\begin{algorithm}[H]
	{\bf Input:} $x_0\in \R^d$,  $\alpha_0\geq 0$, counters  $\INNERIDX\in \NN,T\in \NN$, flag $\isconv \in \{\true, \false\}$\\
	{\bf Step } $\outeridx=0,\ldots,\OUTERIDX-1$:
\begin{align*}
	\alpha_t&:=2^{-t}\alpha_0\\
	x_{\outeridx+1} &:= \SMBM(x_\outeridx, \alpha_t,\INNERIDX, \isconv).	
	\end{align*}
	{\bf Return} $x_{\OUTERIDX}$		
	\caption{$\RSMBM(x_0, \alpha_0, \INNERIDX, \OUTERIDX, \isconv)$
	}
	\label{alg:stoc_prox_outer}
\end{algorithm}
\smallskip

Given assumptions~\ref{assumption:sample}-\ref{assumption:lipschitz}, the following theorem shows that the first restart strategy (Algorithm~\ref{alg:stoc_prox_outer}) converges linearly to $\cS$. 

\begin{thm}[Informal]\label{thm:first_informal}
Fix a target accuracy $\varepsilon>0$, failure probability $\delta\in (0,\frac{1}{4})$, and a point  $x_0\in \cT_{\gamma\sqrt{\delta}}$ for some $\gamma\in (0,1)$. 
Then with appropriate parameter settings, the point $x=\RSMBM(x_0, \alpha_0, \INNERIDX, \OUTERIDX, \false)$ will satisfy $\dist(x,\cS)<\varepsilon$ with probability at least $1-4\delta$. Moreover, the number of samples $z_i\sim\PP$ generated by the algorithm is at most
\begin{equation*} 
 \mathcal{O}\left(\left( \frac{ \lipsymb}{\delta \mu}\right)^2 \log^3 \left(\frac{\gamma\mu/\quadapprox}{\varepsilon}\right)\right).
\end{equation*}
\end{thm}

Theorem~\ref{thm:first_informal} has interesting consequences not only for convergence to global minimizers, but also for ``active manifold identification." For example, when $\cS = \cX^*$, Theorem~\ref{thm:first_informal} shows that with constant probability, Algorithm~\ref{alg:stoc_prox_outer} converges linearly to the true solution set. When $\cS \neq \cX^\ast$ and is instead an  ``active manifold'' in the sense of Lewis~\cite{lewis_active}, Algorithm~\ref{alg:stoc_prox_outer} linearly converges to the active manifold. In our numerical evaluation, we illustrate this phenomenon for a sparse logistic regression problem. We empirically observe that the method converges linearly to the support of the solution, even though the overall convergence to the true solution may be sublinear. 

In our numerical experiments, we find that Algorithm~\ref{alg:stoc_prox_outer} almost always succeeds. Theorem~\ref{thm:first_informal}, on the other hand, only guarantees Algorithm~\ref{alg:stoc_prox_outer} succeeds with high probability when we greatly increase its sample complexity and initialize it close to $\cS$. We would like to boost Algorithm~\ref{alg:stoc_prox_outer} into a new algorithm whose sample complexity and initialization requirements scale only polylogarithmicaly in $1/\delta$. As a first attempt, we discuss the following two probabilistic techniques, both of which have limitations: 
\begin{itemize}
\item[] {\bf (Markov)} One approach is to call Algorithm~\ref{alg:stoc_prox_outer} multiple times for a moderately small value $\delta$ and pick out the ``best" iterate from the batch.
 This approach is flawed since, even in the convex setting, there is no procedure to test which iterate is ``best" without increasing sample complexity. 
 \item[] {\bf (Ensemble)} An alternative approach is based on a well-known resampling trick, which applies when $\cS=\{\bar x\}$ is a singleton set \cite[p. 243]{complexity}, \cite{hsu2016loss},\cite[Algorithm 1]{tan2018phase}: Run $m$ trials of Algorithm~\ref{alg:stoc_prox_outer}  with any fixed $\delta<1/4$, and denote the returned points by $\{ x_i\}_{i=1}^m$. Then with high probability, the majority of the points $\{ x_i\}_{i=1}^m$ will be  close to $\bar x$. Finally, to find an estimate near $\bar x$, choose any point that has at least $m/2$ other points close to it. 
\end{itemize}

The ensemble technique is promising, but it requires $\cS$ to be a singleton. This limits its applicability since many low-rank recovery problems (e.g. blind deconvolution, matrix completion, robust PCA) have uncountably many solutions. We overcome this issue by embedding Algorithm~\ref{alg:stoc_prox_outer} and the ensemble method within a proximal-point method. At each stage of this algorithm, we run multiple copies of a stochastic-model based method on a quadratically regularized problem that has a unique solution. Among those copies, we use the ensemble technique to pick out a ``successful" iterate. We summarize the resulting nested procedure in 
Algorithms~\ref{alg:stoc_prox_sc}-\ref{alg:stoc_prox_outer_sc}: Algorithm~\ref{alg:stoc_prox_sc} is a generic model-based algorithm applied on a quadratically regularized problem; Algorithm~\ref{alg:stoc_prox_high_prob} calls Algorithm~\ref{alg:stoc_prox_sc} as suggested by the ensemble technique; finally Algorithm~\ref{alg:stoc_prox_outer_sc} updates the regularization term, in the style of a proximal point method.

\smallskip
\begin{algorithm}[H]
	{\bf Input:} $y_0\in \R^d$, proximal parameter $\rho > \quadapprox$, scalar $\alpha > 0$, and iteration count $\INNERIDX$\\
	{\bf Step } $k=0,\ldots,K$:\\		
	\begin{equation*}\left\{
	\begin{aligned}
	&\textrm{Sample } \data_\inneridx \sim P\\
	& \textrm{Set } y_{\inneridx+1} = \argmin_{y \in \cX}~ \left\{f_{y_\inneridx}(y,\data_\inneridx) + \frac{1}{2\alpha} \|y - y_\inneridx\|^2 + \frac{\rho}{2}\|y - y_0\|^2\right\}
	\end{aligned}\right\},
	\end{equation*}
	Sample $\INNERIDX^*\in \{0,\ldots,\INNERIDX\}$ uniformly at random.
	
	{\bf Return} $y_{\INNERIDX^*}$ 		
	\caption{$\SMBMSC(y_0, \rho, \alpha, \INNERIDX)$
	}
	\label{alg:stoc_prox_sc}
\end{algorithm}

\smallskip
\begin{algorithm}[H]
	{\bf Input:} $y_0\in \R^d$, proximal parameter $\rho > \quadapprox$, scalar $\sigma > 0 $, iteration count $\INNERIDX$, trial count $m$, relative error tolerance $\epsilon$ \\
	{\bf Step } $j=1,\ldots,m$:\\
	\hspace{20pt} Set $y_j = \SMBMSC(y_0, \rho, \alpha, \INNERIDX)$.\\
	{\bf Step } $j=1,\ldots,m$:\\
	\hspace{20pt} \textbf{if } $|B_{2\epsilon}(y_j) \cap \{y_i\}_{i=1}^m| > \frac{m}{2}$\\
	\hspace{40pt} {\bf Return} $y_j$		\\
	\caption{$\SMBMSCE(y_0, \rho, \alpha , \INNERIDX, m, \epsilon)$
	}
	\label{alg:stoc_prox_high_prob}
\end{algorithm}
\smallskip

\smallskip
\begin{algorithm}[H]
	{\bf Input:} $x_0\in \R^d$, proximal parameter $\rho_0  > \quadapprox$, initial accuracy $\epsilon_0 > 0$, stepsize $\alpha_0 > 0$, counts $\INNERIDX,\TRIALCOUNT,\OUTERIDX\in \NN$\\
	{\bf Step } $\outeridx=0,\ldots,\OUTERIDX-1$:\\
	\begin{align*}
	\rho_t &:= 2^\outeridx \rho_0 \\
	\epsilon_t &:= 2^{-t}\epsilon_0 \\
	\alpha_t &:= 2^{-\outeridx}\alpha_0\\
	x_{\outeridx+1} &:= \SMBMSCE(x_\outeridx, \rho_\outeridx, \alpha_t, \INNERIDX, \TRIALCOUNT, \epsilon_\outeridx).	
	\end{align*}
	{\bf Return} $x_{\OUTERIDX + 1}$		
	\caption{$\RSMBMSCE(x_0,\rho_0, \alpha_0, \INNERIDX,\epsilon_0 , \TRIALCOUNT, \OUTERIDX)$
	}
	\label{alg:stoc_prox_outer_sc}
\end{algorithm}
\smallskip

We will establish the following guarantee. In the theorem, we assume $\cS=\cX^*$.

\begin{thm}[Informal]\label{thm:first_informal_high}
Fix a target accuracy $\varepsilon>0$, failure probability $\delta\in (0,1)$, and a point  $x_0\in \cT_{\gamma}$ for some $\gamma\in (0,\frac{1}{4})$. 
	Then with appropriate parameter settings, the point $x=\RSMBMSCE(x_0, \rho_0, \alpha_0, \INNERIDX,\varepsilon_0 , \TRIALCOUNT, \OUTERIDX)$ will satisfy $\dist(x,\cX^*)<\varepsilon$ with probability at least $1-\delta$. Moreover the total number of samples $z_i\sim\PP$ generated by the algorithm is at most
	\begin{equation*} 
	\mathcal{O}\left(\left( \frac{ \lipsymb}{ \mu}\right)^2 \log \left(\tfrac{\gamma\mu/\quadapprox}{\varepsilon}\right)\cdot \log\left(\frac{\log(\frac{\gamma\mu/\quadapprox}{\varepsilon})}{\delta}\right)\right).
	\end{equation*}
\end{thm}

Theorem~\ref{thm:first_informal_high} resolves the initialization and sample complexity issues of Theorem~\ref{thm:first_informal}. Incidentally, its claimed sample complexity also depends more favorably on $\varepsilon$ and on the problem parameters $\mu$ and $\quadapprox$. Theorem~\ref{thm:first_informal_high} is new in the nonconvex setting and also improves on prior work by Xu et al.~\cite{xu2016accelerated} for convex problems. There, the results require stochastic subgradients to be almost surely bounded, hence, sub-Gaussian. In contrast, Theorem~\ref{thm:first_informal_high} guarantees that $\dist(x, \cX^\ast) <\varepsilon$ with high-probability assuming only the local second moment bound~\ref{assumption:lipschitz}.


\section{Proofs of main results}\label{sec:formal}
In this section, we establish linear convergence guarantees for Algorithms~\ref{alg:stoc_prox_outer} and \ref{alg:stoc_prox_outer_sc} ---the main contributions of this work. Throughout this section, we assume that Assumptions~\ref{assumption:sample}-\ref{assumption:lipschitz} hold.

\subsection{Warm-up: convex setting}\label{sec:convex_setting}

We begin with a short proof of linear convergence for Algorithm~\ref{alg:stoc_prox_outer} in the convex setting. We use this simplified case to explain the general  proof strategy and point out the difficulty of extending the argument to the nonconvex setting.
Since we restrict ourselves to the convex setting, throughout this section (Section~\ref{sec:convex_setting}) we suppose:

\begin{itemize}
	\item Assumption~\ref{assumption:oneside} holds with $\quadapprox = 0$ and Assumption~\ref{assumption:sharp} holds with $\cS = \cX^\ast$. 
	\item The models $f_x(\cdot, \data)$ are $L(\data)$-Lipschitz on $\R^d$ for all $x$, where $L\colon\Omega\to\R$ is a measurable function satisfying $\sqrt{\EE_{\data}\left[L( \data)^2\right]} \leq \lipsymb$.
\end{itemize}
In particular, the tube $\cT_2$ is the entire space $\cT_2 = \RR^d$, which alleviates the main difficulty of the nonconvex setting.
The proof of convergence relies on the following known sublinear convergence guarantee for Algorithm~\ref{alg:stoc_prox}. 
\begin{thm}[{\cite[Theorem 4.1]{davis2019stochastic}}] \label{thm:convergenceC_nonstrong} 
Fix an initial point $y_0 \in \RR^d$ and let $\alpha = \frac{C}{\sqrt{\INNERIDX+1}}$ for some $C>0$. Then for any index $\INNERIDX \in \mathbb{N}$, the point  $y=\SMBM(y_0, \alpha, \INNERIDX, \true)$ satisfies
\begin{equation}\label{eqn:conv_1}
\EE\left[f\left(y\right) - \min_{\cX} f\right] \leq \frac{\tfrac{1}{2}\dist^2(y_0, \cX^\ast) + C^2\lipsymb^2 }{C\sqrt{\INNERIDX+1}}.
\end{equation}
\end{thm}

The proof of linear convergence now follows by iteratively applying Theorem~\ref{thm:convergenceC_nonstrong} with a carefully chosen parameter $C > 0$. The key idea of the proof is much the same as in the deterministic setting~\cite{nest_orig,NEMIROVSKII198521,NIPS2017_6712}. The proof proceeds by induction on the outer iteration counter $\outeridx$.  At the start of each inner iteration, we choose $C$ to minimize the ratio in Equation~\eqref{eqn:conv_1}, taking into account an inductive estimate on the initial square error $\dist^2(y_0, \cX^\ast)$. We then run the inner loop until the estimate decreases by a fixed fraction. This strategy differs from deterministic setting in only one way: since the output of the inner loop is random, we extract a bound on the initial distance using Markov's inequality. 

\begin{thm}[Linear convergence under convexity]\label{cor:setting}
Fix an initial point $x_0 \in \RR^d$, real $ \varepsilon>0$, $ \delta \in (0, 1)$, and an upper bound $R_0 \geq \dist(x_0, \cX^\ast)$. Define parameters  
$$
T : = \left\lceil \log_2 \left(\frac{R_0}{\varepsilon}\right) \right\rceil, \qquad  \qquad K:= \left\lfloor 8 \cdot T^2 \cdot \left( \frac{ \lipsymb}{\delta \mu}\right)^2\right\rfloor,\qquad \alpha_0 = \sqrt{\frac{R_0^2}{2\lipsymb^2(\INNERIDX+1)}}.
$$
Then with probability at least 
$
1   - \delta
$, 
the point 
$x_{\OUTERIDX} = \RSMBM(x_0, \alpha_0, \INNERIDX, \OUTERIDX, \true)$ satisfies
$
\dist(x_\OUTERIDX, \cX^*) \leq \varepsilon .
$ Moreover, the total number of samples $\data_\inneridx \sim \PP$ generated by the algorithm is bounded by 
$$
TK \leq 8 \left( \frac{ \lipsymb}{\delta \mu}\right)^2 \left\lceil\log_2 \left(\frac{R_0}{\varepsilon}\right)^3\right\rceil.
$$
\end{thm}
\begin{proof}
Let $x_t$ and $\alpha_t$ be generated by Algorithm~\ref{alg:stoc_prox_outer}.
In what follows, set $C_t=\frac{R_0}{\lipsymb\cdot 2^{t+\frac{1}{2}} }$ and note the equality $\alpha_t=\frac{C_t}{\sqrt{K+1}}$ for every index $t$.
Let $ E_\outeridx $ denote the event that $\dist(x_\outeridx, \cS) \leq 2^{-\outeridx} \cdot R_0$.  We wish to show the inequality
\begin{equation}\label{eqn:restart_ineq}
\PP(E_{t+1}) \geq \PP(E_{t})- \frac{ 2^{3/2} \lipsymb }{ \mu\sqrt{\INNERIDX+1}},
\end{equation}
for all $\outeridx \in \{0, \ldots, \OUTERIDX\}$. To that end, observe
\begin{equation}\label{eqn:need1}
\PP\left(E_{\outeridx+1}\right) \geq \PP(E_{\outeridx + 1} \mid E_\outeridx) \PP(E_{\outeridx}).
\end{equation}
To lower bound the right-hand-side, observe by Markov's inequality the estimate
\begin{equation}\label{eqn:need2}
\PP(E_{\outeridx + 1}^c \mid E_\outeridx) \leq \frac{\EE\left[  \dist(x_{t+1}, \cX^\ast)\mid E_\outeridx \right]}{2^{-(\outeridx+1)} R_0}  =  \frac{\EE\left[  \dist(x_{t+1}, \cX^\ast)1_{E_\outeridx}\right]}{2^{-(\outeridx+1)} R_0 \PP(E_\outeridx)} .
\end{equation}
Combining assumption~\ref{assumption:sharp} and
Theorem~\ref{thm:convergenceC_nonstrong}, we deduce
\begin{equation}\label{eqn:need3}
\begin{aligned}
\EE\left[  \dist(x_{t+1}, \cX^
\ast)1_{E_\outeridx}\right] \leq \EE\left[ \mu^{-1}(f(x_{\outeridx+1}) - \min_\cX f)1_{E_\outeridx}\right] &\leq \frac{\frac{1}{2}\EE\left[\dist^2(x_\outeridx,\cX^*) 1_{E_t}\right]  + C^2_\outeridx\lipsymb^2 }{\mu C_\outeridx\sqrt{\INNERIDX+1}}\\
&\leq \frac{2^{-(1+2\outeridx)} R_0^2 + C^2_\outeridx\lipsymb^2 }{\mu C_\outeridx\sqrt{\INNERIDX+1}} = \frac{2^{\frac{1}{2}-\outeridx} R_0 \lipsymb}{\mu\sqrt{\INNERIDX+1}}.
\end{aligned}
\end{equation}
Therefore, combining~\eqref{eqn:need1}, \eqref{eqn:need2}, and \eqref{eqn:need3} we arrive at the claimed estimate
\begin{align*}
\PP\left(E_{\outeridx+1}\right) \geq  \left(1 - \frac{1}{\PP(E_\outeridx)} \cdot \frac{2^{3/2} \lipsymb}{\mu\sqrt{\INNERIDX+1}}\right)  \PP(E_{\outeridx}) \geq \PP(E_\outeridx) - \frac{2^{3/2} \lipsymb}{\mu\sqrt{\INNERIDX+1}}.
\end{align*}
Iterating \eqref{eqn:restart_ineq} 
and using the definition of $T$ and $K$, we conclude that with probability
$$\PP(E_T) \geq 1 -  \frac{ 2^{3/2} T\lipsymb}{ \mu\sqrt{\INNERIDX+1}}\geq 1-\delta,
$$
the estimate
$$\dist(x_T,\cX^*)\leq 2^{-T} R_0\leq \varepsilon,$$
holds as claimed. This completes the proof.
\end{proof}

\subsection{Nonconvex setting}\label{sec:weaklyconvex}

We now present the convergence guarantees for Algorithm~\ref{alg:stoc_prox_outer} in the nonconvex setting under Assumptions~\ref{assumption:sample}-\ref{assumption:lipschitz}. The proof of linear convergence proceeds by inductively applying the following Lemma, which is similar to Lemma~\ref{thm:convergenceC_nonstrong}. Compared to the convex setting, the nonconvex setting presents a new challenge: the region of linear convergence, $\cT_\gamma$, is local. The iterates of Algorithm~\ref{alg:stoc_prox_outer} must therefore be shown to never leave $\cT_\gamma$. We show this through a simple stopping time argument in the proof of the following Lemma (see Section~\ref{sec:proof_key_lem}).

\begin{lem}\label{lem:key_lem}
	Fix real numbers $\initquallem \in (0, 1)$, $\gamma\in (0,2)$, $K \in \NN$, and $\alpha>0$. Let $y_0$ be a random vector and let $B$ denote the event $\{ y_0 \in \cT_{\gamma\sqrt{\delta}}\}$. Define 
	$$
	y_{\INNERIDX^\ast} = \SMBM(y_0, \alpha, \INNERIDX, \false).
	$$
	Then for any $\varepsilon > 0$, the estimate $\dist(y_{\INNERIDX^\ast}, \cS) \leq \varepsilon$ holds with probability at least
$$
 \PP(B) - \initquallem -  \left(\frac{\quadapprox}{\gamma\mu}\right)^2 \INNERIDX\lipsymb^2\alpha^2 - \frac{1}{\varepsilon } \cdot  \frac{ \initquallem\left(\frac{\gamma\mu}{\quadapprox}\right)^2 + (\INNERIDX+1)\lipsymb^2\alpha^2}{(2-\gamma) \mu(\INNERIDX+1) \alpha}. 
$$
\end{lem}

The proof of linear convergence of Algorithm~\ref{alg:stoc_prox_outer} in the nonconvex  setting now follows by inductively applying Lemma~\ref{lem:key_lem}.

\begin{thm}[Linear convergence without convexity]\label{thm:weak_setting}
Fix real numbers $ \varepsilon>0$, $ \stepprob \in (0, 1)$, $\gamma \in (0, 2)$. 
Let $R_0$ denote the initial distance estimate satisfying  $\dist(x_0, \cS) \leq R_0\leq \frac{\gamma \mu}{\eta}$. 
Furthermore, define algorithm parameters 
$$
T : = \left\lceil \log_2 \left(\frac{R_0}{\varepsilon}\right) \right\rceil, \qquad  \qquad K:= \left\lfloor \frac{16}{(2-\gamma)^2} \cdot T^2 \cdot \left( \frac{ \lipsymb}{\stepprob \mu}\right)^2\right\rfloor,\qquad \alpha_0 = \sqrt{\frac{R_0^2}{\lipsymb^2(\INNERIDX+1)}}.
$$
Then with probability at least 
$
1   - (8/3)R_0^2\left(\tfrac{\eta}{\gamma\mu}\right)^2  - \stepprob
$, 
the point 
$x_{\OUTERIDX} = \RSMBM(x_0, \alpha_0, \INNERIDX, \OUTERIDX, \false)$ satisfies
$
\dist(x_\OUTERIDX, \cS) \leq \varepsilon .
$ Moreover, the total number of samples $\data_\inneridx \sim \PP$ generated by the algorithm is bounded by 
$$
TK \leq  \frac{16}{(2-\gamma)^2} \left( \frac{ \lipsymb}{\delta_2 \mu}\right)^2 \left\lceil\log_2 \left(\frac{R_0}{\varepsilon}\right)^3\right\rceil.
$$
\end{thm}
\begin{proof}
 For all $\outeridx$, let $E_\outeridx$ be the event  $\{\dist(x_\outeridx, \cS) \leq  2^{-\outeridx}\cdot R_0\}$. In addition, define $\initqual := R_0^2\left(\tfrac{\eta}{\gamma\mu}\right)^2$. We claim that the inequality
$$\PP(E_{t+1})\geq \PP\left(E_t\right) - 2\initqual 2^{-2t} - \frac{4 \lipsymb}{(2-\gamma)\mu \sqrt{K+1}},$$
holds for all $t\in \{0,1\ldots, T\}$. To see this, apply Lemma~\ref{lem:key_lem} with $y_0 = x_t$, $\varepsilon=2^{-(t+1)}R_0$, $\initquallem:=\initqual 2^{-2t}$, $\alpha = 2^{-t} \alpha_0$, thereby yielding 
\begin{align*}
\PP(E_{\outeridx+1})
&\geq \PP\left(E_t\right) - \initqual 2^{-2t}  -  \left(\frac{\quadapprox}{\gamma\mu}\right)^2 \INNERIDX\lipsymb^2\alpha^2_t - \frac{1}{\varepsilon } \cdot  \frac{ \initquallem\left(\frac{\gamma\mu}{\quadapprox}\right)^2 + (\INNERIDX+1)\lipsymb^2\alpha^2_t}{(2-\gamma) \mu(\INNERIDX+1) \alpha_t} \\
&\geq \PP\left(E_t\right) - 2\initqual 2^{-2t} - \frac{4 \lipsymb}{(2-\gamma)\mu \sqrt{K+1}},
\end{align*}
where the last inequality follows from the definitions of $\alpha_0$ and $R_0$. Iterating the inequality, we conclude 
\begin{equation}\label{eqn:prob_lower_last_weak}
\PP(E_T) \geq 1 - 2\initqual\sum_{i = 0}^{T-1}  2^{-2i} - \frac{4 \lipsymb T}{(2-\gamma)\mu \sqrt{K+1}}\geq 1-(8/3)\delta_1-\delta_2.
\end{equation}
 This completes the proof.
\end{proof}

Observe that the probability of success $
1   - (8/3)R_0^2\left(\tfrac{\eta}{\gamma\mu}\right)^2  - \stepprob
$ in Theorem~\ref{thm:weak_setting} depends both on the initialization quality $R_0$ and on $\delta_2$. Moreover, $\delta_2$ also appears inversely in the sample complexity  $\widetilde{O}\left( \frac{ \lipsymb^2}{\delta_2^2 \mu^2} \right)$. In the next section, we introduce an algorithm with probability of success independent of $R_0$ and with sample complexity that depends only logarithmically on its success probability.

We close this section with the following remark, which shows that the results of Theorem~\ref{thm:weak_setting} extend beyond the weakly convex setting.

\begin{remark}[Beyond weakly convex problems]\label{remark:beyondweak}
{\rm Assumptions~\ref{assumption:oneside} and~\ref{assumption:convex} imply that $f$ is $\quadapprox$-weakly convex on $U$, meaning that the assignment $x \mapsto f(x) + \frac{\quadapprox}{2}\|x\|^2$ is convex \cite[Lemma 4.1]{davis2019stochastic}. Revisiting the proof of Lemma~\ref{lem:key_lem}, however, we see that~\ref{assumption:oneside} may be replaced by the following weaker assumption:

\begin{enumerate}[label=$\mathrm{\overline{(A\arabic*)}}$]
\setcounter{enumi}{2}
\item \label{assumption:twopoint} {\bf (Two-point accuracy)} There exists $\quadapprox > 0$ and an open convex set $U$ containing $\cX$ and a measurable function $(x,y,\data)\mapsto f_x(y,\data)$, defined on $U\times U\times\Omega$, satisfying  $$\EE_{\data}\left[f_x(x,\data)\right]=f(x)
	\qquad \forall x\in U,$$ and 
	$$\EE_{\data}\left[f_x(y,\data)-f(y)\right]
	\leq\frac{\quadapprox}{2}\|y-x\|^2\qquad \forall x \in U, y \in \argmin_{w \in  \proj_{\cS}(x)} f(w).$$
\end{enumerate}
In the case $\cS = \cX^\ast$, this assumption requires the model to touch the function at $x$ and to lower bound it, up to quadratic error, at its nearest minimizer. This condition does not imply that $f$ is weakly convex.
}
\end{remark}

\subsection{Convergence with high probability}\label{sec:highprob}
In this section, we show that Algorithm~\ref{alg:stoc_prox_outer_sc} succeeds with high probability. Throughout this section (Section~\ref{sec:highprob}), we impose Assumptions~\ref{assumption:sample}-\ref{assumption:lipschitz} with $\cS=\cX^*$.

The following lemma guarantees that with appropriate step size, the proximal point of the problem \eqref{eqn:stoch_approx} at $y\in \cT_{\gamma}$ lies in $\proj_{\cX^*}(y)$.
 We present the proof in Section~\ref{sec:aux_lemma}.

\begin{lem}\label{lem:prox_subproblem}
Fix $\gamma \in (0, 2)$, $\rho>\quadapprox$, and a point $y \in \cT_\gamma$. Then the proximal subproblem 
\begin{equation}\label{eqn:prox_subprob:strong}
\min_{x \in \cX} \left\{f(x) + \frac{\rho}{2}\|x - y\|^2\right\} 
\end{equation}
is strongly convex and therefore has a unique minimizer $\bar y$. Moreover, if $ \rho < \left(\tfrac{2-\gamma}{2\gamma}\right)\quadapprox$, then the inclusion $\bar y\in\proj_{\cX^\ast} (y)$ holds.
\end{lem}

Lemma~\ref{lem:prox_subproblem} shows that, unlike Algorithm~\ref{alg:stoc_prox}, we can expect the output of Algorithm~\ref{alg:stoc_prox_sc} to be near the minimizer $\bar y_0 \in \cX^\ast$ of the proximal subproblem $f(y) + \frac{\rho}{2} \|y - y_0\|^2$, at least with constant probability. This lemma underlies the validity of Lemma~\ref{lem:key_lem_sc}. We present its proof in Section~\ref{sec:proof_key_lem_sc}.

\begin{lem}\label{lem:key_lem_sc}
	Fix real numbers $\initquallem \in (0, 1)$, $\gamma \in (0, 2)$, $\alpha>0$ $K \in \NN$, and $\rho$ satisfying $\eta < \rho < \left(\frac{2-\gamma\sqrt{\initquallem} }{2\gamma\sqrt{\initquallem}}\right)\quadapprox$.  Choose any point $y_0 \in \cT_{\gamma\sqrt{\initquallem}}$
	and set 
	$
	\bar y_0 := \argmin_{x \in \cX}  \left\{ f(x) + \frac{\rho}{2} \| x - y_0\|^2\right\}.
	$
	Define 
	$$
	y_{\INNERIDX^\ast} = \SMBMSC(y_0, \rho, \alpha,  \INNERIDX).
	$$
	Then for all $\varepsilon > 0$, with probability at least
\begin{align*}
P(\varepsilon) &:= 1 - \initquallem -  \left(\frac{\quadapprox}{\gamma\mu}\right)^2\INNERIDX\lipsymb^2 \alpha^2 - \frac{1}{\varepsilon^2} \cdot \frac{(K+1) \lipsymb^2 \alpha + (\alpha^{-1} + \quadapprox) \cdot\initquallem \left(\frac{\gamma \mu}{\quadapprox}\right)^2}{(\rho - \quadapprox)(\INNERIDX + 1)}
\end{align*}
we have 
	$\|y_{\INNERIDX^\ast} - \bar y_0\| \leq \varepsilon$.
\end{lem}

The next lemma shows that we can boost the success probability of the inner loop arbitrarily high at only a logarithmic cost. This is an immediate application of Lemma~\ref{lem:key_lem_sc} and the ensemble technique described in Section~\ref{sec:algorithmandresults}, which is formally stated in Lemma~\ref{lem:ensemble}.

\begin{cor}\label{cor:ensemble}
Assume the setting of Lemma~\ref{lem:key_lem_sc} and suppose $P(\varepsilon/3) \geq 2/3$. Then for any $\delta' > 0$, in the regime $\TRIALCOUNT > 48 \log(1/\delta')$, the point
$$\hat y = \SMBMSCE(y_0, \rho, \alpha, \INNERIDX, \TRIALCOUNT, \varepsilon/3)$$
satisfies
$$\PP(\|\hat y - \bar y_0\| \leq \varepsilon)\geq 1-\delta'.$$
\end{cor}

Finally, we are ready to establish convergence guarantees of Algorithm~\ref{alg:stoc_prox_outer_sc}.

\begin{thm}[Linear convergence with high probability]\label{thm:high_prob}
Fix constants $\gamma \in (0, 2)$, $\varepsilon>0$, and $\delta' \in (0, 1)$. Let $R_0$ denote the initial distance estimate satisfying  $\dist(x_0, \cX^\ast) \leq R_0\leq \frac{\gamma \mu}{4\quadapprox}$. 
Furthermore, define algorithm parameters 
$$
\rho_0 = \frac{\mu}{2R_0}, \qquad\qquad \epsilon_0 = \frac{ R_0}{3}, \qquad\qquad \alpha_0 = \sqrt{\frac{R_0^2}{\lipsymb^2(\INNERIDX+1)}},  
$$
and 
$$
M = \left\lceil48 \log(\OUTERIDX/\delta')\right\rceil,\qquad\qquad T  = \left\lceil \log_2 \left(\frac{R_0}{\varepsilon}\right) \right\rceil, \qquad  \qquad K= \left\lfloor   \left( \frac{864 \lipsymb}{ \mu}\right)^2\right\rfloor.
$$
Then with probability at least $
1 -  \delta'$, the point 
$$x_{\OUTERIDX} = \RSMBMSCE(x_0, \rho_0, \alpha_0, \INNERIDX,\epsilon_0 , \TRIALCOUNT, \OUTERIDX)$$
satisfies 
$
\dist(x_\OUTERIDX, \cX^*) \leq \varepsilon.
$
Moreover, the total number of samples $\data_\inneridx \sim \PP$ generated by the algorithm is bounded by 
$$
 K T   M \leq \left(\frac{864 \lipsymb}{\mu}\right)^2\cdot\left\lceil\log_2\left(\frac{R_0}{\varepsilon}\right)\right\rceil\cdot\left\lceil48\log\left(\frac{\left\lceil\log_2\left(\frac{R_0}{\varepsilon}\right)\right\rceil}{\delta'}\right)\right\rceil.
$$ 
\end{thm}
\begin{proof}
 For all $\outeridx$, let $E_\outeridx$ be the event  $\{\dist(x_\outeridx, \cX^\ast) \leq 2^{-t}\cdot R_0\}$. 
Our goal is to show for all $\outeridx \in \{0, \ldots, \OUTERIDX\}$ the estimate $\PP(E_\outeridx) \geq 1 - \outeridx\delta'/\OUTERIDX$ holds. 

We proceed by induction.
The base case follows since $\PP(E_0) = 1$ by definition of $R_0$. Now suppose that the claimed estimate $\PP(E_\outeridx) \geq 1 - \outeridx\delta'/\OUTERIDX$ is true for index $\outeridx$. We will show it remains true with $t$ replaced by $t+1$. We will apply Lemma~\ref{lem:key_lem_sc} conditionally with $y_0 = x_t$ and error tolerance $\epsilon_t$ in the event $E_\outeridx$. To this end, we  define $\initqual := R_0^2\left(\tfrac{\eta}{\gamma\mu}\right)^2$ and set $\rho = \rho_\outeridx$, $\initquallem:=\initqual2^{-2t}$, $\alpha = \alpha_\outeridx $. Before we apply the Lemma, we verify that $\rho$ meets the conditions of Lemma~\ref{lem:key_lem_sc}, namely that $\quadapprox < \rho <  \left(\frac{2- \gamma\sqrt{\initqual}2^{-t}}{2\gamma\sqrt{\initqual} 2^{-t}}\right) \quadapprox$. Indeed, given that $\rho = \frac{2^t}{2\gamma\sqrt{\initqual}}\cdot \quadapprox$ (by definition of $\initqual$), the bounds follow immediately from the restrictions $\delta_1<1/16$ and $\gamma \leq 2$. In particular, it is straightforward to verify the bound
%
%
\begin{align}\label{eq:rho_bounds}
\rho - \quadapprox \geq \frac{\quadapprox 2^{\outeridx-2}}{\gamma\sqrt{\initqual}} = \frac{2^{\outeridx}\mu}{4R_0}.
\end{align}
%
Now Lemma~\ref{lem:key_lem_sc} yields that the random vector $ y_{\INNERIDX^\ast} = \SMBMSC(x_t, \rho_t, \alpha_\outeridx, \INNERIDX_t)$ satisfies 
\begin{align*}
 &\PP\left(\|y_{\INNERIDX^\ast} - \bar y_0\| \leq \epsilon_t \mid E_t, y_0 = x_t\right)\\
 &\geq 1 - \initqual 2^{-2\outeridx} -  \left(\frac{\quadapprox}{\gamma\mu}\right)^2\INNERIDX\lipsymb^2 \alpha^2 - \frac{1}{\epsilon_t^2} \cdot \frac{(K+1) \lipsymb^2 \alpha + (\alpha^{-1} + \quadapprox) \cdot\initquallem \left(\frac{\gamma \mu}{\quadapprox}\right)^2}{(\rho - \quadapprox)(\INNERIDX + 1)}\\
  &\geq 1 - 2\initqual 2^{-2\outeridx}  - \frac{9}{2^{-2(t+1)} R_0^2} \cdot \frac{\lipsymb 2^{-t} R_0}{(\rho - \quadapprox)\sqrt{\INNERIDX + 1}} - \frac{36\quadapprox }{(\rho - \quadapprox) (K+1)}\\
 &\geq 1 - 2\initqual 2^{-2\outeridx}  -  \frac{144(\lipsymb/\mu) }{\sqrt{\INNERIDX + 1}} - \frac{144\gamma\sqrt{\initqual 2^{-2t}} }{K+1} \geq 2/3,
\end{align*}
where the second inequality follows from~\eqref{eq:rho_bounds}, while the third inequality uses the definition of $K$ and the bound $\delta_1 < 1/16$ and $\lipsymb \geq \mu$.  Therefore, since $M \geq 48\log(\OUTERIDX/\delta')$, we may apply  Corollary~\ref{cor:ensemble} (conditionally) to deduce 
$$
\PP\left(\|x_{\outeridx + 1} - \bar y_0\| \leq 3\epsilon_t \mid E_t, y_0 = x_\outeridx \right) \geq 1 - \delta'/\OUTERIDX.
$$
Consequently, 
\begin{align*}
\PP\left(\dist(x_{\outeridx + 1}, \cX^\ast) \leq 2^{-(\outeridx + 1)}R_0 \right) 
&\geq \PP\left(\dist(x_{\outeridx + 1}, \cX^\ast) \leq  2^{-(\outeridx + 1)}R_0\mid E_\outeridx \right) \PP(E_\outeridx)  \\
&\geq \EE_{y_0}\left[ \PP\left(\|x_{\outeridx + 1} - \bar y_0\| \leq  2^{-(\outeridx + 1)}R_0\mid E_\outeridx, y_0 = x_t \right) \right] \PP(E_\outeridx) \\
&\geq (1-\delta'/\OUTERIDX)(1-\outeridx\delta'/\OUTERIDX) \\
&\geq 1 - (\outeridx + 1)\delta'/\OUTERIDX,
\end{align*}
as desired. This completes the proof. 
\end{proof}




\section{Consequences for statistical recovery problems}\label{sec:examples}

Recent work has shown that a variety of statistical recovery problems are both
sharp and weakly convex. Prominent examples include robust matrix sensing
\cite{li2018nonconvex}, phase retrieval \cite{duchi_ruan_PR}, blind
deconvolution \cite{charisopoulos2019composite}, quadratic and bilinear sensing
and matrix completion \cite{charisopoulos2019low}. In this section, we briefly
comment on how our current work leads to linearly convergent
streaming algorithms for robust phase retrieval and blind deconvolution problems.

\subsection{Robust Phase retrieval}
Phase retrieval is a common task in computational science, with numerous applications
including imaging, X-ray crystallography, and speech processing. In this section, we consider the real counterpart of this problem. For details and a historical account of the phase retrieval problem, see for example \cite{duchi_ruan_PR,goldstein2018phasemax,7078985,MR3069958}.
Throughout this section, we fix a signal $\bar x \in \RR^d$ and consider the following
measurement model.

\begin{assumption}[Robust Phase Retrieval]\label{assump:phase_assump} Consider random $a \in \RR^d$, $\xi \in \RR$, and $\bern \in \{0, 1\}$ and the measurement model
$$
b = (a^T \bar x)^2 + \bern \cdot \xi.
$$
We make the following assumptions on the random data.
\begin{enumerate}
\item The variable $\bern$ is independent of $\xi$ and $a$. The failure probability $\pfail$ satisfies $$\pfail := P(\bern \neq 0) < 1/2.$$
\item The first absolute moment of $\xi$ is finite, $\EE\left[ |\xi| \right] < \infty$.
\item There exist constants $\weakphase, \tilde \mu, \tilde \lipsymb > 0$ such that for all $v, w \in \sphere^{d-1}$, we have
$$
\tilde \mu \leq \EE \left[|\dotp{a, v}\dotp{a, w} |\right],  \qquad \sqrt{\EE\left[\dotp{a, v}^2\|a\|^2\right]} \leq \tilde \lipsymb, \qquad \EE\left[ \dotp{a, v}^2\right] \leq \weakphase.
$$
\end{enumerate}
\end{assumption}

Based on the above assumptions, the following theorem develops three models for the robust phase retrieval problem. We defer the proof to Section~\ref{appendix:thm:phase_params}.
\begin{thm}[Phase retrieval parameters]\label{thm:phase_params}
Consider the population data $z = (a, \bern, \xi)$ and form the optimization problem 
$$
\min_{x}~f(x) = \EE_{z}\left[ f(x, z)\right]\quad  \text{ where } \quad f(x, z) :=  | (a^Tx)^2 - b|.
$$
Then the sharpness property~\ref{assumption:sharp} holds with $\cS=\cX^*=\{\pm \bar x\}$ and $\mu=(1-2\pfail)\tilde \mu \|\bar x\|$.  Moreover, given a measurable selection $G(x, z) \in \partial f(x, z)$, the models
\begin{itemize}
\item[] {\bf (subgradient)} $f^s_x(y, z) = f(x, z) + \dotp{G(x, z), y - x}$
\item[] {\bf (clipped subgradient)} $f^{cl}_x(y, z) = \max \{ f(x, z) + \dotp{G(x, z), y - x}, 0\}$
\item[] {\bf (prox-linear)} $f^{pl}_x(y, z) = | (a^Tx)^2 - b +  2(a^T x) a^T(y - x)| $
\end{itemize}
 satisfy Assumptions~\ref{assumption:sample}-\ref{assumption:lipschitz} with $\quadapprox = 2\weakphase$, $L(x,
 z) = 2|\dotp{a, x}|\|a\|$ and $\lipsymb \leq 2\tilde \lipsymb \|\bar x\|
 \left(1 + \frac{(1-2\pfail)\tilde \mu}{\weakphase}\right)$. 
\end{thm}

With this theorem in hand, we deduce that on the phase retrieval problem, Algorithm~\ref{alg:stoc_prox_outer_sc} with subgradient, clipped subgradient, and prox-linear models converges linearly to $\cX^\ast$ with high probability, whenever the method is initialized within constant relative error of the optimal solution.
\begin{thm}\label{thm:phase_convergence}
Fix constants $\gamma\in (0,2)$, $\varepsilon>0$, and $\delta'\in (0, 1)$.
Consider the subgradient, clipped subgradient, and prox-linear oracles
developed in Theorem~\ref{thm:phase_params}. Suppose we are given a point $x_0$
satisfying
$$
\dist(x_0, \{\pm x\}) \leq \gamma   \frac{(1-2\pfail)\tilde \mu }{8\weakphase} \cdot \|\bar x\|.
$$
Set parameters $ \rho_0, \alpha_0, \INNERIDX,\epsilon_0 , \TRIALCOUNT,
\OUTERIDX$ as in Theorem~\ref{thm:high_prob}. In addition, define the iterate
$x_\OUTERIDX = \RSMBMSCE(x_0, \rho_0, \alpha_0, \INNERIDX,\epsilon_0 ,
\TRIALCOUNT, \OUTERIDX)$. Then with probability $1-\delta'$, we have $\dist(x_\OUTERIDX,
\{\pm \bar x\}) \leq \varepsilon \|\bar x\|$ after
$$
O\left(  \left(\frac{\tilde \lipsymb \left(1 + \frac{(1-2\pfail)\tilde \mu}{2\weakphase}\right)}{ \tilde \mu(1-2\pfail) }\right)^2\log\left(\frac{(1-2\pfail)\tilde \mu /\weakphase}{\varepsilon}\right)\log\left(\frac{\log\left(\frac{(1-2\pfail)\tilde \mu /\weakphase}{\varepsilon}\right)}{\delta'}\right)\right)
$$
stochastic subgradient, stochastic clipped subgradient, or stochastic prox-linear iterations.
\end{thm}

We now examine Theorem~\ref{thm:phase_convergence} in the setting where the
measurement vectors $a$ follow a Gaussian distribution. We note, however, that
the results of this section extend far beyond the Gaussian setting to heavy
tailed distributions.

\begin{example}[Gaussian setting]\label{example:gaussianphase}
{\rm Let us analyze the population setting where $a \sim N(0, I_{d\times d})$.
In this case, it is straightforward to show by direct computation that
$$
\tilde \mu \gtrsim 1; \qquad \eta = 1; \qquad \lipsymb \lesssim \sqrt{d}.
$$
Consequently, if $x_0 \in \RR^d$ has error
$
\dist(x_0, \{\pm \bar x\}) \leq c(1-2\pfail) \cdot \|\bar x\|,
$ for some numerical constant $c$,
then with probability $1-\delta$, Algorithm~\ref{alg:stoc_prox_outer_sc} will
produce a point $x_\OUTERIDX$ satisfying $\dist(x_\OUTERIDX, \{\pm x\}) \leq
\varepsilon \| \bar x\|$ using only
$$
O\left(\frac{d}{(1-2\pfail)^2} \log\left(\frac{1}{\varepsilon}\right)\log\left(\frac{\log\left(\frac{1}{\varepsilon}\right)}{\delta}\right)\right)
$$
samples. We note that the spectral initialization of Duchi and Ruan~\cite[Proposition 3]{duchi_ruan_PR} produces such a point $x_0$ with sample complexity $O(d(1-2\pfail)^{-2})$ with high probability. Therefore, when taken together, combining this spectral initialization with Algorithm~\ref{alg:stoc_prox_outer_sc} produces a point $x_T$ satisfying $\dist(x_\OUTERIDX, \{\pm \bar x\}) \leq \varepsilon \| \bar x\|$ with $O\left(\frac{d}{(1-2\pfail)^2} \log\left(\frac{1}{\varepsilon}\right)\log\left(\log\left(\frac{1}{\varepsilon}\right)/\delta\right)\right)$ samples, which is the best known sample complexity for Gaussian robust phase retrieval, up to logarithmic factors. We note by leveraging standard concentration results, it is possible to prove similar results for empirical average minimization $\min_x \frac{1}{m}\sum_{i=1}^m f(x,z_i)$, provided $z_i$ are i.i.d samples of $z$ and the number of samples satisfies 
 $m \gtrsim d (1-2\pfail)^{-2}$. }
\end{example}

\subsection{Robust blind deconvolution}
We next apply the proposed algorithms to the blind deconvolution problem. For a detailed discussion of the the problem, see for example the papers \cite{ahmed2014blind,li2018rapid}.
Henceforth, fix integers $d_1, d_2 \in \NN$ and an underlying signal $(\bar x, \bar y) \in \RR^{d_1}\times \R^{d_2}$.
Define the quantity
$$
\blindnorm := \|\bar x\|\|\bar y\|.
$$
Without loss of generality, we will assume $\|\bar x\| = \|\bar y\|$.  We consider
the following measurement model:

\begin{assumption}[Robust Blind Deconvolution]\label{assump:blind_assump}
Consider random $\ell \in \RR^{d_1}$, $r \in \RR^{d_2}$, $\xi \in \RR$, and
$\bern \in \{0, 1\}$ and the measurement model
$$
b = \dotp{\ell, \bar x}\dotp{r, \bar y}  + \bern \cdot \xi.
$$
We make the following assumptions on the random data.
\begin{enumerate}
\item The variable $\bern$ is independent of $\xi$, $\ell$, and $r$. The failure probability $\pfail$ satisfies $$\pfail := P(\bern \neq 0) < 1/2.$$
\item We have $\EE\left[ |\xi| \right] < \infty$.
\item There exists constants $\weakphase, \tilde \mu, \tilde \lipsymb > 0$ such
that for all $M\in \RR^{d_1 \times d_2}$ with $\|M\|_F = 1$ and $\rank(M) \leq
2$, we have
$$
\tilde  \mu \leq \EE\left[ |\ell^T M r| \right] \leq \weakphase
$$
\item There exists constants $\tilde \lipsymb > 0$ such that for all $v \in \sphere^{d_1-1}, w \in \sphere^{d_2-1}$, we have
$$
 \sqrt{\EE\left[\left(|\dotp{\ell, v}| \|r\| + |\dotp{r, w}|\|\ell\|\right)^2\right]} \leq \tilde \lipsymb.
$$
\end{enumerate}
\end{assumption}

Based on the above assumptions, the following theorem develops three models for the robust blind deconvolution problem. We defer the proof to Section~\ref{appendix:thm:blind_params}.
\begin{thm}[Blind deconvolution parameters]\label{thm:blind_params}
	Fix a real $\nu > 1$ and define the  set:
	$$
	\cX = \{ (x, y) \in \RR^{d_1 + d_2} \colon \|x\| \leq \nu \blindnorm, \| y\| \leq \nu \blindnorm\}.
	$$
Consider the population data $z = (a, \bern, \xi)$ and the form the optimization problem
$$
\min_{x,y} f(x, y) := \EE\left[ f((x,y), z)\right]\quad  \text{ where } \quad f((x,y), z) :=  | \dotp{\ell,  x}\dotp{r,  y} - b|.
$$
Then the optimal solution set is $\cX^\ast = \{(\alpha \bar x, (1/\alpha)\bar y) \mid (1/\nu) \leq |\alpha| \leq \nu \}$ and $f$ satisfies the sharpness assumption~\ref{assumption:sharp} with $\cS=\cX^*$ and 
$$
\mu =   \frac{\tilde \mu (1-2\pfail)\sqrt{\blindnorm}}{2\sqrt{2}(\nu + 1)}.
$$
 Moreover, given a measurable selection $G((x,y), z) \in \partial f((x,y), z)$, the models
\begin{itemize}
\item[] {\bf (subgradient)} $f^s_{(x,y)}((\hat x, \hat y), z) = f((x, y), z) + \dotp{G((x, y), z), (\hat x, \hat y) - (x, y)}$
\item[] {\bf (clipped subgradient)} $$f^{cl}_{(x, y)}((\hat x, \hat y), z) = \max \{ f((x, y), z) + \dotp{G((x, y), z), (\hat x, \hat y) - (x, y)}, 0\}$$
\item[] {\bf (prox-linear)} $$f^{pl}_{(x, y)}((\hat x, \hat y), z) = | \dotp{\ell,  x}\dotp{r,  y} - (\dotp{\ell, \bar x}\dotp{r, \bar y}  + \bern \cdot \xi) +  \dotp{\ell, x}\dotp{r, \hat y - y} + \dotp{r, y} \dotp{\ell , \hat x -x}| .$$
\end{itemize}
 satisfy Assumptions~\ref{assumption:sample}-\ref{assumption:lipschitz} with $\quadapprox = \weakphase$, $L((x,
 y), z) = |\dotp{\ell, x}| \|r\| + |\dotp{r, y}|\|\ell\|$ and $\lipsymb
 =\nu\tilde \lipsymb \sqrt{D}  $. 
\end{thm}

With this theorem in hand, we deduce that on the blind deconvolution problem, Algorithm~\ref{alg:stoc_prox_outer_sc} with subgradient, clipped subgradient, and prox-linear models converges linearly to $\cX^\ast$ with high probability, whenever the method is initialized within constant relative error of the solution set.
\begin{thm}\label{thm:blind_convergence}
Fix constants $\gamma\in (0,2)$, $\varepsilon>0$, and $\delta'\in (0, 1)$.
Consider the subgradient, clipped subgradient, and prox-linear oracles
developed in Theorem~\ref{thm:phase_params}. Suppose we are given a pair $(x_0,y_0)
\in \RR^{d_1 + d_2}$ satisfying
$$
\dist((x_0, y_0), \cX^\ast ) \leq  \gamma \frac{\tilde \mu (1-2\pfail)\sqrt{\blindnorm}}{8\sqrt{2}(\nu + 1)\weakphase}.
$$
Set parameters $\rho_0, \alpha_0, \INNERIDX,\epsilon_0 , \TRIALCOUNT,
\OUTERIDX$ as in Theorem~\ref{thm:high_prob}. In addition, we define the
iterate $(x_\OUTERIDX,y_\OUTERIDX)= \RSMBMSCE((x_0, y_0), \rho_0, \alpha_0, \INNERIDX,\epsilon_0
, \TRIALCOUNT, \OUTERIDX)$. Then with probability $1-\delta'$, we have
$\dist((x_\OUTERIDX,y_\OUTERIDX), \cX) \leq \varepsilon \sqrt{D}$ after
$$
O\left(  \left(\frac{\nu^2\tilde \lipsymb}{ \tilde \mu (1-2\pfail) }\right)^2\log\left(\frac{(1-2\pfail)\tilde \mu/\weakphase}{\varepsilon}\right)\log\left(\frac{\log\left(\frac{(1-2\pfail)\tilde \mu /\weakphase}{\varepsilon}\right)}{\delta'}\right)\right)
$$
 stochastic subgradient, stochastic clipped subgradient, or stochastic
 prox-linear iterations.
\end{thm}

We now examine Theorem~\ref{thm:blind_convergence} in the setting where the
measurement vectors $\ell, r$ follow a Gaussian distribution. We note, however,
that the results of this section extend far beyond the Gaussian setting to
heavy tailed distributions.

\begin{example}[Gaussian setting]\label{example:gaussianblind}
{\rm Let us analyze the population setting where $(\ell, r) \sim N(0, I_{(d_1 +
d_2) \times (d_1 + d_2)})$. In this case, one can show by direct computation
that
$$
\tilde \mu \gtrsim 1; \qquad \eta \lesssim 1; \qquad \lipsymb \lesssim \sqrt{d_1 + d_2}.
$$
Consequently, if $(x_0, y_0) \in \RR^{d_1 + d_2}$ has error
$
\dist((x_0, y_0), \cX^\ast) \leq c(1-2\pfail) \cdot \sqrt{D}/\nu,
$ for some numerical constant $c>0$,
then with probability $1-\delta$, Algorithm~\ref{alg:stoc_prox_outer_sc} will
produce a pair $(x_\OUTERIDX, y_\OUTERIDX)$ satisfying $\dist((x_\OUTERIDX,
y_\OUTERIDX), \cX^\ast) \leq \varepsilon \sqrt{D}$ using only
$$
O\left(\frac{\nu^2(d_1 + d_2)}{(1-2\pfail)^2} \log\left(\frac{1}{\varepsilon}\right)\log\left(\frac{\log\left(\frac{1}{\varepsilon}\right)}{\delta}\right)\right)
$$
samples. We note that the spectral initialization of  Charisopoulos et al.~\cite[Theorem 5.4 and Corollary 5.5]{charisopoulos2019composite} can produce such a pair $(x_0, y_0)$ with sample complexity $O(\nu^2(d_1 + d_2) (1-2\pfail)^{-2})$ with high probability with $\nu \leq \sqrt{3}$. Therefore, when taken together, combining this spectral initialization with Algorithm~\ref{alg:stoc_prox_outer_sc} produces a pair $(x_T, y_T)$ satisfying $\dist((x_\OUTERIDX, y_\OUTERIDX), \cX^\ast) \leq \varepsilon \| \bar x\|$ with $O\left(\frac{d_1 + d_2}{(1-2\pfail)^2} \log\left(\frac{1}{\varepsilon}\right)\log\left(\log\left(\frac{1}{\varepsilon}\right)/\delta\right)\right)$ samples, which is the best known sample complexity for Gaussian robust blind deconvolution, up to logarithmic factors. We note by leveraging standard concentration results, it is possible to prove similar results for empirical average minimization $\min_{(x,y)\in \cX} \frac{1}{m}\sum_{i=1}^m f((x,y),z_i)$, provided $z_i$ are i.i.d samples of $z$ and the number of samples satisfies 
 $m \gtrsim (d_1 + d_2) (1-2\pfail)^{-2}$.}
\end{example}

\section{Numerical Experiments}\label{sec:numerics}

We now evaluate how Algorithm~\ref{alg:stoc_prox_outer} performs both on the
statistical recovery problems of Section~\ref{sec:examples} and on a sparse logistic regression problem. We test the convergence behavior, sensitivity to step size, and convergence to an active manifold. While testing the algorithms, we found that Algorithms~\ref{alg:stoc_prox_outer} and~\ref{alg:stoc_prox_outer_sc} perform similarly, despite Algorithm~\ref{alg:stoc_prox_outer_sc} having superior theoretical guarantees. Thus, we do not evaluate Algorithm~\ref{alg:stoc_prox_outer_sc}.
The problems of Section~\ref{sec:examples} are both convex composite losses of the form in Example~\ref{example:composite}. For these problems, we therefore implement all four models from Example~\ref{example:composite}, using the closed-form solutions developed in~\cite[Section 5]{davis2019stochastic}. For the sparse logistic regression problem, we implement the stochastic proximal gradient method and measure convergence to the optimal support pattern.
We provide a reference implementation~\cite{RefImpl} of the methods in \texttt{Julia}.

\subsection{Convergence behavior}\label{sec:convergencebehav}
In this section, we demonstrate that Algorithm~\ref{alg:stoc_prox_outer} converges linearly on the Gaussian
robust phase retrieval and blind deconvolution problems of Section~\ref{sec:examples} for a particular dimension, noise distribution, corruption frequency, and initialization quality. In phase retrieval, we set $d = 100$ and in blind deconvolution, we set $d_1 = d_2 = d. $ The measurements are corrupted independently with probability $\pfail$: for phase retrieval, the corruption obeys $\xi = \abs{g},
\; g \sim N(0, 100)$, while for blind deconvolution, it obeys $\xi \sim N(0, 100)$.
The algorithms are all randomly initialized at a fixed distance $R_0 > 0$ from the ground truth. The ground truth is normalized in all cases.
%
We use Examples~\ref{example:gaussianphase} and~\ref{example:gaussianblind} to estimate $\lipsymb$, $\quadapprox$, and $\mu$, and we set $\gamma = 1$, $R_0 = 0.25$, $\delta_2 =
\frac{1}{\sqrt{10}}$, and target accuracy $\varepsilon = 10^{-5}$ to obtain
$T, K$ and $\alpha_0$ parameters as in Theorem~\ref{thm:weak_setting}


Figures \ref{fig:conv-finsample-pfail-2} and \ref{fig:conv-pfail-2} depict the convergence behavior of
Algorithm~\ref{alg:stoc_prox_outer} on robust
phase retrieval and blind deconvolution problems in finite sample and streaming settings, respectively. In these plots, solid lines with markers show the
mean behavior over $10$ runs, while the transparent overlays show one sample standard deviation above and below the mean.
In the finite-sample
instances,  we use $m = 8 \cdot d$ measurements and corrupt a fixed fraction $\pfail$ with large magnitude sparse noise; see Figure~\ref{fig:conv-finsample-pfail-2}. In the streaming instances, we draw a new i.i.d.\ sample at each iteration and corrupt it independently with probability $\pfail$; see Figure~\ref{fig:conv-pfail-2}. In both figures, we plot in red the rate guaranteed by Theorem~\ref{thm:weak_setting} and observe that the algorithms behave consistently with these guarantees.
In presence of noise, the algorithms all converge linearly at the rate predicted by Theorem~\ref{thm:weak_setting}, while
in the noiseless case, all except the
subgradient method converge to an exact solution (modulo numerical accuracy) within far fewer iterations.

\begin{figure}[h!]
	\centering
	\begin{minipage}{0.48 \textwidth}
		\includegraphics[width=\linewidth]{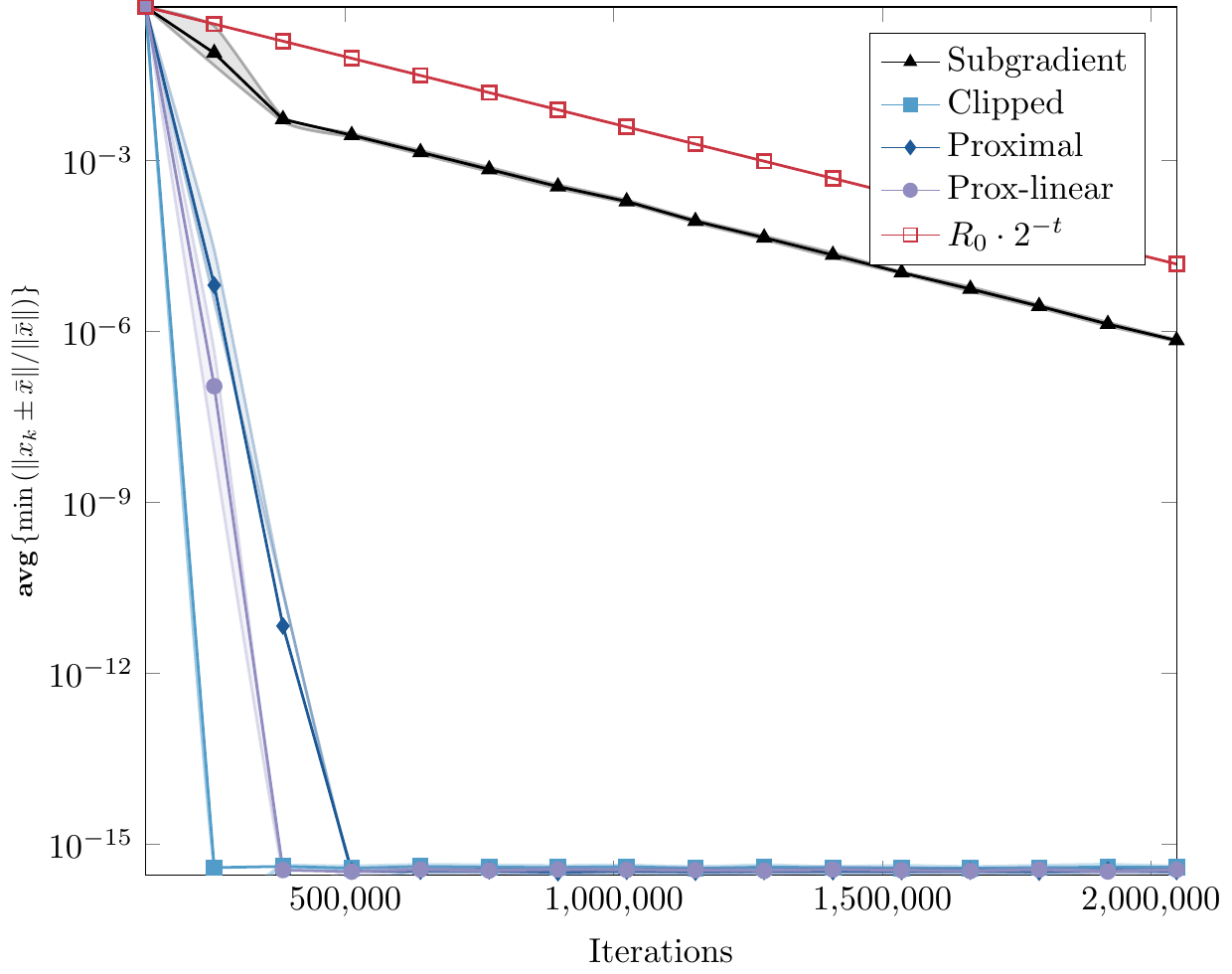}
	\end{minipage}
	\begin{minipage}{0.48 \textwidth}
		\includegraphics[width=\linewidth]{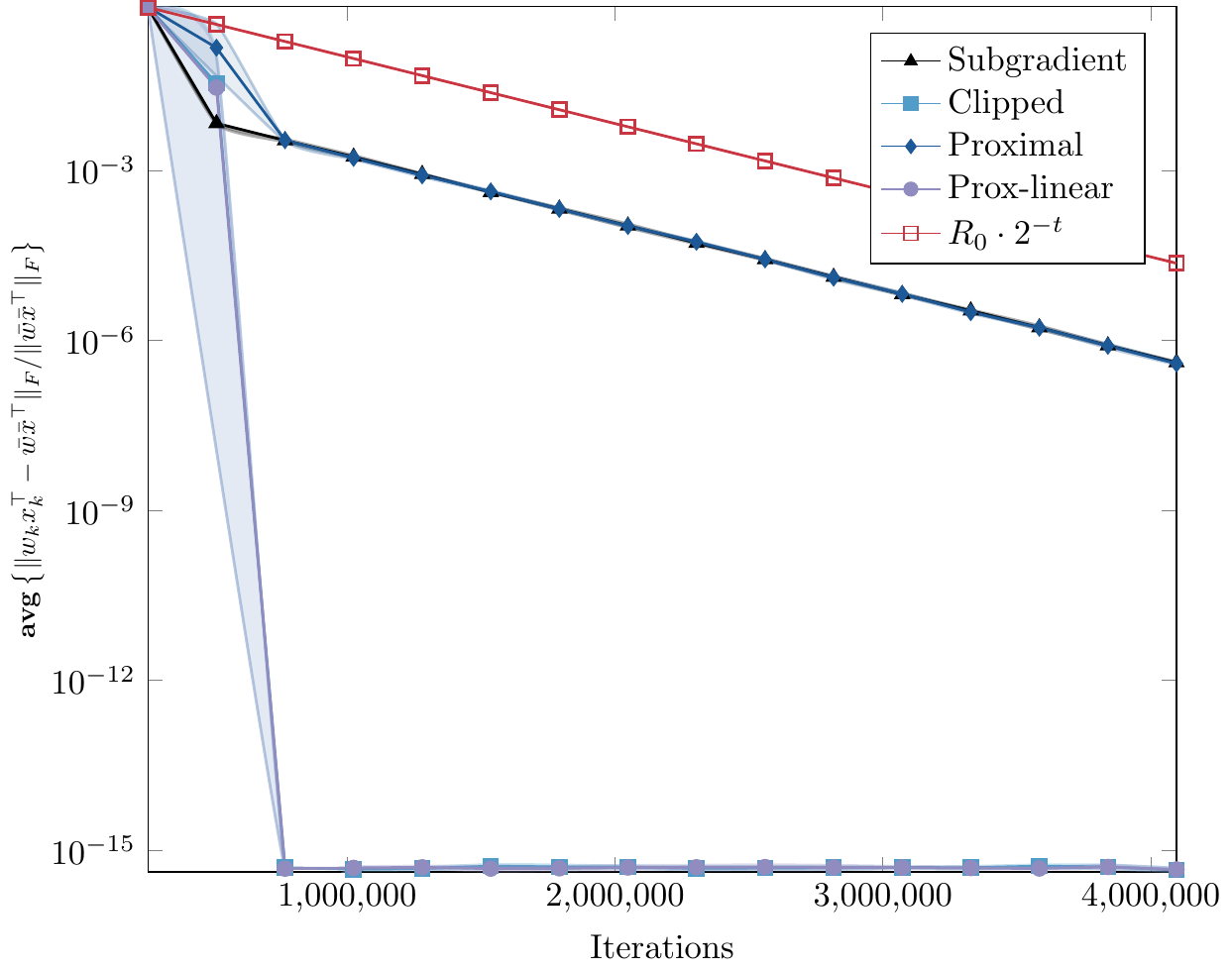}
	\end{minipage}

	\begin{minipage}{0.48 \textwidth}
		\includegraphics[width=\linewidth]{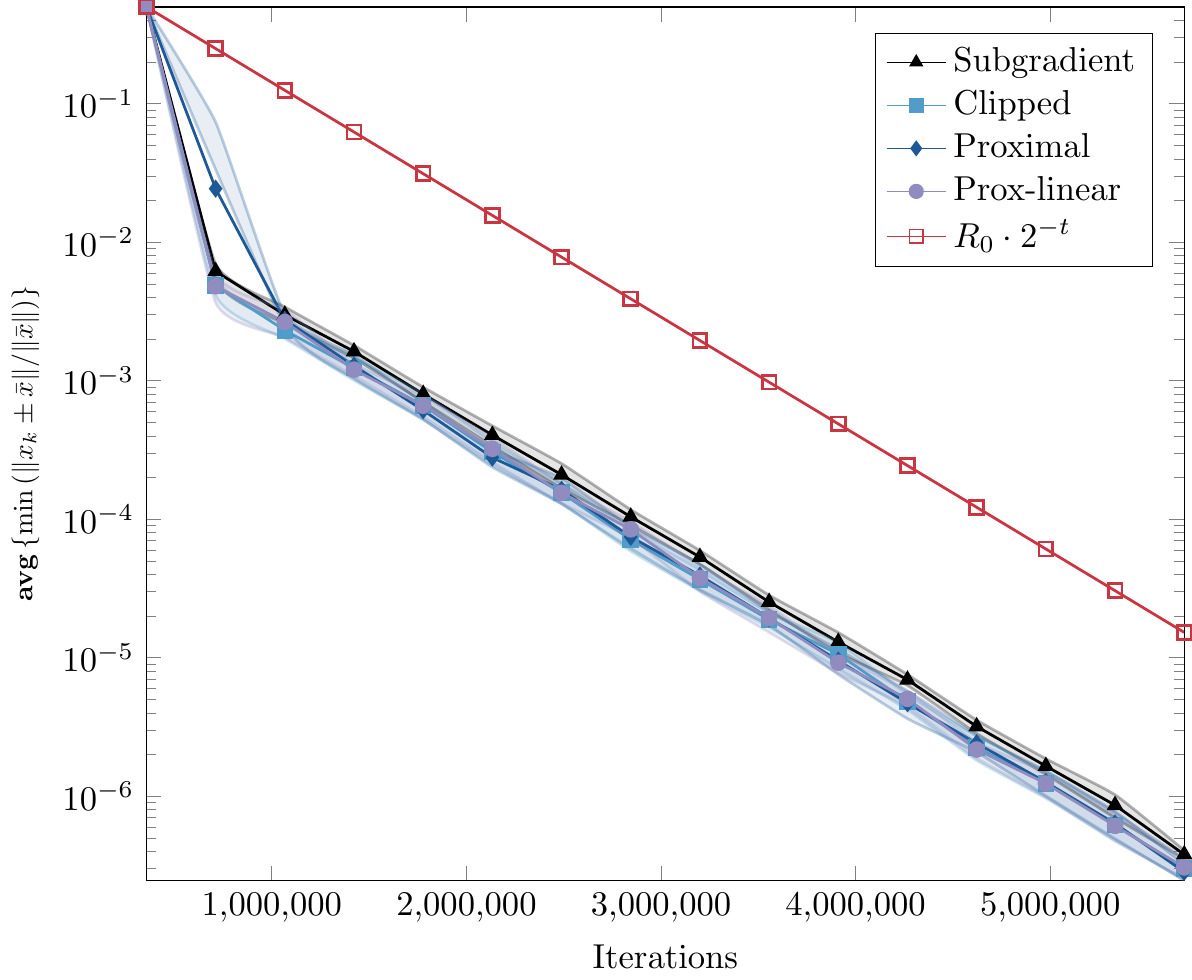}
	\end{minipage}
	\begin{minipage}{0.48 \textwidth}
	   \includegraphics[width=\linewidth]{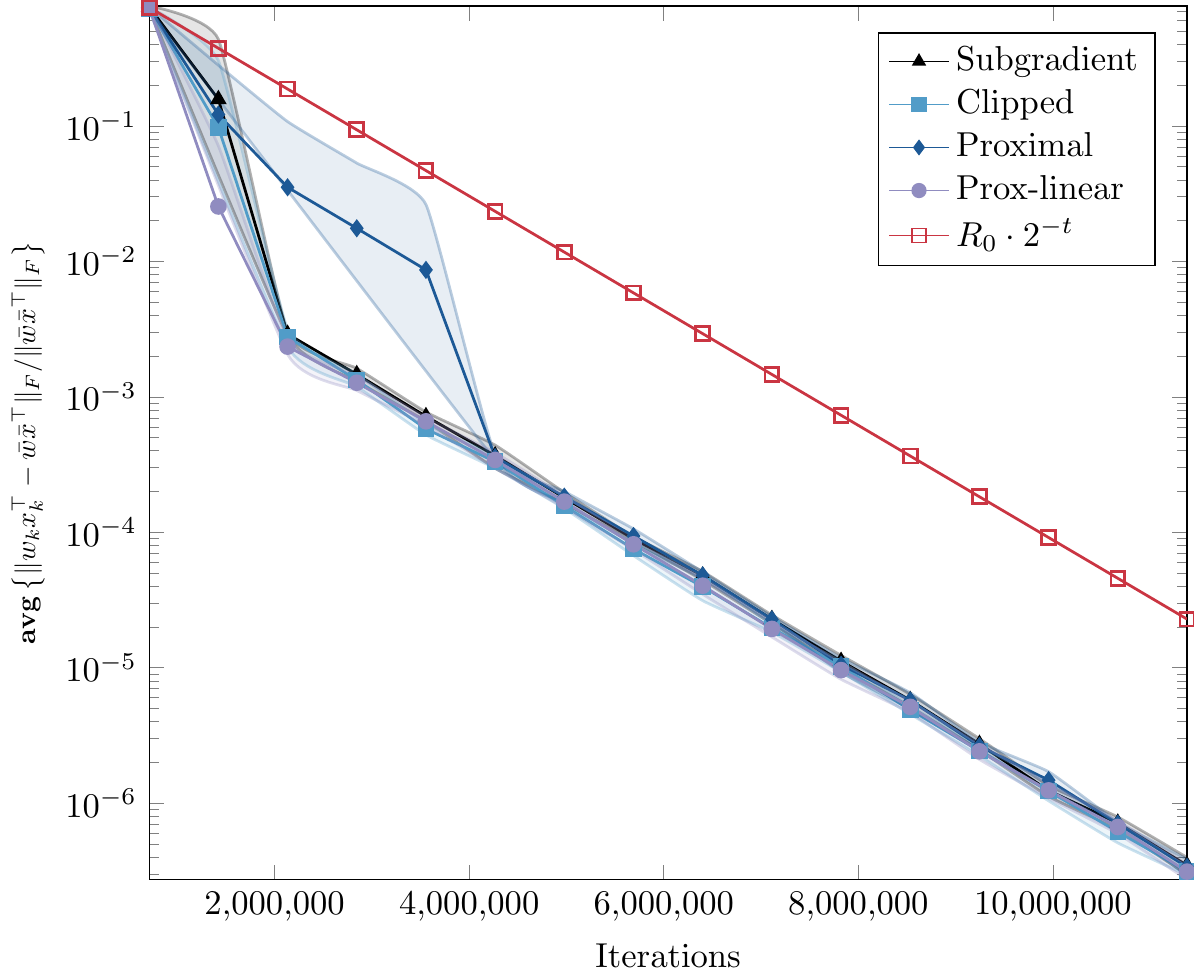}
	\end{minipage}
	\caption{Convergence behavior for $d = 100$ with finite
	sample size $m=8\cdot d$.  Phase Retrieval (left column), Blind Deconvolution (right column),   $\pfail = 0.0$ (top row), $\pfail =0.2$ (bottom row). Average over $10$ runs.}
	\label{fig:conv-finsample-pfail-2}
\end{figure}

\begin{figure}[h!]
	\begin{minipage}{0.48 \textwidth}
		\includegraphics[width=\linewidth]{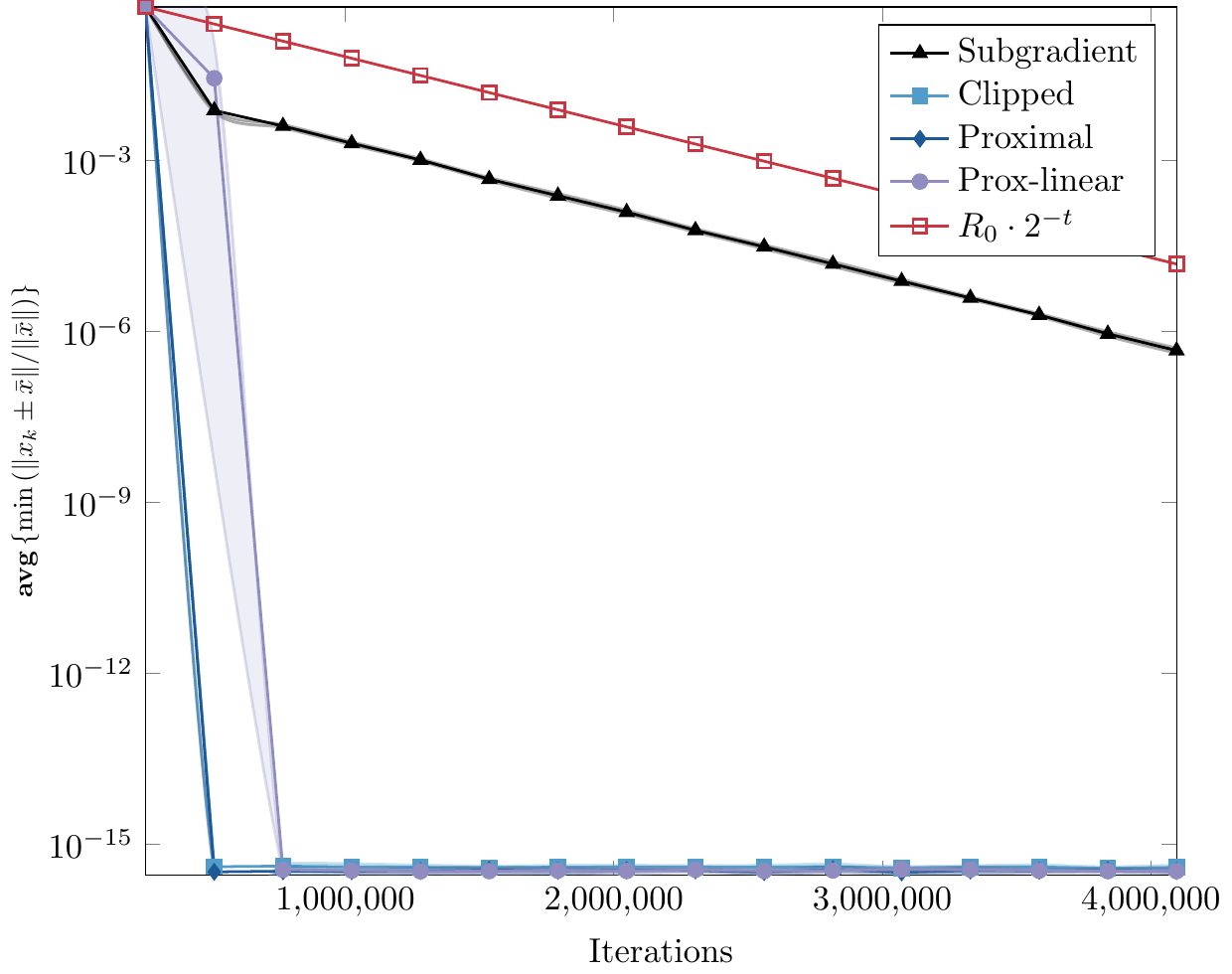}
	\end{minipage}
	\begin{minipage}{0.48 \textwidth}
		\includegraphics[width=\linewidth]{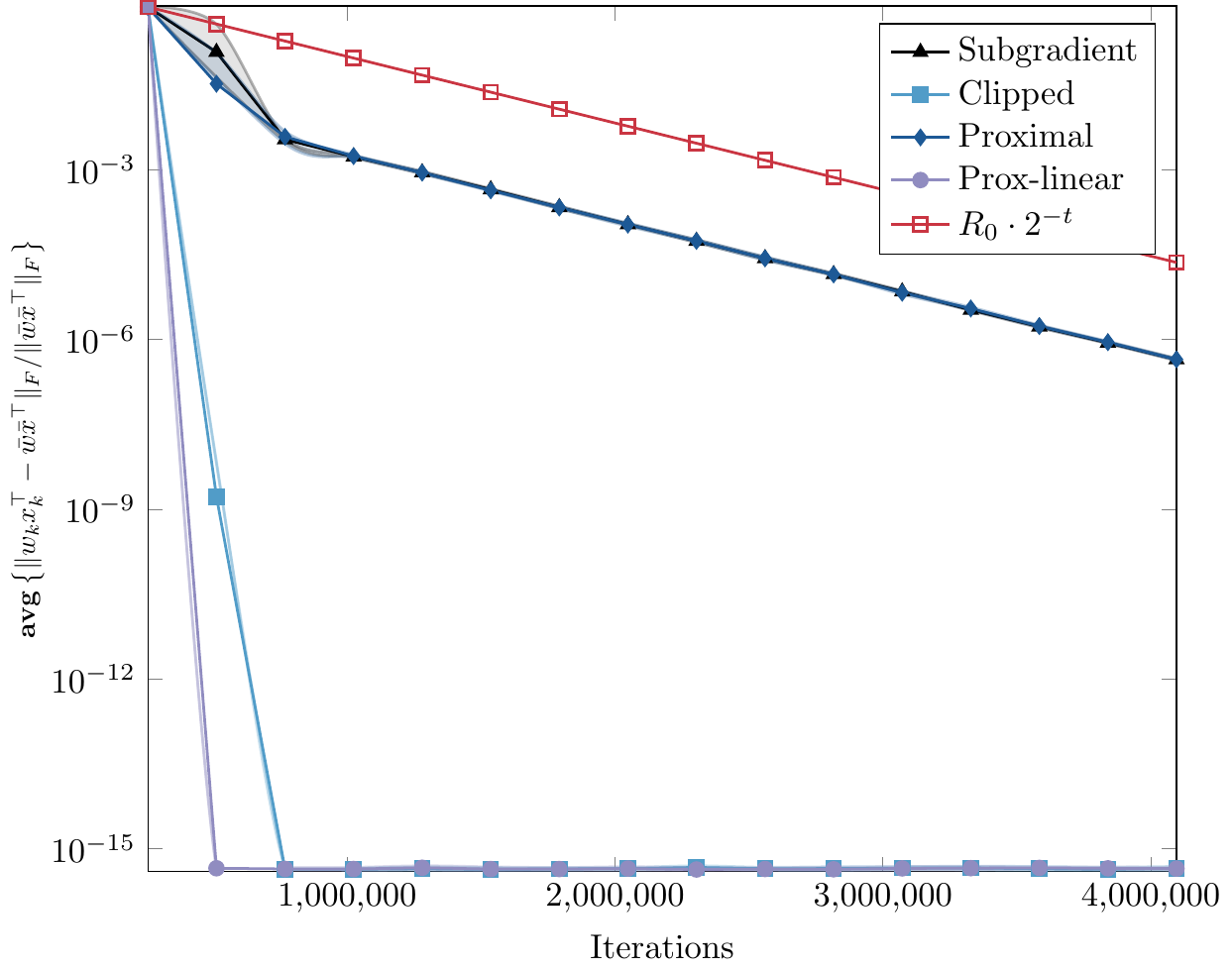}
	\end{minipage}

	\begin{minipage}{0.48 \textwidth}
		\includegraphics[width=\linewidth]{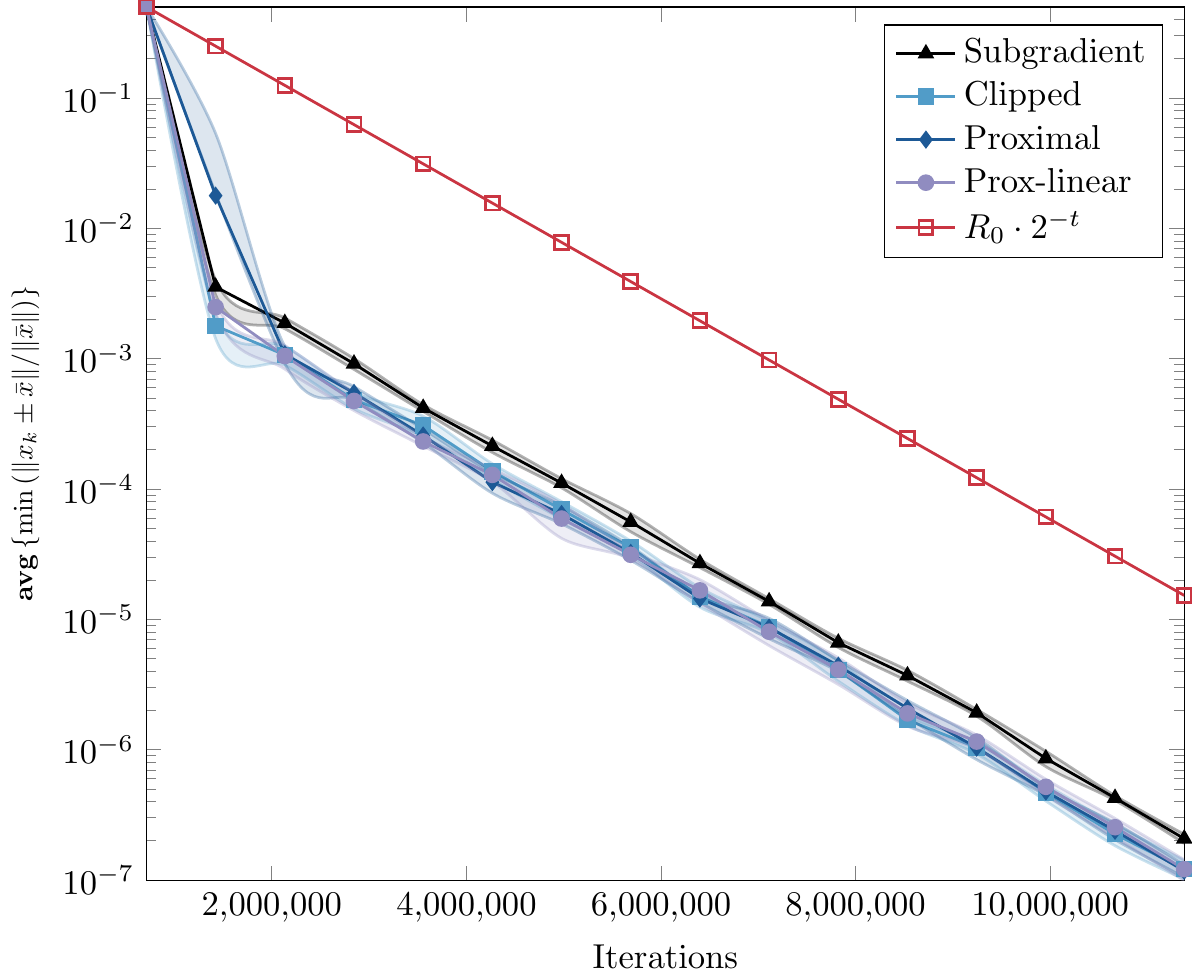}
	\end{minipage}
	\begin{minipage}{0.48 \textwidth}
		\includegraphics[width=\linewidth]{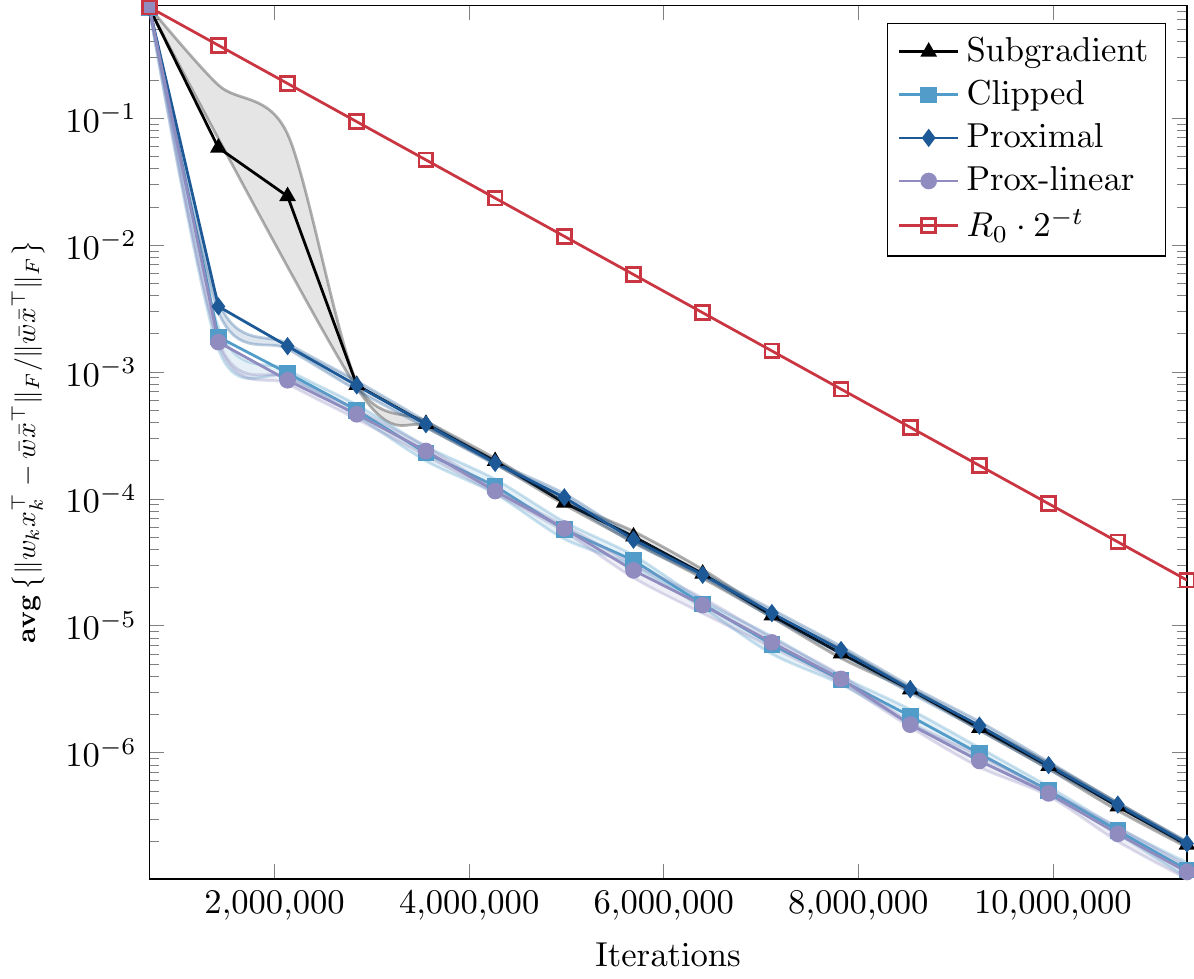}
	\end{minipage}
	\caption{Convergence behavior for $d = 100$ with streaming data.  Phase Retrieval (left column), Blind Deconvolution (right column),   $\pfail = 0.0$ (top row), $\pfail =0.2$ (bottom row). Average over $10$ runs.}
	\label{fig:conv-pfail-2}
\end{figure}

\subsection{Sensitivity to step size}
We next explore how Algorithm~\ref{alg:stoc_prox_outer} performs when $\alpha_0$ is misspecified. Throughout, we scale $\alpha_0$ by $\lambda := 2^p$ for integers $p $ between $-10$ and $10$. We run $25$ trials of the
algorithm and for each model and scalar $\lambda$, we report two different
metrics:
\begin{itemize}
\item We report the sample mean and standard deviation of the
number of ``inner'' loop iterations, or samples, needed to reach accuracy $\varepsilon = 10^{-5}$. We use the parameters of Section~\ref{sec:convergencebehav} to cap the number of total iterations by
$$
\frac{16}{(2-\gamma)^2} \left( \frac{ \lipsymb}{\delta_2 \mu}\right)^2 \left\lceil\log_2 \left(\frac{R_0}{\varepsilon}\right)^3\right\rceil
$$
as Theorem~\ref{thm:weak_setting} prescribes. This number is depicted as a dotted line in the figure.
\item We report the sample mean and standard deviation of the distance of the
final iterate to the solution set.
\end{itemize}

Figures~\ref{fig:stepsize-pr-clean} and \ref{fig:stepsize-bd-clean} show the
results for phase retrieval and blind deconvolution problems with $d = 100$ and $\pfail \in \{0, 0.2\}$. In these plots, solid lines with markers show the
mean behavior over $25$ runs, while the transparent overlays show one sample standard deviation above and below the mean. The plots show that Algorithm~\ref{alg:stoc_prox_outer} continues to perform as predicted by Theorem~\ref{thm:weak_setting} even if $\alpha_0$ is misspecified by a few orders of magnitude.

The prox-linear, proximal point, and clipped models perform similarly in all plots. As
reported in~\cite{asi2019importance}, the prox-linear and clipped methods
produce the same iterates. The iterates produced by the stochastic proximal point method and stochastic prox-linear are not identical, but they are
practically indistinguishable. This is due to two factors: the proximal and prox-linear models agree up to an error that increases quadratically as we move from the basepoint, and the proximal subproblems force iterates to remain near the basepoint.  Running the proximal point method for a much larger stepsize produces different iterates than the prox-linear method, though then the method fails to converge within the specified level of accuracy.


\begin{figure}[h!]
    \centering
    \begin{minipage}{0.48 \textwidth}
        \includegraphics[width=\linewidth]{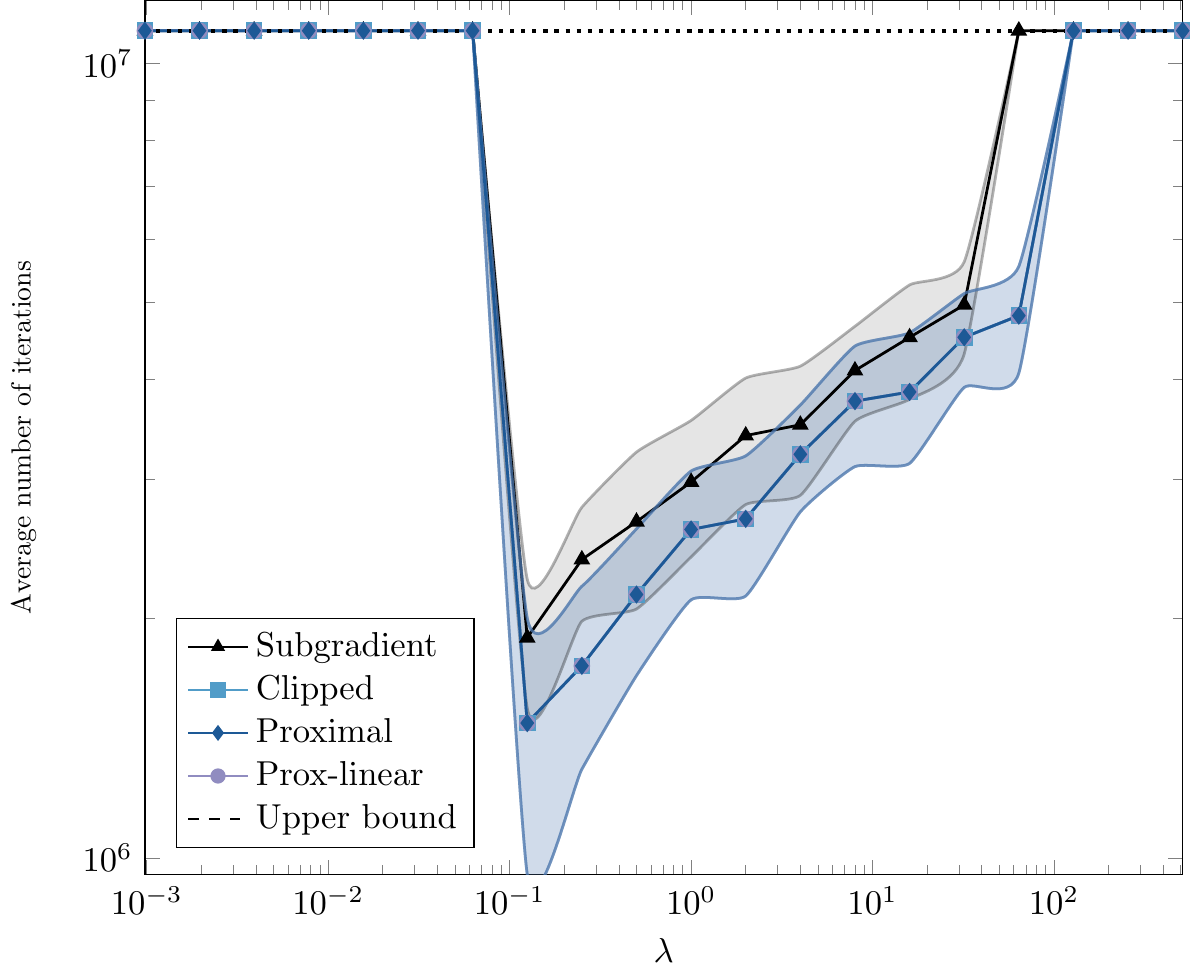}
	\end{minipage} \quad
	\begin{minipage}{0.48 \textwidth}
		\includegraphics[width=\linewidth]{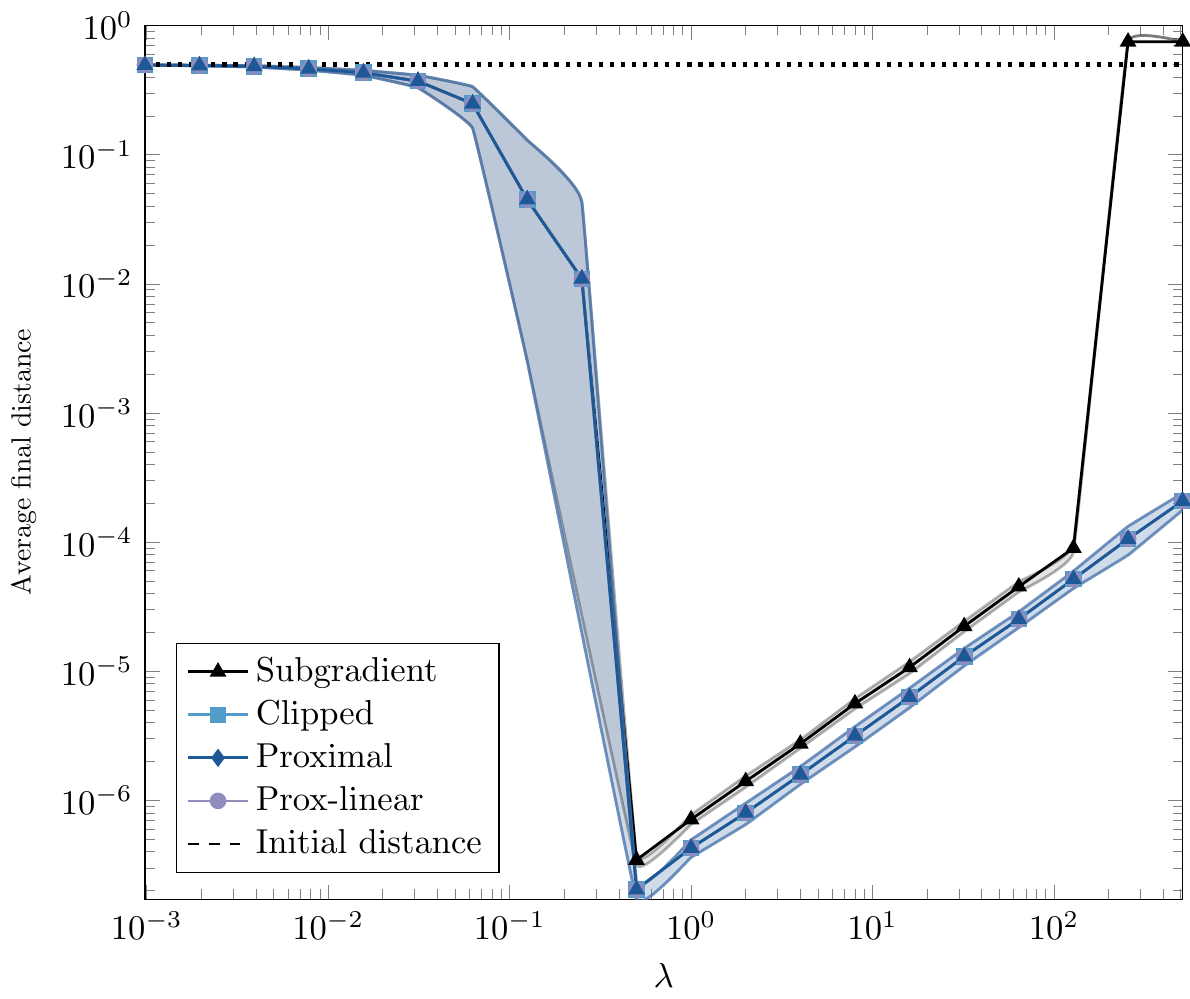}
    \end{minipage}
	\begin{minipage}{0.48 \textwidth}
		\includegraphics[width=\linewidth]{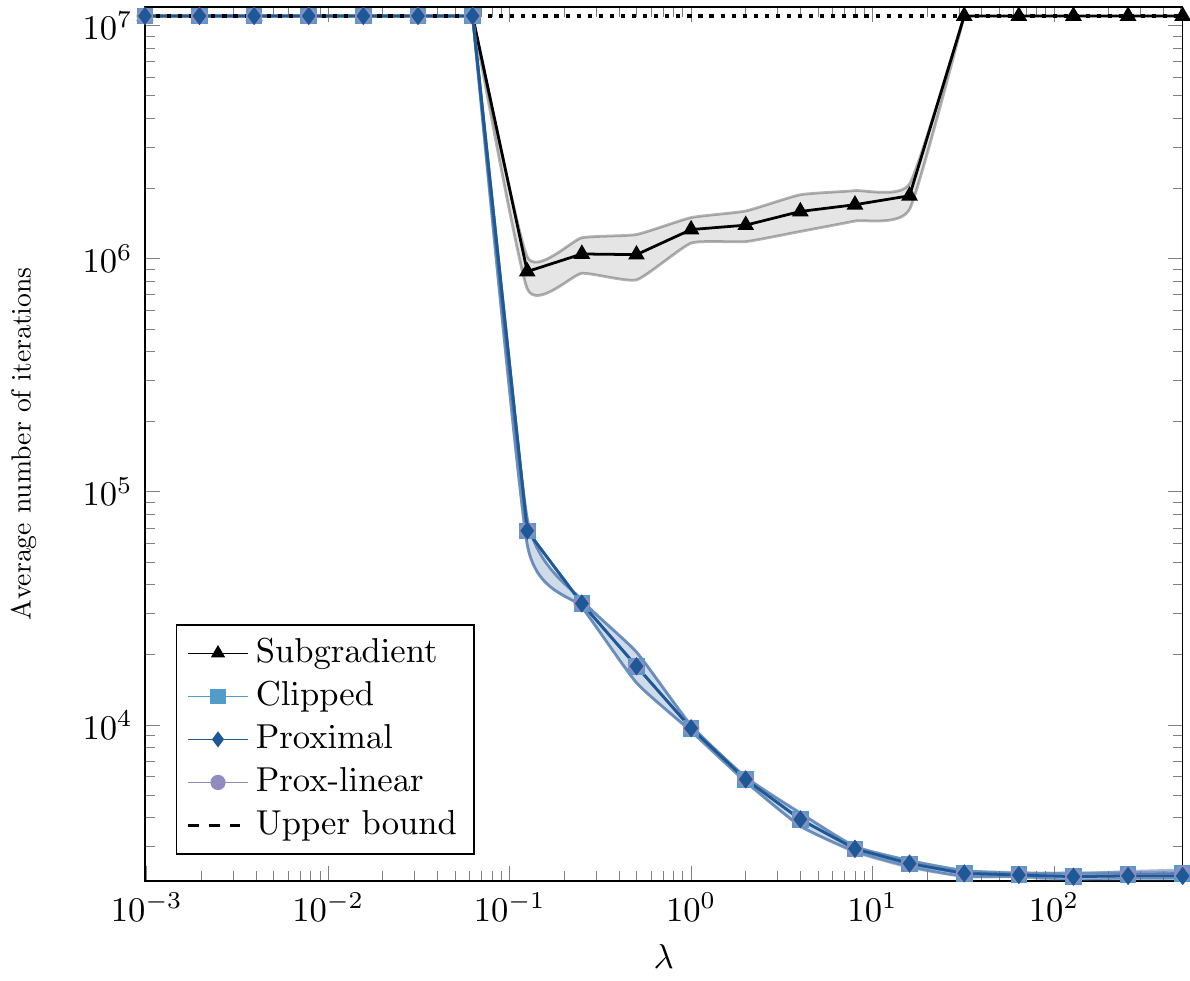}
	\end{minipage}
	\begin{minipage}{0.48 \textwidth}
		\includegraphics[width=\linewidth]{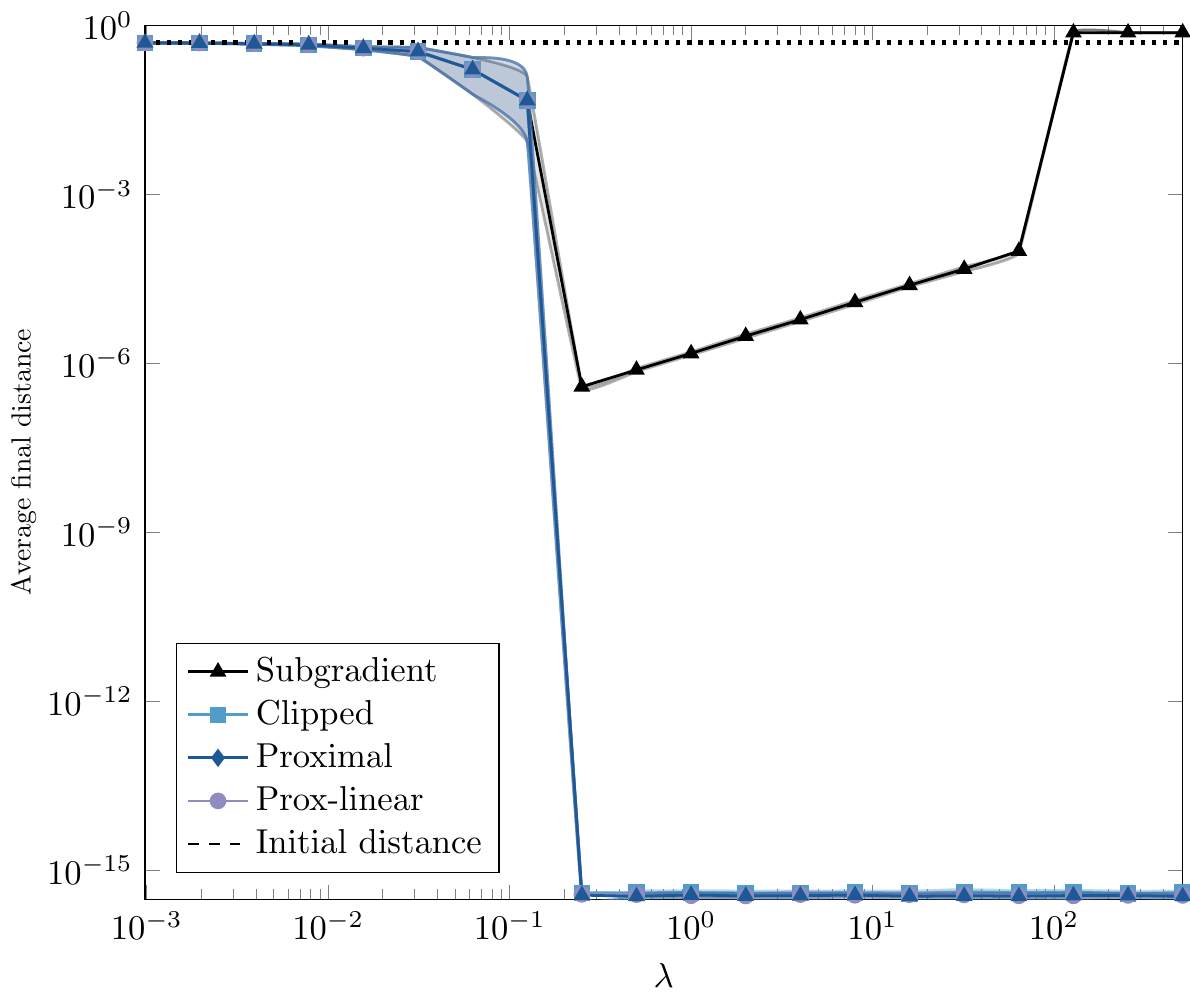}
	\end{minipage}
	\caption{Sensitivity to step size for the phase retrieval problem with $d = 100$, $\pfail
	= 0.2$ (top row), $\pfail
	= 0$ (bottom row). Left: average number of iterations to achieve distance $10^{-5}$.
Right: average final distance with a fixed computational budget.}
	\label{fig:stepsize-pr-clean}
\end{figure}

\begin{figure}[h!]
	\centering
	\begin{minipage}{0.48 \textwidth}
		\includegraphics[width=\linewidth]{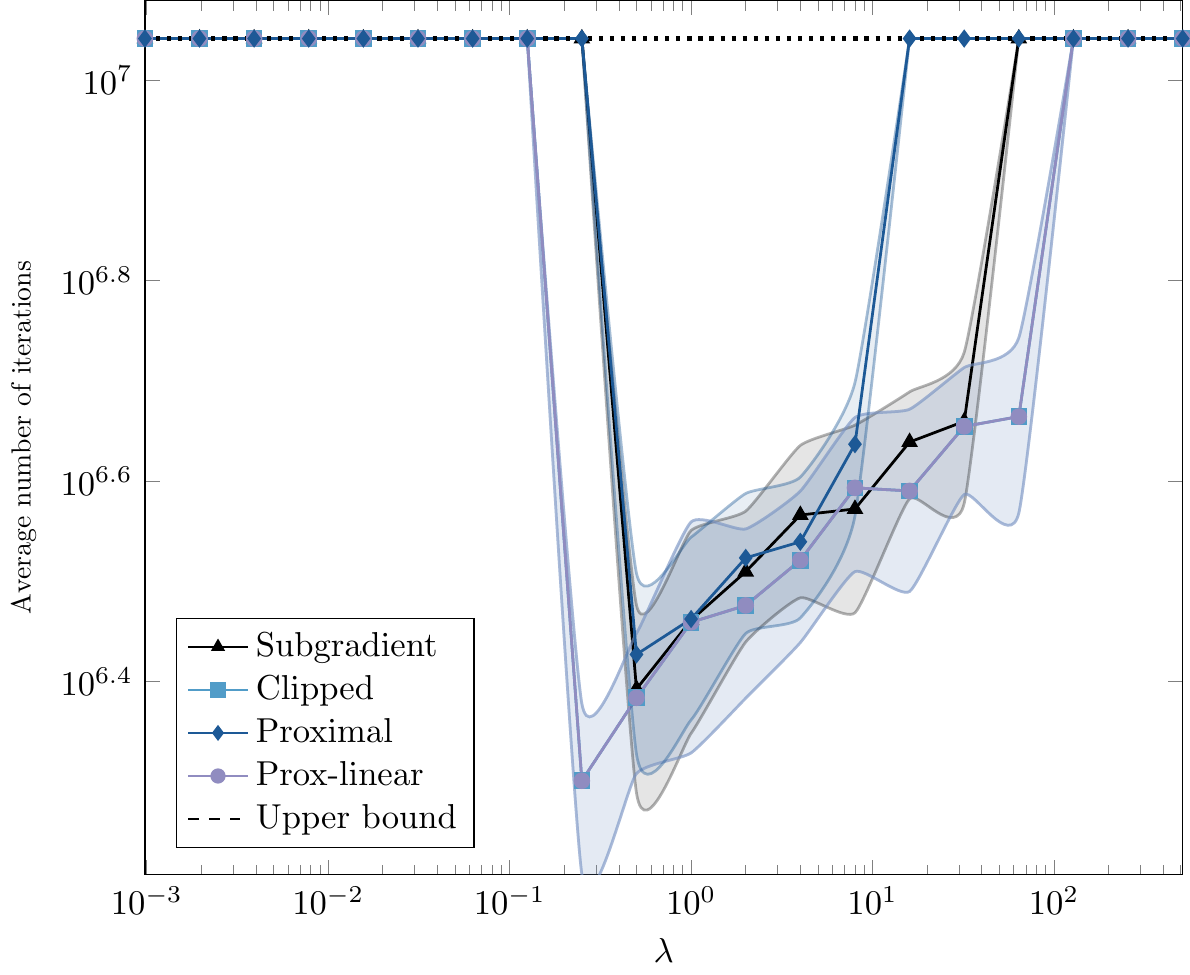}
	\end{minipage}
	\begin{minipage}{0.48 \textwidth}
		\includegraphics[width=\linewidth]{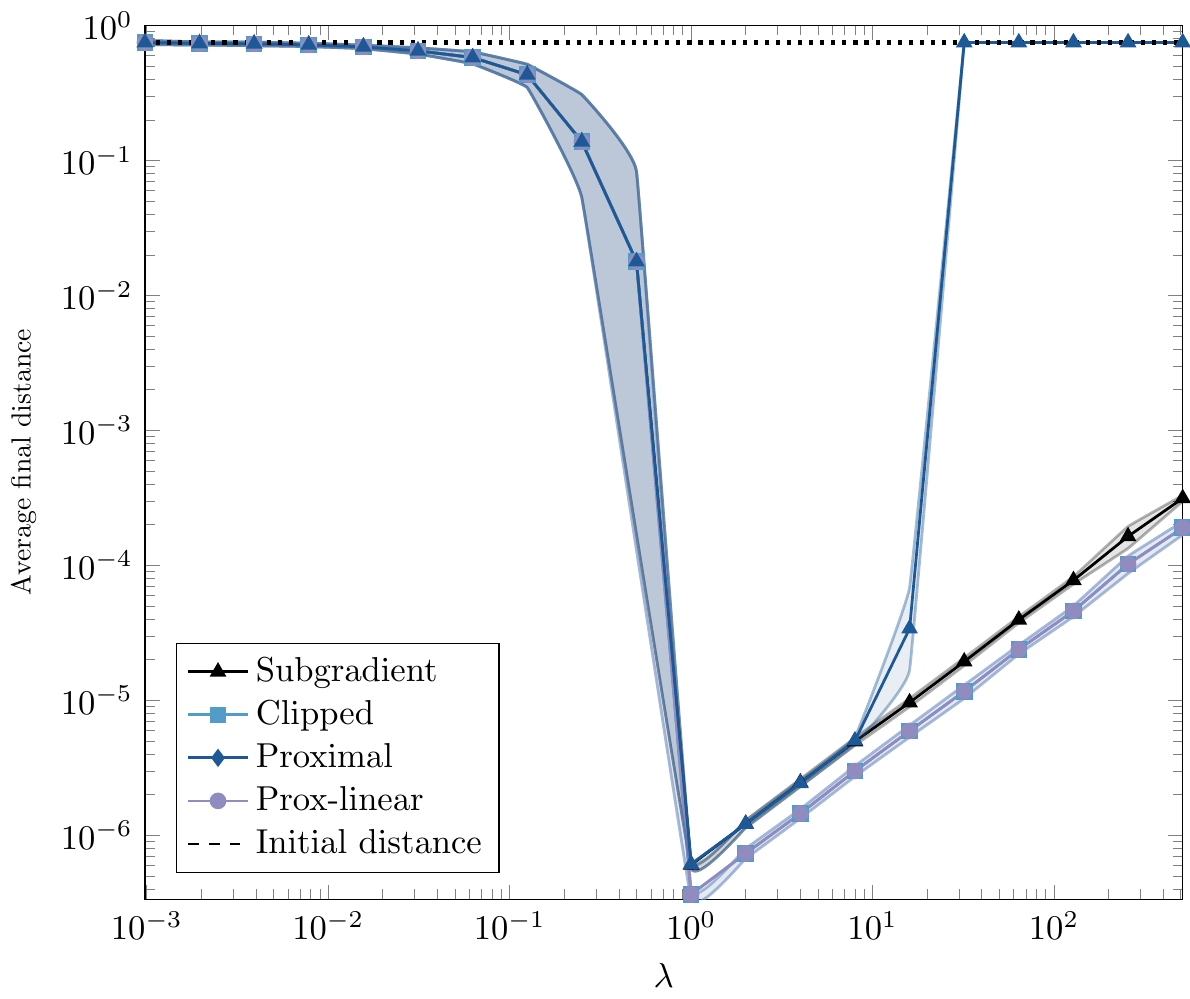}
	\end{minipage}
	\begin{minipage}{0.48 \textwidth}
		\includegraphics[width=\linewidth]{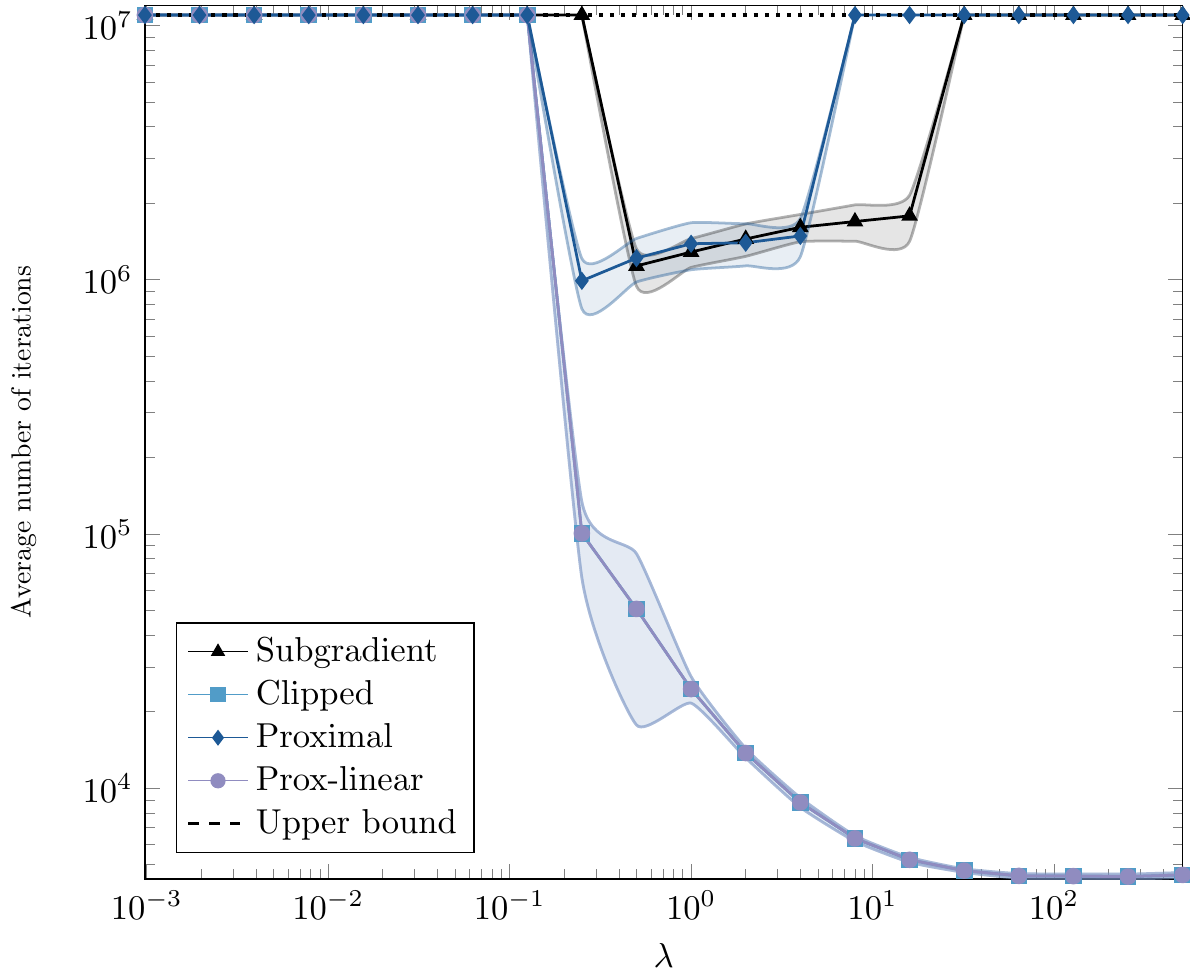}
	\end{minipage}
	\begin{minipage}{0.48 \textwidth}
		\includegraphics[width=\linewidth]{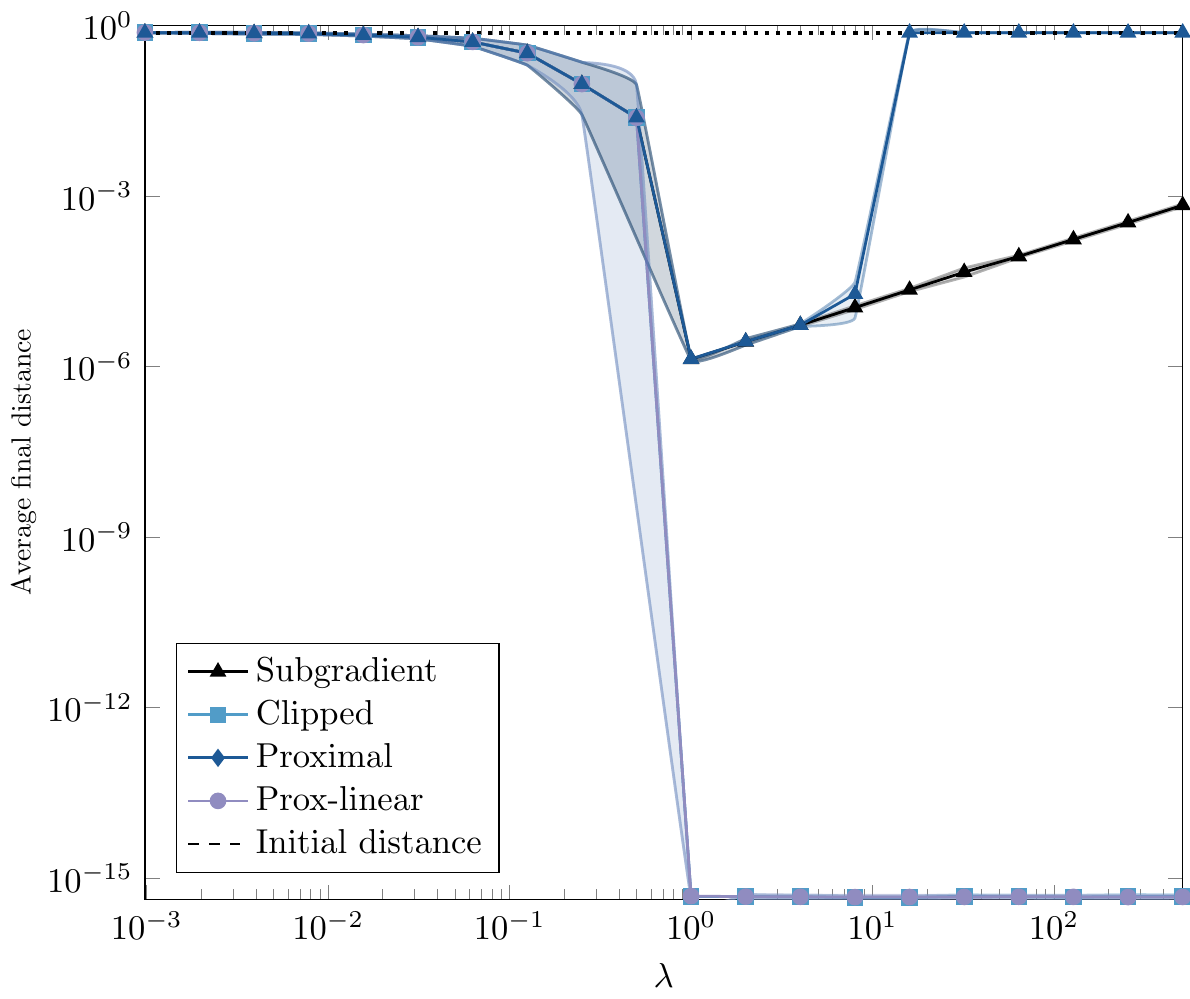}
	\end{minipage}
	\caption{Sensitivity to step size for the blind convolution problem with $d = 100$, $\pfail
	= 0.2$ (top row), $\pfail
	= 0$ (bottom row). Left: average number of iterations to achieve distance $10^{-5}$.
	Right: average final distance with a fixed computational budget.}
	\label{fig:stepsize-bd-clean}
\end{figure}

\subsection{Activity identification}
In this section, we demonstrate that Algorithm~\ref{alg:stoc_prox_outer} linearly converges to the active set of nonzero components in a sparse logistic regression problem. We model our experiment on~\cite[Section 6.2]{LeeWright12}. There, the authors find a sparse classifier for distinguishing between digits $6$ and $7$ from the MNIST dataset of handwritten digits~\cite{MNIST}. In this problem, we are given a set of $N = 12183$ samples $(x_i, y_i) \in \RR^d
\times \{-1, 1\}$, representing $28\times 28$ dimensional images of digits and their labels, and we seek a target vector $z := (w, b) \in \RR^{d+1}$, so that $w$ is sparse and $\sign(\ip{w, x_i} + b) = \sign(y_i)$ for most $i$. To find $z$, we minimize the function
$$\min_{z}~ \frac{1}{N}\sum_{i=1}^N f(z; i)+\tau\|w\|_1,$$
where each component is a logistic loss:
\begin{equation}
	f(z; i) \equiv f(w, b; i)
	:= \log\left(1 + \exp\left(-y_i(\ip{w, x_i} + b)\right) \right).
	\label{eq:logistic-loss}
\end{equation}
We let $(\tilde{w}, \tilde b)$ denote the minimizer of the logistic loss, which we find using the standard proximal gradient algorithm.  Given $\tilde{w}$, we denote its support set by $S:= \{ i \in [d] :
\abs{\tilde{w}_i} > \epsilon \}$, where $\epsilon$ accounts for the
numerical precision of the system.

Our goal is to converge linearly to the set
$$
\cS = \{ (w, b) \in \RR^{d+1} \mid w_{S^c} = 0\}.
$$
To that end, we will apply Algorithm~\ref{alg:stoc_prox_outer} with the stochastic proximal gradient model:
$$		f_{z}((y, a), i ) = f(z, i) + \ip{\nabla f(z, i), (y, a) - z}
		+ \tau \norm{y}_1.$$
Algorithm~\ref{alg:stoc_prox_outer} equipped with this model results in the
standard stochastic proximal gradient method. To apply the algorithm, we set
parameters using Theorem~\ref{thm:weak_setting}. We set $\gamma = 1$,
$\delta_2 = 1/\sqrt{10}$, $\varepsilon = 10^{-5}$.
We initialize $w_0 = 0, b_0 = 0$ and set
set $R_0 = \|(\tilde w, \tilde b) - (w_0, b_0)\|$. We estimate $\lipsymb$ by the formula
$
	\mathsf{L} := \sqrt{\frac{1}{m} \sum_{i=1}^m \norm{x_i}_2^2}.
$
Finally, we estimate $\mu$ by grid search over $(\tau, p)$ using the formula:
$\mu = \tau \cdot \sqrt{d} \cdot 2^{-p}$.

\subsubsection{Evaluation}
We compare the performance of Algorithm~\ref{alg:stoc_prox_outer} with the
Regularized Dual Averaging method (RDA)~\cite{nesterov2009primal,xiao2010dual}, which was shown to have favorable manifold identification properties in~\cite{LeeWright12}.
In our setting, the latter method solves the following subproblem:
\begin{align}
	(w_{t+1}, b_{t+1}) \in \argmin_{w, b} \left\{
		\left\langle \bar{g}_t, \begin{pmatrix} w \\ b \end{pmatrix}
		\right\rangle + \tau \norm{w}_1
		+ \frac{\gamma}{2 \sqrt{t}} \norm{\begin{pmatrix} w \\ b
		\end{pmatrix}}_2^2
	\right\},
	\label{eq:rda-update}
\end{align}
where $\bar{g}_t$ in~\eqref{eq:rda-update} is the running average over all
stochastic gradients $g_k := \nabla f(z; i_k)$ sampled up to step $t$, and
$\gamma$ is a tunable parameter which is again determined by a simple grid
search. Following the discussion in~\cite{LeeWright12}, RDA is initialized
at $(w_0, b_0) = 0$; therefore we choose the same initial point $(w_0, b_0)$ for both
methods.
In addition to RDA, we also performed a comparison with the standard stochastic proximal
gradient method, equipped with a range of polynomially decaying step sizes of the form
\[
	\lambda_k := c k^{-p}, \; p \in \{1/2, 2/3, 3/4, 1\}.
\]
We found that the stochastic proximal gradient method performed comparably with RDA in all metrics, and therefore chose to omit it below.

The convergence plots in Figure~\ref{fig:mnist-rmba-vs-rda} confirm that the iterates
of Algorithm~\ref{alg:stoc_prox_outer} converge to the set $\cS$ at a linear rate,
while the function values converge at a sublinear rate. In contrast, the
iterates generated by RDA converge sublinearly in both metrics.

\begin{figure}[ht]
	\centering
	\begin{minipage}{0.5\textwidth}
		\includegraphics[width=\linewidth]{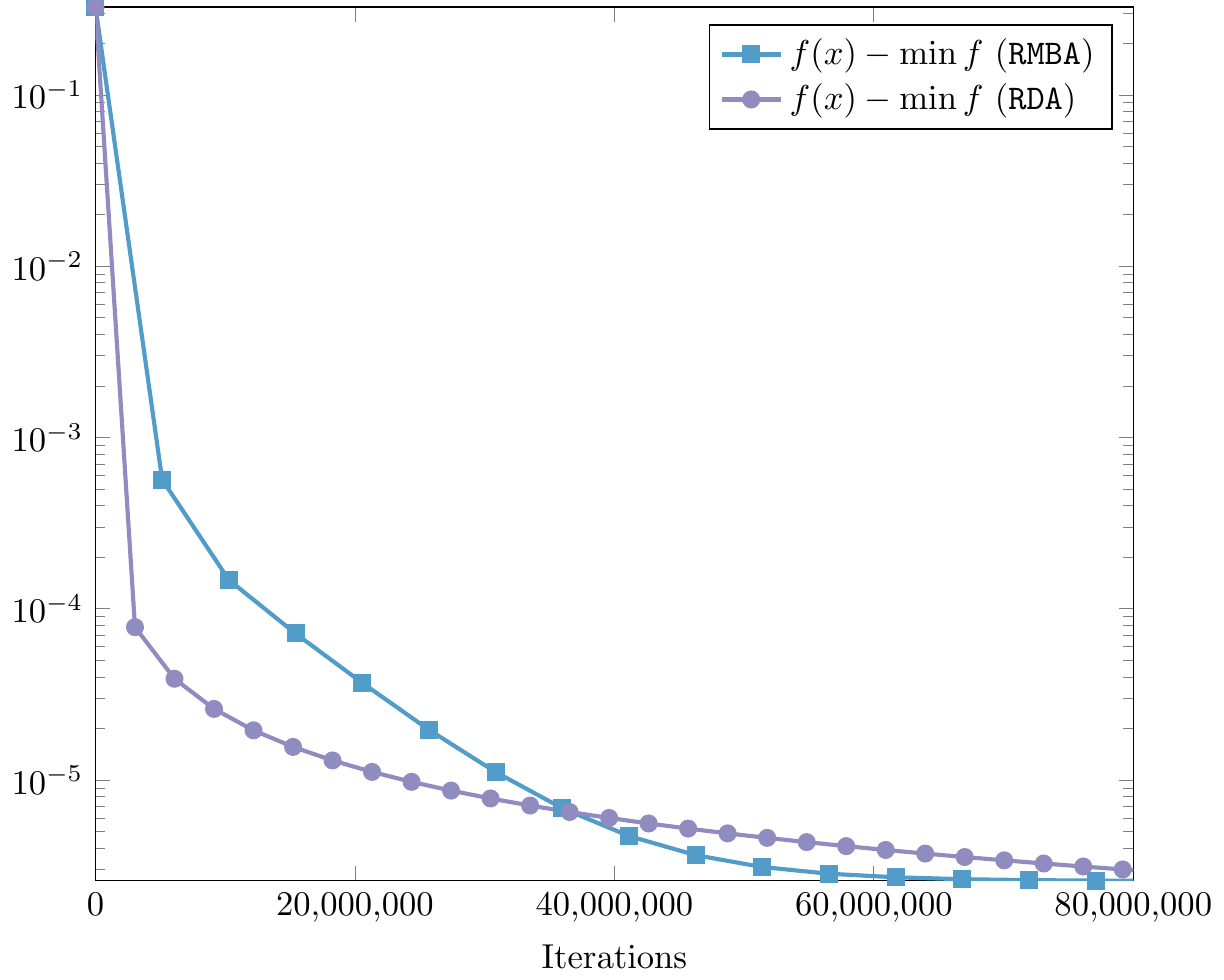}
	\end{minipage}~
	\begin{minipage}{0.5\textwidth}
		\includegraphics[width=\linewidth]{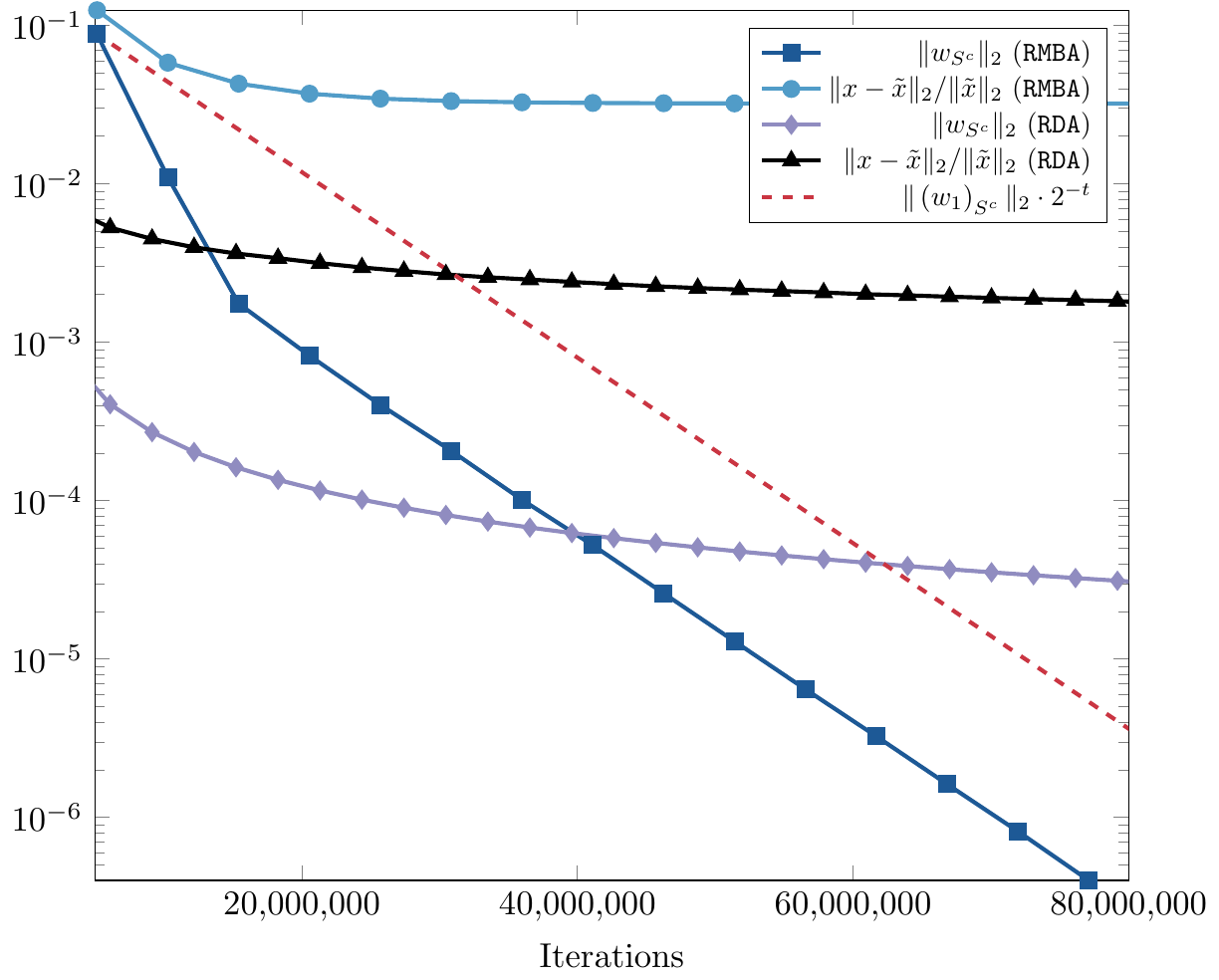}
	\end{minipage}
	\caption{Performance of RMBA vs. RDA on the sparse logistic regression
		problem. Left: function value gap. Right: distance to the support set
		and the solution found by the proximal gradient method.}
	\label{fig:mnist-rmba-vs-rda}
\end{figure}

			\bibliographystyle{plain}
	\bibliography{bibliography}
	
	\appendix

\section{Proofs from Section~\ref{sec:weaklyconvex}}

We will need the following elementary lemma.
\begin{lem}\label{lem:lipschitz}
	Suppose Assumptions~\ref{assumption:convex} and~\ref{assumption:lipschitz} hold.
	Fix an arbitrary $\gamma \in (0, 2)$ and consider two points $y \in \cT_\gamma$ and  $x \in \RR^d$. Then the estimate holds:
	$$
	f_y(y, \data) \leq f_y(x, \data) + L(y, \data) \|x - y\| .
	$$
\end{lem}
\begin{proof}

	Let $v \in \partial f_y(y, \data)$  be the minimal norm subgradient.  Then, by definition
	$
	f_y(y, \data) \leq f_y(x, \data) + \dotp{v, y-x} 
	$
	Applying~\eqref{eqn:lip_mod} and Cauchy-Schwarz completes the proof.
\end{proof}

\subsection{Proof of Lemma~\ref{lem:key_lem}}\label{sec:proof_key_lem}
Throughout the proof, we suppose that Assumptions~\allasumption hold. We let $\cF_\inneridx$ denote the $\sigma$-algebra generated by the history of the algorithm up to iteration $\inneridx$ and define the shorthand for conditional expectation $\EE_k\left[ \cdot \right] := \EE\left[ \cdot \mid \cF_k\right]$.
	Define the stopping time
	$$
	\tau := \min\{ \inneridx\geq 0 \mid y_\inneridx \notin \cT_{\gamma}\},
	$$
	and the sequence of events
	$$
	B : = \{ y_0 \in \cT_{\gamma\sqrt{\delta} } \} \qquad \text{ and } \qquad A_{\inneridx} := \{ \tau > \inneridx\} \cap B.
	$$
	Define also for all indices $k$, the quantity:
		$$
		D_\inneridx := \dist(y_\inneridx, \cS).
		$$
Recall that our goal is to lower bound $\PP(D_{K^*}\leq \varepsilon)$. To this end, we successively compute
\begin{align}
\PP\left(D_{\INNERIDX^\ast} \leq \varepsilon \right) &\geq\PP\left(D_{\INNERIDX^\ast} \leq \varepsilon \text{ and }  A_\INNERIDX \right)\notag\\
&=\PP\left(D_{\INNERIDX^\ast} \leq \varepsilon \mid A_\INNERIDX\right) \PP(A_\INNERIDX)\notag\\
&=\left(1-\PP\left(D_{\INNERIDX^\ast} \geq \varepsilon \mid A_\INNERIDX\right)\right)\PP(A_\INNERIDX)\notag\\
&\geq\left(1-\frac{\EE\left[D_{\INNERIDX^\ast} \mid A_\INNERIDX \right]}{\varepsilon}\right)\PP(A_\INNERIDX)\label{eqn:mark_append}\\
&=\PP(A_K)-\frac{\EE\left[D_{\INNERIDX^\ast} 1_{A_\INNERIDX} \right]}{\varepsilon}\notag
 ,
\end{align}
where \eqref{eqn:mark_append} follows from Markov's inequality. The result will now follow immediately from the following two propositions, which establish an upper bound on $\EE\left[D_{\INNERIDX^\ast} 1_{A_\INNERIDX} \right]$ and a lower bound on $\PP(A_K)$, respectively.
		We note that the first proposition is a quick modification of~\cite[Lemma 4.2]{davis2019stochastic}. We include a proof for completeness.
	\begin{proposition}\label{claim:exp_1}
The following bounds hold:
	\begin{align}
	\EE_\inneridx\left[D_{k+1}^21_{A_{\inneridx}}   \right]&\leq
	D_{\inneridx}^21_{A_{\inneridx}}  + \lipsymb^2\alpha^2 -(2-\gamma)\mu\alpha D_{\inneridx}1_{A_{\inneridx}}, \label{eq:descent_ak}	\\
 \EE\left[D_{\INNERIDX^\ast}1_{A_{\INNERIDX}}\right]  &\leq \frac{ \delta\left(\frac{\gamma\mu}{\quadapprox}\right)^2 + (\INNERIDX+1)\alpha^2\lipsymb^2}{(2-\gamma)\mu(\INNERIDX+1) \alpha}.\label{eqn:prop_1recur}
\end{align}
	\end{proposition}
\begin{proof}
	The loss function $y\mapsto f_{y_\inneridx}(y, \data_\inneridx)+\frac{1}{2\alpha}\|y-y_\inneridx\|^2$ is strongly convex on $\cX$ with constant $1/\alpha$ and $y_{\inneridx+1}$ is its minimizer. Hence for any $y\in \cX$, the inequality holds:
	\begin{align*}
	\left(f_{y_\inneridx}(y, \data_\inneridx)+ \tfrac{1}{2\alpha}\|y-y_\inneridx\|^2\right) \geq \Big(&f_{y_\inneridx}( y_{\inneridx+1}, \data_\inneridx) +\tfrac{1}{2\alpha} \|y_{\inneridx+1}-y_\inneridx\|^2\Big) + \tfrac{1}{2\alpha}\|y - y_{\inneridx+1}\|^2.
	\end{align*}
	Rearranging and taking expectations, we successively deduce that provided $y_\inneridx \in \cT_\gamma$, we have
	\begin{align}
	&\frac{1}{2\alpha}\cdot\EE_\inneridx\left[\|y - y_{\inneridx+1}\|^2 + \|y_{\inneridx+1}-y_\inneridx\|^2 - \|y - y_{\inneridx}\|^2\right] \notag\\
	&\leq \EE_\inneridx\left[ f_{y_\inneridx}(y, \data_\inneridx)  - f_{y_\inneridx}( y_{\inneridx+1}, \data_\inneridx) \right]\notag\\
	&\leq \EE_\inneridx\left[ f_{y_\inneridx}(y, \data_\inneridx) -  f_{y_\inneridx}( y_{\inneridx}, \data_\inneridx) +  L(y_k, \data)\|y_{\inneridx+1}-y_\inneridx\|\right] \label{eqn:lip_modelcvx0}\\
	&\leq f(y) -  f( y_{\inneridx}) +  \frac{\quadapprox}{2}\|y - y_{\inneridx}\|^2 + \sup_{w \in \cT_2}\sqrt{\EE_{\data}[L(w, \data)^2]}\cdot\sqrt{\EE_\inneridx[\|y_{\inneridx+1}-y_\inneridx\|^2]},\label{eqn:CS0} 
\end{align}
where \eqref{eqn:lip_modelcvx0} follows from Lemma~\ref{lem:lipschitz} while inequality \eqref{eqn:CS0} follows from Cauchy-Schwarz and Assumption~\ref{assumption:oneside}.

	Define $c:=\sqrt{\EE_\inneridx[\| y_{\inneridx+1}-y_\inneridx\|^2]}$ and notice $c\geq \EE_\inneridx\|y_\inneridx - y_{\inneridx+1}\|$. Then, rearranging \eqref{eqn:CS0}, we immediately deduce that if $y_k \in \cT_\gamma$, we have
	\begin{align*}
	  \frac{1}{2\alpha}\EE_\inneridx\left[\inf_{y \in \proj_{\cS}(y_\inneridx)}\| y - y_{\inneridx+1}\|^2  \right] &\leq  \frac{1}{2\alpha}\inf_{y \in \proj_{\cS}(y_\inneridx)}\EE_\inneridx\left[\| y - y_{\inneridx+1}\|^2  \right] \\
	  &\leq \frac{\alpha^{-1}+\quadapprox}{2}D_\inneridx^2  - \frac{c^2}{2\alpha} +  \lipsymb c - (f(y_{\inneridx}) -  \inf_{y \in \proj_{\cS}(y_\inneridx)} f(y))\\
	 &\leq \frac{\alpha^{-1}+\quadapprox}{2}D_\inneridx^2  - \frac{c^2}{2\alpha} +  \lipsymb c - \mu D_{\inneridx}\\
	&\leq \frac{\alpha^{-1}+\quadapprox}{2}D_\inneridx^2  + \frac{\alpha\lipsymb^2 }{2} - \mu D_{\inneridx}.
	\end{align*}
	where the third inequality follows from assumption~\ref{assumption:sharp},  and the fourth inequality follows by maximizing the right-hand-side in $c \in\R$. Then, dividing through by $\frac{1}{2\alpha}$ and multiplying by $1_{A_\inneridx}$, we arrive at
\begin{align}
\EE_\inneridx\left[D_{k+1}^21_{A_{\inneridx+1}}   \right] \leq \EE_\inneridx\left[D_{k+1}^21_{A_{\inneridx}}   \right] &\leq \left((1+\alpha\quadapprox )D_{\inneridx}^2  + \alpha^2\lipsymb^2 - 2\mu\alpha D_{\inneridx}\right) 1_{A_{\inneridx}} \notag\\
&\leq  \left(D_{\inneridx}^2  + \alpha^2\lipsymb^2 - \alpha\left(2\mu - \quadapprox D_{\inneridx} \right) D_{\inneridx}\right) 1_{A_{\inneridx}} \notag\\
&\leq D_{\inneridx}^21_{A_{\inneridx}}  + \alpha^2\lipsymb^2 - \alpha(2-\gamma)\mu D_{\inneridx}1_{A_{\inneridx}},\notag 
\end{align}
where the first inequality follows since $A_{\inneridx+1} \subseteq A_{\inneridx}$, the second inequality follows since $A_{\inneridx}$ is $\cF_k$ measurable, and the fourth inequality follows since on the event $A_k$, we have $y_k \in \cT_{\gamma}$. This completes the proof of \eqref{eq:descent_ak}. Next, applying the law of total expectation, we obtain
$$
\EE\left[D_{k+1}^21_{A_{\inneridx+1}}   \right]  \leq \EE\left[D_{\inneridx}^21_{A_\inneridx}\right]  + \alpha^2\lipsymb^2 - (2-\gamma)\alpha\mu\EE\left[D_{\inneridx}1_{A_\inneridx}\right].
$$
Iterating the inequality and rearranging, we deduce
$$\sum_{k=0}^K (2-\gamma)\alpha\mu\EE\left[D_{\inneridx}1_{A_\inneridx}\right]\leq (K+1)\alpha^2\lipsymb^2 +\EE\left[D_{0}^21_{A_0}\right]$$
Dividing through by $(K+1) (2-\gamma)\alpha\mu$, we recognize the left-hand side as $\EE\left[D_{K^\ast}1_{A_{\INNERIDX^\ast}}\right]$, and therefore
\begin{align*}
 \EE\left[D_{K^\ast}1_{A_{\INNERIDX^\ast}}\right] \leq \frac{ \delta\left(\frac{\gamma\mu}{\quadapprox}\right)^2 + (\INNERIDX+1)\alpha^2\lipsymb^2}{(2-\gamma)(K+1)\mu\alpha}.
\end{align*}
Finally, note $\EE\left[D_{\INNERIDX^\ast}1_{A_{\INNERIDX}}\right] \leq \EE\left[D_{K^\ast}1_{A_{\INNERIDX^\ast}}\right]$
since $A_\INNERIDX \subseteq A_{\INNERIDX^\ast}$. This completes the proof of the proposition.
\end{proof}

Now we will estimate the probability of the event $A_{\INNERIDX}$.
\begin{proposition}\label{claim:event_bound}
The estimate holds:
\begin{align*}
\PP(A_\INNERIDX) \geq  \PP(B) - \left( \delta+ \left(\frac{\quadapprox}{\gamma\mu}\right)^2K \alpha^2\lipsymb^2\right).
\end{align*}
\end{proposition}
\begin{proof}
Observe the decomposition
$$
\PP(B) = \PP(B \text{ and } \tau \leq \INNERIDX)  + \PP(B \text{ and } \tau > \INNERIDX) = \PP(B \text{ and } \tau \leq \INNERIDX) + \PP(A_\INNERIDX) ,
$$
and therefore
\begin{equation}\label{eqn:comp_prob1}
\PP(A_\INNERIDX)\geq \PP(B)-\PP(B \text{ and } \tau \leq \INNERIDX).
\end{equation}
We aim to upper bound $\PP(B \text{ and } \tau \leq \INNERIDX)$. To this end,
let $Y_\inneridx = D^2_{\inneridx \wedge \tau}1_{B}$ denote the stopped process.
We successively compute
\begin{equation}\label{eqn:comp_prob2}
\PP(B \text{ and } \tau \leq \INNERIDX)  = \PP\left(Y_{\tau} 1_{\tau \leq \INNERIDX} > \left(\frac{\gamma\mu}{\quadapprox }\right)^2\right) \leq \frac{ \EE\left[ Y_\tau 1_{\tau \leq K} \right]}{\left(\frac{\gamma\mu}{\quadapprox }\right)^2},
\end{equation}
where the last estimate uses Markov's inequality. Next, observe
\begin{equation}\label{eqn:comp_prob3}
\EE\left[ Y_\tau 1_{\tau \leq \INNERIDX} \right]\leq \EE\left[ Y_\INNERIDX 1_{\tau > \INNERIDX} \right] + \EE\left[ Y_\INNERIDX 1_{\tau \leq \INNERIDX} \right]=\EE\left[ Y_\INNERIDX \right].
\end{equation}
We next upper bound $\EE[Y_K]$. To this end, observe
\begin{align}
\EE_{\inneridx} \left[ Y_{\inneridx+1} \right] &= \EE_{\inneridx} \left[ Y_{\inneridx+1} 1_{\tau \leq \inneridx}\right] + \EE_{\inneridx} \left[ Y_{\inneridx+1}1_{\tau > \inneridx}\right]\notag\\
&= Y_{\inneridx} 1_{\tau \leq \inneridx} + \EE_\inneridx\left[D^2_{\inneridx+1}1_{A_{\inneridx}}   \right]\notag\\
&\leq Y_{\inneridx} 1_{\tau \leq \inneridx} + D_{\inneridx}^21_{A_{\inneridx}}  +\alpha^2\lipsymb^2 - (2-\gamma)\alpha\mu D_{\inneridx}1_{A_{\inneridx}}   \label{eqn:gettingannoyed}\\
&\leq Y_{\inneridx} 1_{\tau \leq \inneridx} +D_{\inneridx\wedge\tau}^21_{\tau > \inneridx}1_{B} + \alpha^2\lipsymb^2 = Y_\inneridx + \alpha^2\lipsymb^2,\notag
\end{align}
where \eqref{eqn:gettingannoyed} follows from~\eqref{eq:descent_ak}. We now use the law of total expectation to iterate the above inequality:
\begin{equation}\label{eqn:comp_prob4}
\EE\left[ Y_\INNERIDX \right] \leq \EE\left[ Y_0 \right] + \INNERIDX \alpha^2\lipsymb^2 \leq D_0^21_{B} + \INNERIDX \alpha^2\lipsymb^2 \leq\delta \left(\frac{\gamma\mu}{\quadapprox }\right)^2 + \INNERIDX  \alpha^2\lipsymb^2.
\end{equation}
Combining the estimates \eqref{eqn:comp_prob1}, \eqref{eqn:comp_prob2}, \eqref{eqn:comp_prob3}, and \eqref{eqn:comp_prob4}  completes the proof.

\end{proof}

\section{Proofs from Section~\ref{sec:highprob}}

\subsection{Proof of Lemma~\ref{lem:prox_subproblem}}\label{sec:aux_lemma}
Fix $\gamma \in (0, 2)$ and a point $y \in \cT_\gamma$. Recall that $f$ is $\quadapprox$-weakly convex on an open convex set containing $\cX$ \cite[Lemma 4.1]{davis2019stochastic}. Consequently, the proximal subproblem \eqref{eqn:prox_subprob:strong} is $(\rho-\quadapprox)$-strongly convex.
Before proving the remaining portion of the lemma, we first show that subgradients of the extended valued function $f + \delta_\cX$ are bounded below. This was essentially already observed in \cite[Lemma 2.1]{MR3869491}. We provide a quick proof for completeness.

\begin{lem}\label{lem:appeal:sharp_subgrad_out}
The estimate:
$$
\dist(0, \partial f(x) + N_\cX(x)) \geq \left(1- \tfrac{\gamma}{2}\right)\mu \qquad \textrm{holds for all }x \in \cT_\gamma\backslash \cX^\ast.
$$
\end{lem}
\begin{proof}
Fix any $x \in \cT_\gamma\backslash \cX^\ast$ and $v \in \partial f(x) + N_{\cX}(x)$, and let $\bar x \in \proj_{\cX^\ast}(x)$. We successively compute
$$
\mu\cdot \dist(x, \cX^\ast) \leq f(x) - \inf_{\cX} f \leq \dotp{v,x - \bar x} + \frac{ \quadapprox }{2}\dist(x, \cX^\ast)^2,
$$
where the first inequality follows from sharpness~\ref{assumption:sharp}, and the second from  weak convexity of $f$ and convexity of $\cX$.
Rearranging and using the Cauchy-Schwarz inequality, we deduce
$$
\left( \mu - \frac{\quadapprox }{2}\dist(x, \cX^\ast)\right) \dist(x, \cX^\ast)  \leq \dist(x, \cX^\ast)\|v\|.
$$
Dividing both sides by $\dist(x, \cX^\ast)$ yields the result.
\end{proof}

Now, let $\bar y$ be any minimizer of the proximal problem \eqref{eqn:prox_subprob:strong} and suppose $ \rho < \left(\tfrac{2-\gamma}{2\gamma}\right)\quadapprox$. Clearly, to establish Lemma~\ref{lem:prox_subproblem}, it suffices to argue the inclusion $\bar y\in \cX^{*}$.
To this end, choose any $\hat y \in \proj_{\cX^\ast}(y)$. Observe
$$
\frac{\rho}{2}\|\bar y-y\|^2 \leq  f(\hat y) - f(\bar y) + \frac{\rho}{2}\|\hat y - y\|^2\leq \frac{\rho}{2}\dist^2(y,\cX^*) \leq \frac{\rho \gamma^2 \mu^2}{2\quadapprox^2},$$
where the first inequality follows from the definition of $\bar y$, the second uses the definition of $\hat y$, and the third follows from the assumption $y\in \cT_{\gamma}$.
Thus, we deduce
 $\|\bar y-y \| \leq \frac{\gamma \mu}{\quadapprox }$.
Consequently, using sharpness we conclude
$$
\mu \dist(\bar y, \cX^\ast) \leq f(\bar y) - f(\hat  y)  \leq \frac{\rho}{2} \|\hat y - y\|^2 \leq \frac{\rho}{2}  \frac{\gamma^2 \mu^2}{\quadapprox ^2} \leq \frac{\gamma \mu^2}{2\quadapprox }
$$
where the last inequality follows from the assumption $\rho < \left(\tfrac{2-\gamma}{2\gamma}\right)\quadapprox$.
Consequently, $\bar y$ lies in the tube $\cT_\gamma$.
Now, define $v := \rho(y - \bar y) \in \partial f(\bar y) + N_{\cX}(\bar y)$. Appealing to Lemma~\ref{lem:appeal:sharp_subgrad_out}, we deduce  in the case $\bar y \notin \cX^\ast$, the contradiction $(1-\tfrac{\gamma}{2}) \mu/ \rho \leq \|v\|/\rho = \|y - \bar y\| \leq \frac{\gamma \mu}{\quadapprox }$. Therefore, $\bar y$ lies in $ \cX^\ast$, as we had to show.

\subsection{Proof of Lemma~\ref{lem:key_lem_sc}}\label{sec:proof_key_lem_sc}
As in the proof of Lemma~\ref{lem:key_lem}, we let $\cF_\inneridx$ denote the $\sigma$-algebra generated by the history of the algorithm up to iteration $\inneridx$ and define the shorthand for conditional expectation $\EE_k\left[ \cdot \right] := \EE\left[ \cdot \mid \cF_k\right]$. Define the stopping time
	$$
	\tau := \min\{ \inneridx \mid y_\inneridx \notin \cT_{\gamma}\},
	$$
	 the sequence of events
	$
	A_{\inneridx} := \{ \tau > \inneridx\},
	$
	and the quantities
	$$D_\inneridx := \|y_k - \bar y_0\| \qquad \textrm{and}\qquad E_\inneridx := \frac{\alpha^{-1} + \quadapprox}{2}D_\inneridx^2 +  \frac{\rho}{2}\|y_{\inneridx} - y_0\|^2.$$
	Note that by Lemma~\ref{lem:prox_subproblem}, the
	inclusion $\bar y_0 \in \proj_{\cX^\ast}(y_0)$ holds. We will use this observation throughout.

We begin with the estimate
\begin{align*}
\PP\left(D_{\INNERIDX^\ast}^2 \leq \varepsilon^2 \right)\ &\geq \PP\left(D_{\INNERIDX^\ast}^2 \leq \varepsilon^2 \mid A_\INNERIDX\right) P(A_\INNERIDX) \\
&=\left( 1 - \PP\left(D_{\INNERIDX^\ast}^2 \geq \varepsilon^2 \mid A_\INNERIDX\right)
\right) \PP(A_\INNERIDX)\\
&\geq \left( 1 - \frac{\EE\left[D_{\INNERIDX^\ast}^2 \mid  A_\INNERIDX\right]}{\varepsilon^2}
\right) \PP(A_\INNERIDX)\\
&=\left( 1 - \frac{\EE\left[D_{\INNERIDX^\ast}^2 1_{A_\INNERIDX}\right]}{\varepsilon^2\PP(A_\INNERIDX)}
\right) \PP(A_\INNERIDX)\\
&=\PP(A_\INNERIDX)-\frac{\EE\left[D_{\INNERIDX^\ast}^2 1_{A_\INNERIDX}\right]}{\varepsilon^2}
\end{align*}
%
%
The result will now follow immediately from the following two propositions, which establish an upper bound on $\EE\left[D_{\INNERIDX^\ast}^2 1_{A_\INNERIDX} \right]$ and a lower bound on $\PP(A_K)$, respectively.
We note that the first proposition is a quick modification of~\cite[Lemma 4.2]{davis2019stochastic}. We include a proof for completeness.
	\begin{proposition}\label{claim:exp_1_sc}
	Define $\nu := \rho - \quadapprox$. Then the following bounds hold:
	\begin{align}
	\EE_\inneridx \left[ E_{\inneridx+1} 1_{A_k}\right] &\leq E_k1_{A_k}  - \frac{\strongconvex}{2} D_\inneridx^21_{A_k} + \frac{\lipsymb^2 \alpha}{2 }. \label{eq:E_kdescent}\\
 \EE\left[D^2_{\INNERIDX^\ast}1_{A_{\INNERIDX}}\right]&\leq  \frac{(K+1) \lipsymb^2 \alpha + (\alpha^{-1} + \eta) \cdot\initquallem \left(\frac{\gamma \mu}{\quadapprox}\right)^2}{\nu(\INNERIDX + 1)}.
\end{align}
	\end{proposition}
\begin{proof}	Define the function $g(y)  := f(y) + \frac{\rho}{2} \|y- y_0\|^2$ and notice that $g$ is strongly convex with parameter $\nu$.
	Observe also that the loss function $y\mapsto f_{y_\inneridx}(y, \data_\inneridx)+\frac{1}{2\alpha}\|y-y_\inneridx\|^2 + \frac{\rho}{2}\|y - y_0\|^2$ is strongly convex on $\cX$ with constant $\alpha^{-1} + \rho $ and $y_{\inneridx+1}$ is its minimizer. Hence for any $y\in \cX$, the inequality holds:
	\begin{align*}
	\left(f_{y_\inneridx}(y, \data_\inneridx) +  \tfrac{1}{2\alpha}\|y-y_\inneridx\|^2+\tfrac{\rho}{2}\|y - y_0\|^2\right) \geq \Big(&f_{y_k}(y_{\inneridx +1}, \data_k)  +\tfrac{1}{2\alpha} \|y_{\inneridx+1}-y_\inneridx\|^2+\tfrac{\rho}{2}\|y_{k+1} - y_0\|^2\Big) \\
	&\hspace{20pt}+ \tfrac{\alpha^{-1} + \rho}{2}\|y - y_{\inneridx+1}\|^2.
	\end{align*}
	Rearranging and taking expectations we successively deduce that if $y_\inneridx \in \cT_\gamma$, then
	\begin{align}
	&\EE_\inneridx\left[\frac{\alpha^{-1} + \rho}{2}\|y - y_{\inneridx+1}\|^2 +\frac{1}{2\alpha} \|y_{\inneridx+1}-y_\inneridx\|^2 - \frac{1}{2\alpha}\|y - y_{\inneridx}\|^2\right] \notag\\
	&\leq \EE_\inneridx\left[ f_{y_\inneridx}(y, \data_\inneridx) + \frac{\rho}{2}\|y - y_0\|^2  - (f_{y_\inneridx}( y_{\inneridx+1}, \data_\inneridx) + \frac{\rho}{2}\|y_{\inneridx + 1} - y_0\|^2) \right]\notag\\
	&\leq \EE_\inneridx\left[  f_{y_\inneridx}(y, \data_\inneridx) + \frac{\rho}{2} \|y- y_0\|^2- ( f_{y_\inneridx}( y_{\inneridx}, \data_\inneridx) + \frac{\rho}{2} \|y_{\inneridx + 1} - y_0\|^2) \right] \notag\\
	&\hspace{20pt} +  \EE_k\left[L(y_k,\data_\inneridx)\|y_{\inneridx+1}-y_\inneridx\|\right]\label{eqn:lip_modelcvx}\\
	&\leq \EE_\inneridx\left[f(y) + \frac{\rho}{2} \|y- y_0\|^2- ( f(y_\inneridx) + \frac{\rho}{2} \|y_{\inneridx} - y_0\|^2) \right] + \EE_\inneridx\left[\frac{\rho}{2} \|y_{\inneridx} - y_0\|^2 - \frac{\rho}{2} \|y_{\inneridx + 1} - y_0\|^2 \right] \notag \\
	&\hspace{20pt} + \EE_k\left[\frac{\quadapprox}{2}\|y - y_{\inneridx}\|^2 +  L(y_k, \data_\inneridx)\|y_{\inneridx+1}-y_\inneridx\|\right] \notag \\
	&\leq g(y) -  g( y_{\inneridx}) +  \frac{\quadapprox}{2}\|y - y_{\inneridx}\|^2 + \sup_{w \in \cT_2}\sqrt{\EE_{\data}[L(w, \data)^2]}\cdot\sqrt{\EE_\inneridx[\|y_{\inneridx+1}-x_\inneridx\|^2]} \notag\\
	&\hspace{20pt}   +\EE_\inneridx\left[\frac{\rho}{2} \|y_{\inneridx} - y_0\|^2 - \frac{\rho}{2} \|y_{\inneridx + 1} - y_0\|^2 \right] \label{eqn:CS}
\end{align}
where \eqref{eqn:lip_modelcvx} follows from Lemma~\ref{lem:lipschitz}, while inequality \eqref{eqn:CS} follows from Cauchy-Schwarz and Assumption~\ref{assumption:oneside}.

	Define $c:=\sqrt{\EE_\inneridx[\| y_{\inneridx+1}-y_\inneridx\|^2]}$ and notice $c\geq \EE_\inneridx\|y_\inneridx - y_{\inneridx+1}\|$.
	Thus, letting $y  = \bar y_0$ and rearranging \eqref{eqn:CS}, 
we immediately deduce that if $y_\inneridx \in \cT_\gamma$, we have
	\begin{align*}
	\EE_\inneridx\left[\frac{\alpha^{-1} + \quadapprox }{2}D_{\inneridx + 1}^2 + \frac{\rho}{2} \|y_{\inneridx + 1} - y_0\|^2  \right]   &\leq \frac{\alpha^{-1} + \eta}{2}D_{\inneridx}^2 + \frac{\rho}{2}\|y_{\inneridx} - y_0\|^2 - \frac{ c^2}{2\alpha} +  \lipsymb c - (g(y_{\inneridx}) -  g(\bar y_0))  \\
	 &\leq \frac{\alpha^{-1}+\eta}{2}D_{\inneridx}^2+  \frac{\rho}{2}\|y_{\inneridx} - y_0\|^2  + \frac{\lipsymb^2 \alpha}{2 } - \frac{\nu}{2} D_k^2,
	\end{align*}
	where the second inequality follows from strong convexity of $g$  and by maximizing the right-hand-side in $c \in\R$.
	Thus, multiplying through by $1_{A_k}$, we deduce that
	\begin{align*}
	\EE_\inneridx \left[ E_{\inneridx+1} 1_{A_k}\right] &\leq E_k1_{A_k}  - \frac{\strongconvex}{2} D_\inneridx^21_{A_k} + \frac{\lipsymb^2 \alpha}{2 }. 
	\end{align*}
which proves~\eqref{eq:E_kdescent}. Iterating~\eqref{eq:E_kdescent}, using the tower rule, and rearranging, we deduce
$$
\frac{\nu}{2}\sum_{k=0}^K \EE\left[D_\inneridx^21_{A_{\inneridx}}\right]\leq (K+1)\frac{\lipsymb^2 \alpha}{2 } +E_0 \leq (K+1)\frac{\lipsymb^2 \alpha}{2 } + \frac{\alpha^{-1} + \eta}{2}\cdot\initquallem \left(\frac{\gamma \mu}{\quadapprox}\right)^2,
$$
where the last inequality follows from Lemma~\ref{lem:prox_subproblem} and the assumption $y_0 \in \cT_{\gamma\sqrt{\initquallem}}$.
Dividing through by $\frac{\nu}{2}(\INNERIDX + 1)$, we deduce
$$
 \EE\left[D_{K^\ast}^21_{A_{\INNERIDX^\ast}}\right] \leq \frac{(K+1) \lipsymb^2 \alpha + (\alpha^{-1} + \eta) \cdot\initquallem \left(\frac{\gamma \mu}{\quadapprox}\right)^2}{\nu(\INNERIDX + 1)}.
$$
Finally, note $\EE\left[D_{\INNERIDX^\ast}^21_{A_{\INNERIDX}}\right] \leq \EE\left[D_{K^\ast}^21_{A_{\INNERIDX^\ast}}\right]$
since $A_\INNERIDX \subseteq A_{\INNERIDX^\ast}$. This completes the proof of the proposition.
\end{proof}

Now we estimate the probability of the event $A_{\INNERIDX}$.
\begin{proposition}\label{claim:event_bound_sc}
The estimate holds:
\begin{align*}
\PP(A_\INNERIDX) \geq 1 - \initquallem -   \left(\frac{\quadapprox}{\gamma\mu}\right)^2\INNERIDX\lipsymb^2 \alpha^2.
\end{align*}
\end{proposition}
\begin{proof}
	Define $r := 2(\alpha^{-1} + \quadapprox)^{-1}$, let $\overline Y_{\inneridx} = \dist^2(y_{\inneridx\wedge \tau}, \cX^\ast)$,  and let $Y_\inneridx = rE_{\inneridx \wedge \tau}$ denote the stopped process.
We now estimate
\begin{equation}\label{eqn:highprolem1}
\PP(\tau \leq \INNERIDX ) = \PP\left(\overline Y_{\tau} 1_{\tau \leq \INNERIDX} > \left(\frac{\gamma\mu}{\quadapprox}\right)^2\right) \leq \frac{ \EE\left[ \overline Y_\tau 1_{\tau \leq \INNERIDX} \right]}{\left(\frac{\gamma\mu}{\quadapprox }\right)^2}.
\end{equation}
Next we upper bound the right-hand-side:
\begin{equation}\label{eqn:highprolem2}
\EE\left[ \overline Y_\tau 1_{\tau \leq \INNERIDX} \right] \leq \EE\left[ Y_\tau 1_{\tau \leq \INNERIDX} \right]\leq \EE\left[ Y_\INNERIDX 1_{\tau > \INNERIDX} \right] + \EE\left[ Y_\INNERIDX 1_{\tau \leq \INNERIDX} \right]= \EE\left[ Y_\INNERIDX \right],
\end{equation}
where the first inequality follows from the bound $\overline Y_k \leq Y_k$. Next, observe
\begin{align*}
\EE_{\inneridx} \left[ Y_{\inneridx+1} \right] &= \EE_{\inneridx} \left[ Y_{\inneridx+1} 1_{\tau \leq \inneridx}\right] + \EE_{\inneridx} \left[ Y_{\inneridx+1}1_{\tau > \inneridx}\right]\\
&= Y_{\inneridx} 1_{\tau \leq \inneridx} + \EE_\inneridx\left[rE_{\inneridx+1}1_{A_{\inneridx}}   \right]\\
&\leq Y_{\inneridx} 1_{\tau \leq \inneridx} + rE_{\inneridx}1_{A_{\inneridx}} -  \frac{\strongconvex}{2} D_\inneridx^21_{A_{\inneridx}}  + \frac{\lipsymb^2 \alpha}{2 }   \\
&\leq Y_{\inneridx} 1_{\tau \leq \inneridx} + rE_\inneridx1_{\tau > \inneridx} + r\frac{\lipsymb^2\alpha }{2 } = Y_\inneridx + \frac{\lipsymb^2\alpha  }{(\alpha^{-1} + \quadapprox)},
\end{align*}
where the first inequality follows from~\eqref{eq:E_kdescent}. We now use the law of total expectation to iterate the above inequality:
\begin{equation}\label{eqn:highprolem3}
\EE\left[ Y_\INNERIDX \right] \leq \EE\left[ Y_0 \right] + \frac{K\lipsymb^2\alpha  }{(\alpha^{-1} + \quadapprox)} \leq D_0^2 + \INNERIDX\lipsymb^2\alpha^2 \leq\initquallem \left(\frac{\gamma\mu}{\quadapprox }\right)^2 + \INNERIDX\lipsymb^2 \alpha^2,
\end{equation}
where the second inequality follows from the equality  $rE_{0} = D_0^2 = \dist^2(y_0, \cX^\ast)$. Combining \eqref{eqn:highprolem1}, \eqref{eqn:highprolem2}, and \eqref{eqn:highprolem3} completes the proof.
\end{proof}

\subsection{The ensemble method}

\begin{lem}[Ensemble method]\label{lem:ensemble}
	Let $\{x_i\}_{i=1}^m$ be independent random vectors in $\R^d$. Suppose that for $i = 1, \ldots, m$, the estimate holds:
	$$\PP(\|x_i-\bar x\|\leq \varepsilon)\geq p,$$
	where $p\in (\tfrac{1}{2},1)$ and $\varepsilon>0$ are some real numbers and $\bar x\in \R^d$ is a vector. Then with probability at least $1-\exp(-\frac{1}{2p}m(p-\frac{1}{2})^2)$, there exists an index $i^*$ satisfying
	\begin{equation}\label{eqn:index_want}
	|B_{2\varepsilon}(x_{i^*})\cap \{x_{i}\}_{i=1}^m|> \frac{m}{2}
	\end{equation}
	and for any index $i^*$ satisfying \eqref{eqn:index_want} it must be that $\|x_{i^*}-\bar x\|\leq 3\varepsilon$.
\end{lem}
\begin{proof}
	By Chernoff's bound, with probability at least
	$1-\exp(-\frac{1}{2p}m(p-\frac{1}{2})^2)$, the estimate holds:
	$$\left|\{i:\|x_i-\bar x\|\leq \varepsilon\}\right|>\frac{m}{2}.$$
	In particular, there exists an index $i^*$ satisfying \eqref{eqn:index_want}
	Fix such an index $i^*$.
Clearly, there must exist another index $j$ satisfying $x_j\in B_{\varepsilon}(\bar x)\cap  B_{2\varepsilon}(x_{i^*})$. We therefore conlude $\|x_{i^*}-\bar x\|\leq \|x_{i^*}-x_j\|+\|x_j-\bar x\|\leq 3\varepsilon$. This completes the proof.
\end{proof}

\section{Proofs from Section~\ref{sec:examples}}

\subsection{Proof of Theorem~\ref{thm:phase_params}}\label{appendix:thm:phase_params}
The equality $\cX^*=\{\pm \bar x\}$ and sharpness follows along similar lines as in \cite[Proposition 4]{duchi_ruan_PR} and~\cite[Lemma B.8]{davis2017nonsmooth}. We sketch a quick argument for completeness.
Fix $x \in \RR^d$ throughout the proof.
 Let $\hat f(x, z) = | (a^Tx)^2 - (a^T\bar x)^2|$ denote the ``outlier-free'' loss function and set $\hat f(x):=\EE[\hat f(x, z)]$. Setting $v := \frac{ x- \bar x}{\|x - \bar x\|}$  and $ w := \frac{ x + \bar x}{\|x + \bar x\|}$, we have
\begin{align*}
\hat f(x, z) = | \dotp{a, x - \bar x} \dotp{a, x + \bar x}| = \|x - \bar x \| \| x+ \bar x\| |\dotp{a, v} \dotp{a, w} |.
\end{align*}
Therefore, we deduce
$$
\hat f(x) := \EE\left[ \hat f(x, z) \right]  \geq \tilde \mu\|x - \bar x\| \|x + \bar x\| \geq \tilde \mu \|\bar x\| \cdot \dist(x, \{\pm \bar x\}).
$$
Now, using this bound, we find that
\begin{align*}
f(x) - f(\pm \bar x) &= (1- \pfail)(\hat f(x) - \hat f(\bar x)) + \pfail \EE_{a, \xi} \left[ |(a^Tx)^2 - (a^T \bar x)^2 - \xi | - |\xi| \right]\\
&\geq (1 - \pfail) \hat f(x) - \pfail \EE_a\left[ |(a^T x)^2 - (a^T \bar x)^2|\right]\\
&= (1-2\pfail) \hat f(x) \geq (1-2\pfail)\tilde \mu \|\bar x\| \cdot \dist(x, \{\pm \bar x\}).
\end{align*}
In particular, we deduce the equality $\cX^*=\{\pm\bar x\}$ and the sharpness estimate~\ref{assumption:sharp} with  $\mu=(1-2\pfail)\tilde \mu \|\bar x\|$.
Now we estimate the parameters of the models.

 We begin with an estimate of $\quadapprox$. 
To that end, fix $y \in \RR^d$.  Then, using the expansion
$
\dotp{a, y}^2 = \dotp{a, x}^2 + 2\dotp{a, x}\dotp{a, y -x} + \dotp{a, y - x}^2,
$
we find that for any $z \in \Omega$, we have
\begin{align}\label{eq:pl_phase_lb}
f(y, z) &= | (a^Ty)^2 - ((a^T\bar x)^2 + \bern \cdot \xi)|\notag \\
&= | \dotp{a, x}^2 + 2\dotp{a, x}\dotp{a, y -x} + \dotp{a, y - x}^2 - ((a^T\bar x)^2 + \bern \cdot \xi)|\notag \\
&\geq f^{pl}_x(y, z) - \dotp{a, y- x}^2 \notag.
\end{align}
We use this inequality to estimate $\quadapprox$ for each of the models.
Let us analyze each of the models in turn:
\begin{itemize}
\item[] {(\bf prox-linear)} We have,
$
\EE\left[f_x^{pl}(y, z)\right]\leq\EE\left[f(y, z) +  \dotp{a, y- x}^2\right] \leq f(y) + \weakphase \|y - x\|^2.
$
\item[] {(\bf subgradient)}  By inspection, we have $G(x, z) \in \partial f_x^{pl}(x,z)$. Thus, we have
$$
f_x^{s}(y, z) = f_x^{pl}(x, z) + \dotp{ G(x, z), y - x} \leq f_x^{pl}(y, z) \leq f(y, z) + \dotp{a, y- x}^2,
$$
and consequently,
$
\EE\left[ f_x^{s}(y, z)\right] \leq f(y) + \tilde\eta\|y - x\|^2.
$
\item[] {(\bf clipped Subgradient)} As before, we have
$$
\max\{f_x^{s}(y, z), 0\} = \max\{f_x^{pl}(x, z) + \dotp{ G(x, z), y - x}, 0\} \leq \max\{f_x^{pl}(y, z), 0\} \leq f(y, z) + \dotp{a, y- x}^2,
$$
and consequently, we have
$
\EE\left[ f^{cl}_x(y, z)\right] \leq f(y) + \weakphase\|y - x\|^2.
$
\end{itemize}
Therefore in all three cases, we have $\quadapprox = 2\weakphase$.

Now we analyze $L(x, z)$. Any subgradient of any of the models evaluated at a point $x \in \RR^d$ is of the form
$
v = 2s \dotp{a, x} a \text{ for some $s \in [-1, 1]$}.
$
Consequently, in all three cases, we have
$
\min_{v\in \partial f_x(x,z)} \|v\|\leq 2|\dotp{a, x}|\|a\| = L(x,z),
$
as desired.

\subsection{Proof of Theorem~\ref{thm:phase_params}}\label{appendix:thm:blind_params}

Throughout the proof, let $(x, y), (\hat x, \hat y) \in \cX$. Let $\hat f((x, y), z) = | \dotp{\ell,  x}\dotp{r,  y} - \dotp{\ell, \bar x}\dotp{r, \bar y}|$ denote the outlier free objective and notice that with $M =  \frac{x y^T - \bar x \bar y^T}{\| x y^T - \bar x \bar y^T\|}$, we have
\begin{align*}
\hat f((x, y), z) = | \ell^T( x y^T - \bar x \bar y^T) r| = \|x y^T - \bar x \bar y^T\|_F |\ell^T M r|.
\end{align*}
Now taking into account~\cite[Proposition 4.2]{charisopoulos2019composite}, we have
$$
\|xy^T - \bar x \bar y^T\| \geq  \frac{\sqrt{D}}{2\sqrt{2}(\nu + 1)}\dist((x,y), \cX^\ast).
$$
Thus, it follows that
$$
\hat f(x, y) := \EE\left[ \hat f((x, y), z) \right] = \|x y^T - \bar x \bar y^T\|_F \EE\left[ |a^T M a|\right] \geq  \frac{\tilde \mu\sqrt{D}}{2\sqrt{2}(\nu + 1)}\dist((x, y), \cX^\ast).
$$
Finally, using this bound, we find that for any $\alpha \neq 0$ that
\begin{align*}
&f(x,y) - f( \alpha \bar x, (1/\alpha) \bar y)  \\
&= (1- \pfail)(\hat f(x, y) - \hat f( \alpha \bar x, (1/\alpha) \bar y )) + \pfail \EE_{\ell, r, \xi} \left[ | \dotp{\ell,  x}\dotp{r,  y} - \dotp{\ell, \bar x}\dotp{r, \bar y}  -  \xi | - |\xi| \right]\\
&\geq (1 - \pfail) \hat f(x,y) - \pfail \EE_{\ell, r}\left[ | \dotp{\ell,  x}\dotp{r,  y} - \dotp{\ell, \bar x}\dotp{r, \bar y}|\right]\\
&= (1-2\pfail) \hat f(x,y) \geq  \frac{\tilde \mu(1-2\pfail)\sqrt{D}}{2\sqrt{2}(\nu + 1)}\dist((x,y), \cX^\ast).
\end{align*}
This proves sharpness. Now we estimate the parameters of the models.

Let us begin with $\quadapprox$. To that end, we observe that with $M = \frac{(y - \hat y)(x - \hat x)^T}{\|(y - \hat y)(x - \hat x)^T\|}$, we have
\begin{align*}
|f((\hat x, \hat y), z) - f^{pl}_{(x, y)}((\hat x, \hat y), z)| &\leq |\ell^T(\hat x \hat y^T - x y^T - x(\hat y - y) - y(\hat x - x)^T) r| \\
&= |\ell^T((y - \hat y)(x - \hat x)^T)r| \\
&= |\ell^T M r| \|y-\hat y\|\| x - \hat x\| \\
&\leq \frac{|\ell^T M r|}{2} \left(\|y-\hat y\|^2 + \| x - \hat x\|^2\right).
\end{align*}
We use this inequality to establish the weak quadratic approximation property for each of the models.
Let us analyze each of the models in turn:
\begin{itemize}
\item[] {\bf(Prox-linear)} Taking expectations, we have
$$
\EE\left[ f^{pl}_{(x,y)}((\hat x, \hat y), z) \right] \leq f(\hat x, \hat y) + \frac{\weakphase}{2} \left(\|y-\hat y\|^2 + \| x - \hat x\|^2\right).
$$
\item[] {\bf(Subgradient)} By inspection, the inclusion holds $G((x, y), z) \in \partial f_{(x, y)}^{pl}((x, y), z) $. Therefore,
$$
f_{(x, y)}^{s}((\hat x, \hat y), z) = f_{(x, y)}^{pl}((x, y), \data) + \dotp{ G((x,y), z), (\hat x, \hat y) - (x, y)} \leq f_{(x, y)}^{pl}((\hat x, \hat y),\data),
$$
and consequently
$
\EE\left[ f_{(x, y)}^{s}((\hat x, \hat y), z) \right] \leq \EE\left[ f_{(x, y)}^{pl}((\hat x, \hat y), \data) \right] \leq f(\hat x, \hat y) +   \frac{\weakphase}{2} \left(\|y-\hat y\|^2 + \| x - \hat x\|^2\right).
$
\item[] {\bf(Clipped Subgradient)} As before, we have
$$
 \max\{f_{(x, y)}^{s}((\hat x, \hat y), z), 0\} = \max\{f_{(x, y)}^{pl}((x, y), \data)  + \dotp{ G((x,y), z), (\hat x, \hat y) - (x,y)}, 0\} \leq   f_{(x, y)}^{pl}((\hat x, \hat y),z),
$$
and consequently
$
\EE\left[ f_{(x, y)}^{cl}((\hat x, \hat y), z)\right] \leq \EE\left[ f_{(x, y)}^{pl}((\hat x, \hat y), \data)\right] \leq  f(\hat x, \hat y) +   \frac{\weakphase}{2} \left(\|y-\hat y\|^2 + \| x - \hat x\|^2\right).
$
\end{itemize}
Therefore in all three cases, we have $\quadapprox = \weakphase$.

Now we analyze $L((x, y), z)$. Any subgradient of any of the models evaluated at the point $(x, y)$ is of the form
$
v = (\dotp{r, y} \ell, \dotp{\ell, x} r)s  \text{ for some $s \in [-1, 1]$}.
$
Therefore,
$
\min_{v\in \partial f_x(x,z)}\|v\|\leq \sqrt{\dotp{\ell, x}^2 \|r\|^2 + \dotp{r, y}^2\|\ell\|^2} \leq L((x, y),z),
$ as desired.

Finally, the bound on $\lipsymb$ follows since
$$
\lipsymb^2 \leq \sup_{(x, y) \in \cT_2} \EE\left[L((x,y), z)^2\right] \leq \sup_{(v, w) \in \sphere^{d_1 - 1}  \times \sphere^{d_2 - 1} } \nu^2D \EE\left[L((v,w), z)^2\right] \leq \nu^2D \tilde \lipsymb^2,
$$
as desired.

\section{Sharpness and identifiability}\label{sec:sharpness_identif}
In this section, we explain that local sharp growth of a function $f$ relative to a set $\cS$ is equivalent to $\cS$ be an ``active manifold'' for $f$ locally around its minimizer. This equivalence is in essence well-known, though we have been unable to find a formal statement. To illustrate on a simple example, consider the function $f(x,y)=x^2+|y|$ and the set $\cS=\R\times \{0\}$. Notice that $f$ satisfies two geometric properties. On one hand, $f$ grows sharply (at least linearly) as one moves away from $\cS$. On the other hand, $\cS$ is ``active'' or ``identifiable'' in the sense that the subgradients of $f$ are uniformly bounded away from zero outside of $\cS$. We will see that these two geometric properties are essentially equivalent. To formalize the notion of an ``active set'', we follow the work \cite{ident}, which expands on the earlier papers of Lewis \cite{lewis_active} and Wright \cite{MR1227547}.

Throughout, we use the standard definitions and notation of variational analysis, as set out in the monographs \cite{RW98,Mord_1,penot_book}. Namely, consider  a function $f\colon\R^d\to\overline{\R}$ and a point $x$, with $f(x)$ finite. The {\em Fr\'{e}chet subdifferential}, denoted $\partial f(x)$, consists of all vectors $v\in \R^d$ satisfying
$$
f(y)\geq f(x)+\langle v,y-x\rangle+o(\|y-x\|)\quad\textrm{ as }\quad y\to x.$$
The {\em limiting subdifferential}, denoted $\partial_L f(x)$,  consists of all vectors $v\in \R^d$ for which there exist sequences $x_i\in\R^d$ and $v_i\in \partial f(x_i)$ satisfying $(v_i,f(x_i),v_i)\to ( x,f( x), v)$. Following \cite{prox_reg}, we say that $f$ is {\em prox-regular at $\bar x$ for $\bar v\in \partial_L f(\bar x)$} if there exist real $\epsilon, \rho>0$ such that the estimate
$$f(y)\geq f(x)+\langle v,y-x\rangle-\frac{\rho}{2}\|y-x\|^2,$$
holds for any  $x,y\in\R^d$ and $v\in \partial_L f(x)$ satisfying $\max\{\|y -
\bar x\|, \|x-\bar x\|, \|v-\bar v\|, |f(x)-f(\bar x)|\}<\epsilon$. In particular,
weakly convex functions are prox-regular.

The following is the formal definition of an identifiable (or active) manifold.

\begin{definition}[Identifiable manifold]
	{\rm
		Consider a closed function $f\colon\R^d\to\R\cup\{+\infty\}$. We call a set
		$\cS$ an {\em identifiable manifold at $\bar x$ for $\bar v\in \partial f(\bar x)$}  if the following properties hold.
		\begin{enumerate}
			\item {\bf (smoothness)} The set $\cS$ is a $C^2$-smooth manifold around $\bar x$ and the restriction $f\big|_{\cS}$ is $C^2$-smooth near $\bar x$.
			\item {\bf (finite identification)} For any sequences $(x_i,f(x_i),v_i)\to(\bar x, f(\bar x),\bar v)$ with $v_i\in \partial_L f(x_i)$, the points $x_i$ must all lie in $\cS$ for all sufficiently large indices $i$.
		\end{enumerate}}
	\end{definition}

	Let us first observe that under a very mild condition on the function $f$, identifiability at a critical points implies local sharp growth.
		\begin{thm}[Identification implies sharpness]
		Consider a closed function $f\colon\R^d\to\overline{\R}$ and suppose that
		a closed set $\cS$ is an identifiable manifold at $\bar x$ for $0\in  \partial f(\bar x)$. Then there exist real $\epsilon, \mu>0$ satisfying
		$$f(x)\geq f(\proj_{\cS}(x))+\mu\cdot\dist(x,\cS) \qquad\textrm{for all }x\in B_{\epsilon}(\bar x).$$
	\end{thm}
	\begin{proof}
		First, we record an immediate consequence of~\cite[Proposition
		10.12]{ident}. Namely, there exists $\epsilon > 0$ satisfying the
		following. For all $z\in B_{\epsilon}(\bar x)\cap \cS$  and $v\in \partial_L f(z)\cap B_{\epsilon}(0)$, the inclusion holds:
		$$v+\epsilon \left(\mathbb{S}^{d-1}\cap N_{\cS}(z)\right) \subset \partial f(z).$$
		Next recall that since $\cS$ is a $C^2$-smooth manifold, every point $x$ near $\bar x$ admits a unique nearest-point projection onto $\cS$, characterized by the inclusion $x-\proj_{\cS}(x)\in N_{\cS}(\proj_S(x))$.
		For any point $x$ near $\bar x$, set $\hat x=\proj_{\cS}(x)$.
		Using~\cite[Proposition 10.11]{ident}, we deduce that $f$ is
		prox-regular at $\bar{x}$ for $\bar{v} = 0$.
		Consequently,
		there exist $\epsilon', \gamma', \rho>0$ such that
		\begin{equation}\label{eqn:prox_reg_damn}
		f(x)\geq f(\hat x)+\langle v ,x-\hat x\rangle-\frac{\rho}{2}\|x-\hat x\|^2
		\end{equation}
		for any $x\in B_{\epsilon'}(\bar x)$ and $v\in B_{\gamma}(0)\cap\partial_L f(\hat x)$. Notice that since the subdifferential of $f$ is inner-semicontinuous relative to $\cS$ at $\bar x$ for $\bar v=0$ \cite[Proposition 10.2]{ident},
		decreasing $\epsilon'$ we may ensure that $B_{\gamma}(0)\cap\partial_L f(\hat x)$ is nonempty for all $x\in B_{\epsilon'}(\bar x)$.
		We therefore deduce for every $x\in B_{\epsilon'}(\bar x)$ the estimate:
		\begin{align*}
		f(x)&\geq  f(\hat x)+\langle v+\epsilon\frac{x-\hat x}{\|x-\hat x\|},x-\hat x\rangle-\frac{\rho}{2}\|x-\hat x\|^2\\
		&\geq f(\hat x)+\epsilon\cdot\dist_{\cS}(x)-\|v\|\dist_{\cS}(x)-\frac{\rho}{2}\dist_{\cS}^2(x)\\
		&=f(\hat x)+\left(\epsilon-\gamma-\frac{\rho\epsilon'}{2}\right) \dist_{\cS}(x).
		\end{align*}
		Decreasing $\gamma$ and $\epsilon'$, if necessary, completes the proof.
	\end{proof}

	We next prove the converse, namely that a function always grows sharply away from its identifiable manifolds.
		\begin{thm}[Sharpness implies identification]
		Consider a closed function $f\colon\R^d\to\overline{\R}$ that is prox-regular at a point $\bar x$ for $0\in \partial_L f(\bar x)$.
		Suppose that there is a closed set $\cS$ containing $\bar x$ and real $\epsilon, \mu>0$ satisfying
		$$f(x)\geq \min_{z\in \proj_{\cS}(x)}f(z)+\mu\cdot\dist(x,\cS) \qquad\textrm{for all }x\in B_{\epsilon}(\bar x).$$
		Then for any sequences $(x_i,f(x_i),v_i)\to(\bar x, f(\bar x),\bar v)$ with $v_i\in \partial_L f(x_i)$, the points $x_i$ must all lie in $\cS$ for all sufficiently large indices $i$.
	\end{thm}
	\begin{proof}
		Let $\epsilon, \mu>0$ be the constants in the assumptions of the theorem.
		From the definition of prox-regularity, we deduce that there exist real $\epsilon', \rho>0$ such that the estimate
		$$f(y)\geq f(x)+\langle v,y-x\rangle-\frac{\rho}{2}\|y-x\|^2,$$
		holds for any  $x,y\in\R^d$ and $v\in \partial_L f(x)$ satisfying $\max\{\|x-\bar x\|, \|v\|, |f(x)-f(\bar x)|\}<\epsilon'$. Shrinking $\epsilon'$, we may ensure $0<\epsilon'<\min\{\epsilon,\frac{\mu}{\rho}\}$. Consider now any point $x\in B_{\epsilon'}(\bar x)\setminus \cS$ with $|f(x)-f(\bar x)|<\epsilon'$ and $\dist(0,\partial_L f(x))<\epsilon'$. We will show that the estimate $\dist(0,\partial_L f(x))\geq \frac{\mu}{2}$ holds, thereby completing the proof.
		To verify this estimate, let $v\in \partial f(x)$ have minimal norm and let $\hat x\in \proj_{\cS}(x)$ achieve $\min_{z\in \proj_{\cS}(x)} f(z)$. We then deduce
		$$\mu\cdot\dist(x,\cS)\leq f(x)-f(\hat x)\leq \langle v, x-\hat x\rangle+\frac{\rho}{2}\dist^2(x,\cS).$$
		Using the Cauchy-Schwarz inequality, we therefore conclude
		$\|v\|\geq \mu-\frac{\rho}{2}\dist(x,\cS)\geq \frac{\mu}{2}$. The result follows.
	\end{proof}

	
\end{document}